\documentclass[11pt]{article}
\usepackage{bbm}
 \usepackage{amssymb}
\usepackage{amssymb, amsthm, amsmath, amscd}
\usepackage[usenames,dvipsnames]{color}

\newtheorem{thm}{Theorem}[section]
\newtheorem{lem}[thm]{Lemma}
\newtheorem{prop}[thm]{Proposition}
\newtheorem{cor}[thm]{Corollary}
\theoremstyle{definition}
\newtheorem{NN}[thm]{}
\theoremstyle{definition}
\newtheorem{df}[thm]{Definition}
\theoremstyle{definition}
\newtheorem{rem}[thm]{Remark}

\theoremstyle{definition}
\newtheorem{exm}[thm]{Example}

\renewcommand{\phi}{\varphi}

\newcommand{\N}{\mathbb{N}}

\newcommand{\R}{\mathbb{R}}
\newcommand{\C}{\mathbb{C}}
\newcommand{\T}{\mathbb{T}}

\numberwithin{equation}{section}

\newcommand{\Aff}{\operatorname{Aff}}

\newcommand{\id}{\operatorname{id}}

\newcommand{\morp}{contractive completely positive linear map}

\newcommand{\hm}{homomorphism}
\newcommand{\dt}{\delta}
\newcommand{\ep}{\varepsilon}

\newcommand{\td}{\tilde}

\newcommand{\0}{\mbox{\large \bf 0}}
\newcommand{\la}{\langle}
\newcommand{\ra}{\rangle}
\newcommand{\andeqn}{\,\,\,{\rm and}\,\,\,}
\newcommand{\rforal}{\,\,\,{\rm for\,\,\,all}\,\,\,}
\newcommand{\CA}{$C^*$-algebra}
\newcommand{\SCA}{$C^*$-subalgebra}

\newcommand{\af}{{\alpha}}

\newcommand{\diag}{{\rm diag}}

\newcommand{\wilog}{without loss of generality}
\newcommand{\Wlog}{Without loss of generality}

\newcommand{\beq}{\begin{eqnarray}}
\newcommand{\eneq}{\end{eqnarray}}
\newcommand{\tforal}{\,\,\,\text{for\,\,\,all}\,\,\,}
\newcommand{\tand}{\,\,\,\text{and}\,\,\,}
\newcommand{\zo}{{\cal Z}_0}

\newcommand{\LAff}{{\rm LAff}}

\usepackage{amsfonts}
\usepackage{mathrsfs}
\usepackage{textcomp}
\usepackage[all]{xy}

\title{Unitary groups and augmented Cuntz semigroups 
of separable simple ${\cal Z}$-stable \CA s}
\author{Huaxin Lin }
\date{
}
\begin{document}

\maketitle

\begin{abstract}
Let $A$ be a separable simple exact ${\cal Z}$-stable \CA. 
We show that the unitary group of ${\td A}$ has the cancellation property.
If $A$ has continuous scale then  
the Cuntz semigroup of $\td A$ has 
strict comparison property 
%the strict comparison property 
and a weak cancellation
property. 
Let $C$ be a 1-dimensional non-commutative CW complex with $K_1(C)=\{0\}.$ 
%and 
%$A$ be an exact separable simple stably projectionless ${\cal Z}$-stable \CA\,  
%with continuous scale. 
Suppose that $\lambda: {\rm Cu}^\sim(C)\to {\rm Cu}^\sim(A)$ is a morphism 
in the augmented Cuntz semigroups which is strictly positive.
Then there exists a sequence of \hm s $\phi_n: C\to A$ such 
that $\lim_{n\to\infty}{\rm Cu}^\sim(\phi_n)=\lambda.$  This result leads to the proof 
that every separable amenable simple \CA\, in the UCT class has rationally generalized tracial rank at most one. 
\end{abstract}

\section{Introduction}

Recently there has been some rapid progress in the Elliott program of classification 
of separable amenable  \CA s. For example, all unital separable amenable simple 
Jiang-Su stable \CA s in the UCT class have been classified up to isomorphisms by 
the Elliott invariant (see \cite{GLN}, \cite{GLN2}, \cite{EGLN},  and \cite{TWW}, for example).
%,  also  \cite{Wlocal} and \cite{lininv}, for example).
Let $A$ be a unital ${\cal Z}$-stable \CA,  where ${\cal Z}$ is the Jiang-Su algebra.
It was shown by M. R\o rdam (\cite{Rrzstable}) that $A$ either has stable rank one, i.e., 
the invertible elements in $A$ are dense in $A,$   or $A$ is purely infinite. 
As a consequence,  in the finite case, by \cite{Rf1} and \cite{Rf2}, $A$ has the cancellation of projections  
and $U(A)/U_0(A)=K_1(A).$  There are other regular properties for ${\cal Z}$-stable \CA s (see  also
\cite{Winter-Z-stable-02}).
It is these regular properties that make the class of unital separable amenable simple ${\cal Z}$-stable 
\CA s classifiable.

One may expect  that  non-unital simple ${\cal Z}$-stable \CA s have similar 
properties.  Indeed, by M. R\o rdam (\cite{Rrzstable}),   non-unital simple  ${\cal Z}$-stable \CA s 
also have strict comparison for positive elements and nice picture of Cuntz semigroups (see \cite{ESR-Cuntz}).
It is shown by L. Robert (\cite{Rlz}) that, if $A$ is a simple stably projectionless ${\cal Z}$-stable \CA, 
then $A$ has almost stable rank one, i.e., the invertible elements in ${\td A},$ the unitization of $A,$ are dense in 
$A.$ If $A$ is a separable simple ${\cal Z}$-stable \CA\, which 
is not stably projectionless, then $A$  must have stable rank one.
So we will mainly consider stably projectionless simple  \CA s.
There is a fundamental  difference between unital simple \CA s and stably projectionless simple \CA s.
In \cite{GLII}, we show that there is a unique separable amenable simple stably projectionless 
${\cal Z}$-stable \CA\, $\zo$ 
in the UCT class with a unique tracial state such that $K_i(\zo)=K_i(\C)$ ($i=0,1$). 
Let $A$ be any finite separable amenable simple \CA. Then
% since ${\cal Z}_0$ is also 
%${\cal Z}$-stable,  
$A\otimes \zo$ is a separable 
amenable simple stably projectionless  ${\cal Z}$-stable \CA\, such that $K_i(A\otimes \zo)=K_i(A)$ ($i=0,1$) and 
$T(A\otimes \zo)=T(A).$ This means that there is a rich class of    separable simple amenable 
stably projectionless ${\cal Z}$-stable \CA s.
% than those unital separable simple amenable ${\cal Z}$-stable
There are also separable amenable stably projectionless simple \CA s which cannot be 
written as $A\otimes \zo$ for any separable simple amenable \CA\, $A$ (see \cite{GLIII}).

More recently a  classification theorem for non-unital separable simple amenable ${\cal Z}$-stable \CA s  
with stable rank one in the UCT class 
was presented in the original version of \cite{GLIII}. The motivation  of this note  is to provide
a technical result that  removes the condition of stable rank one.  
We will not, however,  prove that, in general, a separable simple stably projectionless 
${\cal Z}$-stable \CA\, has  stable rank one.  Instead, we will show that 
these \CA s have nice properties which will lead to a reduction theorem, i.e., every 
separable amenable simple stably projectionless \CA\, in the UCT class has rationally  generalized tracial rank one
without assuming that $A\otimes Q$ has stable rank one.   Therefore, as in \cite{GLIII},
the additional condition of stable rank one in the classification theorem mentioned above is removed.
%With the classification results originally presented 
%in \cite{GLIII}, this in fact implies that all separable amenable stably projectionless 
%simple ${\cal Z}$-stable \CA s in the UCT class have stable rank one. 

We begin with the question whether a non-unital separable simple ${\cal Z}$-stable \CA\, $A$ 
still has the cancellation  of projections  for $\td A$ and the property 
$U(\td A)/U_0(\td A)=K_1(A).$ 
In this note, we first show that, indeed, $U(\td A)/U_0(\td A)=K_1(A)$ (see Corollary \ref{C3}).

One notices that we study the unitary group of $M_n(\td A)$ not that of $M_n(A)$ as $M_n(A)$ has no unitaries. 
Naturally we  study the Cuntz semigroup  ${\rm Cu}(\td A)$ of $\td A,$ not ${\rm Cu}(A)$ when $A$ is stably projectionless
but $K_0(A)\not=\{0\}.$   To make the strict comparison more meaningful, we assume that $A$ has continuous scale. 
 It should be noted that $\td A$ is not ${\cal Z}$-stable and 
we do not know whether $\td A$ has stable rank one.
We do not even know whether  ${\rm Cu}(\td A)$ has cancellation of projections.
Nevertheless,  
%we obtain 
we will show that, in the case that $A$ has continuous scale,
indeed,  
${\rm Cu}(\td A)$ has the strict comparison and 
%certain  useful comparison property and 
a weak cancellation property.
% (based on \cite{eglnkk0}).
These two aforementioned
%above mentioned 
properties (one for $K_1$ and one for ${\rm Cu}(\td A)$) are proved 
without assuming $A$ has stable rank one.

L. Robert shows (\cite{Rl}) that the augmented Cuntz semigroup ${\rm Cu}^\sim$ classifies \hm s 
from 1-dimensional noncommutative CW complexes with trivial $K_1$-groups to \CA s of stable rank one. 
This result played an important role in the proof  of the fact
that every unital separable  finite simple \CA\, with finite nuclear dimension  in the UCT class has rationally generalized tracial rank at most one.
Since unital separable simple amenable  ${\cal Z}$-stable  \CA s in the UCT class  with rationally generalized tracial rank at most one are previously shown 
to be classified by the Elliott invariant, this latter result leads to   the classification 
of all unital separable simple amenable \CA s of finite nuclear dimension in the UCT class (see \cite{EGLN}). 

%More recently non-unital separable simple  amenable ${\cal Z}$-stable \CA s in the UCT class are also 
%classified up to isomorphism by the Elliott invariant (see \cite{GLIII}). However, 
%originally, for the finite case, classification was presented for the class of separable 
%simple \CA s with finite nuclear dimension and stable rank one in the UCT class.  
%Of course the condition of stable rank one is expected automatic when these \CA s are not purely infinite.
%With the condition of stable rank one, it is shown that every 
%separable amenable simple stably projectionless \CA\, in the UCT class has rational
%tracial rank at most one. 
The additional condition that \CA s have stable rank one  in the classification 
results for non-unital simple \CA s mentioned above was  used to apply 
the following existence result of L. Robert (\cite{Rl}):
Let $C$ be a 1-dimensional noncommutative CW  complex  with $K_1(C)=\{0\}$ and with a strictly positive element $e_C.$ 
If
$\lambda: {\rm Cu}^\sim(C)\to {\rm Cu}^\sim(A)$ is a morphism in ${\bf Cu}$ 
with $\lambda([e_C])\le [a]$ for some $a\in A_+,$  then there exists a \hm\, 
$\phi: C\to A$ such that ${\rm Cu}^\sim(\phi)=\lambda.$

%As priori, one does not know whether a separable simple amenable ${\cal Z}$-stable \CA\, has stable rank one.
%However, as 
As mentioned in the abstract, we show,  without assuming $A$ has stable rank one, 
%in this note,  
that there is a sequence of \hm s $\phi_k: C\to A$ 
such that $\lim_{k\to\infty}{\rm Cu}^\sim(\phi_k)=\lambda,$
if, in addition,  $A$ is an exact separable simple 
${\cal Z}$-stable \CA\,  and $\lambda([c])\not=0$ for any $c\in C_+\setminus \{0\}$ 
%then 
%there is a sequence of \hm s $\phi_k: C\to A$ 
%such that $\lim_{k\to\infty}{\rm Cu}^\sim(\phi_k)=\lambda$ 
(see Definition \ref{DappCu}). 
  It turns out that this weaker version of  existence  theorem will be  sufficient for the purpose 
of proving  that every separable simple amenable stably projectionless  \CA\, in the UCT class has 
rationally generalized tracial rank at most one.  Therefore, we are able to remove the redundant  condition of stable rank one 
in the original version of \cite{GLIII}. 
 Together with the classification theorem in \cite{GLIII}, every  finite
separable simple amenable ${\cal Z}$-stable \CA\, in the UCT class, in fact, has stable rank one. 

Let $C$ be a 1-dimensional NCCW complex. L. Robert shows 
that there are 1-dimensional NCCW complexes $C_0, C_1,...,C_n$
such that $C_0=C_0((0,1]),$ $C_n=C,$ $C_i$ is either stably isomorphic to $C_{i-1},$
or $C_i$ is the unitization of $C_{i-1},$ or $C_{i-1}$ is the unitization of $C_i,$
$i=1,2,...,n.$
Let $B$ be a  separable simple stably projectionless 
${\cal Z}$-stable \CA. Then $B$ has almost stable rank one. 
We first show that, for $C=C_0,$ a \hm\, $h$ can be produced 
so that ${\rm Cu}^\sim(h)$ will be the given $\lambda.$ We then show our approximate version 
of existence theorem holds for \CA s $C_1$ and beyond.  However, this process 
requires to change the target algebra $B$ to $M_n(\td B)$ (for any integer $n\ge 1$).
The problem is that we do not  know whether $M_n(\td B)$ has stable rank one.

%Without knowing $B$ has stable rank one, one does not know 
%${\rm Cu}^\sim(B)$ has the cancellation. 
Let $\phi, \psi: C\to M_n(\td B)$ be \hm s such that ${\rm Cu}^\sim(\phi)={\rm Cu}^\sim(\psi).$
Suppose that $e\in M_k(C)$  is a nonzero projection and $p=\phi(e)$ and $q=\psi(e).$ 
Note that
%, if $a, b\in M_n(\td B)_+$ for some integer $n,$ then 
$[p]=[q]$ in $\in {\rm Cu}^\sim(\td B)$ if and only if 
there is an integer $1\le m$ ($\le 2$) such that
$
p\oplus 1_m\sim q\oplus 1_m
$
in the Cuntz semigroup of $\td B.$  
However, 
 the classification of \hm s by ${\rm Cu}^\sim$ is not possible without $p\sim q.$
We will not attempt to prove that the functor ${\rm Cu}^\sim$ 
(introduced by Robert) classifies \hm s from 1-dimensional noncommutative 
CW complexes.  
The existence part of Theorem 1.0.1  of \cite{Rl} depends on the uniqueness part of  that. 
  Nevertheless,  we will find a way to circumvent this to obtain an approximate version of existence  theorem
 without the uniqueness theorem.
  
  The paper is organized as follows: In section 2, we list some basics regarding the notion of almost stable rank one 
  and other notations. In section 3, we show that, with slightly more  general assumption, 
  $U(\td A)/U_0(\td A)=K_1(A)$ for any separable simple ${\cal Z}$-stable \CA\, $A.$
  In section 4, we present some crucial technical statements about comparison 
  in $M_n(\td A)$ (for any integer $n\ge 1$) 
  involving unitaries. 
  We show that ${\rm Cu}(\td A)$ has the strict comparison and a weak 
  cancellation when $A$ has continuous scale.
  In section 5, we start some discussion of 
  approximation in  augmented Cuntz semigroups and perturbation of \hm s.  In section 6, we deal with 
  unitization. Finally we present the main results in section 7.

{{{\bf{Acknowledgement}}:    This research
    is partially supported by a NSF grant (DMS 1954600)}}
and  the Research Center for Operator Algebras in East China Normal University.
  %
  %
  %%%%%%%%%%%%%%%%%%%%%%%%%%%%%%%
  %
%
%
%\section{Preliminaries}
%%%%%%

\section{Basics}

\begin{df}[\cite{Rlz}]\label{Dastr1}
If $C$ is a unital \CA, let $GL(C)$ be the group of invertible elements of $C.$
A \CA\, $A$ is said to have almost stable rank one, 
if $GL(\td B)$ is dense in $B$ for every hereditary \SCA\, $B$ of $A,$ 
where $\tilde B$ is the unitization of $B,$ if $B$ is not unital.

If $A$  has almost stable rank one, by the definition, every hereditary \SCA\, of $A$ has almost stable rank one.
%and 
%for $B=M_n(A)$ for every integer $n\ge 1.$
\end{df}

For a separable simple \CA\, $A,$  if $A$ does not have stable rank one, but
has almost stable rank one, then $A$ must be projectionless, by the following 
observation which is known.
% In particular, if $A$ is a finite separable simple ${\cal Z}$-stable 
%\CA\, which does not have stable rank, then $A$ must be stably projectionless (in which case, by \cite{Rlz}, 
%$M_n(A)$ has almost stable rank one for each integer $n\ge 1$). 

%Let us begin with following known statement.
\begin{prop}\label{P1}
Let $A$ be a $\sigma$-unital \CA\,  which has almost stable rank one.
Then $A$ has stable rank one, if $A$ has a nonzero full  projection.
If $A$ is simple and $M_n(A)$ has almost stable rank one for each $n,$ then $A$ either has stable rank one,
or $A$ is stably projectionless. 

Moreover, if $A$ is simple and  has almost stable rank one, then ${\rm Ped}(A),$
the Pedersen ideal of $A,$ has no infinite elements
(see Definition 1.1 of \cite{LZ}), and $\td A$ is finite. 
%(in the sense defined in section 6.3 of \cite{Blb}). 

%If $M_n(A)$ has almost stable rank one for all $n\in \N,$ then $\td A$ is stably finite. 
%and $A$ has a non-zero 2-quasitrace (see \ref{Dqtr} below). 
\end{prop}

\begin{proof}
Fix an integer $n\ge 1.$ Suppose that $M_n(A)$ has almost stable rank one.
Let $p\in M_n(A)$ be a nonzero full projection.
By the definition, the invertible elements of $pM_n(A)p$ is dense in $pM_n(A)p.$
So $pM_n(A)p$ has stable rank one. By \cite{Br}, since $A$ is $\sigma$-unital, $pM_n(A)p$
is stably isomorphic to $A.$ Therefore $A$ has stable rank one. Note that the above works for $n=1.$
This proves the first part of the statement.

Suppose that  $A$ is simple and has almost stable rank one. If $A$ has stable rank one, $\td A$ has stable rank one.
Then $A$ and $\td A$ 
are stably finite.   In particular, ${\rm Ped}(A)$ has no infinite elements. 
Now suppose that $A$ does not have stable rank one but  has almost stable rank one.
%Suppose that 
If ${\rm Ped}(A)$ has an infinite element, by Theorem 1.2 of \cite{LZ},
$A$ has a non-trivial projection. By the first part of the proposition, $A$ has stable rank one.
A contradiction. 

  If $\td A$ is not finite,  there is $v\in \td A$ 
such that $vv^*=1$ and $v^*v\not=1.$ Then $1-v^*v\in A$ is a non-zero projection.  By what has been 
proved, this would imply that $A$ has stable rank one. 
%
%Applying what has been proved so far to $M_n(A),$  if $M_n(A)$ has almost stable rank one
%for each $n\in \N,$ then $\td A$ is stably finite.  By the remark before , one concludes 
%that $A$ has a non-zero 2-quasitrace. 
%
\end{proof}

The following  is a non-unital version of a result of R\o rdam (\cite{Rrzstable})
 which follows  from   Theorem 6.7 of \cite{Rrzstable} and   a result of L. Robert (Theorem 1.2 \cite{Rlz}).

\begin{thm}\label{astrk1z}
Let $A$ be a $\sigma$-unital simple  ${\cal Z}$-stable \CA.  Then one and only one of the following must occur:

(1) $A$ is purely infinite,

(2) $A$ has stable rank one,

(3) $A$ does not have stable rank one, but has almost stable rank one and is stably projectionless.
Moreover $A$ has a non-zero {\rm 2}-quasitrace. 
\end{thm}

\begin{proof}
Suppose that  neither  does $A$ have stable rank one,  nor $A$ is purely infinite.
% stable rank one.

Note, since $A$ is ${\cal Z}$-stable, so is $M_n(A)$ for each $n\in \N.$ 
If $M_n(A)$ contains a non-zero projection $p$ for some $n\in \N,$ then $B:=pM_n(A)p,$ as a unital hereditary \SCA,
is also ${\cal Z}$-stable (Corollary 3.1 of  \cite{TW}).  By  Theorem 4.5 of  \cite{Rrzstable},  $W(B)$ is almost unperforated.  If $B$ does not have stable rank one, then, by Corollary 3.6 of \cite{BP95}, $M_k(B)$ does not have stable rank one for any $k.$ 
Therefore, by Theorem 6.7 of \cite{Rrzstable}, none of $M_k(B)$ are finite.  Then, by Corollary 5.1 of \cite{Rrzstable}
(see also Proposition 4.9 of \cite{FL}),
  $B$ is purely infinite.    By the assumption at the very beginning,
$M_n(A)$ has no non-zero projection for all $n.$
% Let $e\in B$ be a non-zero projection and let $C=eBe=eAe.$ So
%$C$ is also ${\cal Z}$-stable. 
%It follows from Theorem 6.1 of \cite{Rrzstable} that $C$  has stable rank one,
%it $C$ is not finite. 
%Assume that $C$ is finite.  
%By \cite{Br}, $A\otimes {\cal K}\cong C\otimes {\cal K}$  has stable rank one. It follows from ? that $A$ has stable rank one. 
%So $C$ is not finite. This implies that $e$ is infinite.   It follows that $B$ is purely infinite 
 In other words, $A$ is stably projectionless.
Then, by \cite{Rlz}, $A$ has almost stable rank one. Moreover, by Corollary 5.1 of \cite{Rrzstable} (see Proposition 4.9 of \cite{FL}), $A$ has a non-zero
2-quasitrace.
\end{proof}

We do not know, at the moment,  that  case (3) of  Theorem \ref{astrk1z} can actually occur.

\begin{prop}\label{Pmatrixsr1}
Let $A$ be a \CA\, which has almost stable rank one.
Then, for any integer $n\ge 1,$ $GL(M_n(\td A))$ is dense in $M_n(A).$
Moreover, $GL((\td A\otimes {\cal K})^\sim)$ is dense in $A\otimes {\cal K}.$
\end{prop}

\begin{proof}
We prove  the first part  by induction. Suppose that $GL(M_n(\td A))$ is dense in 
$M_n(A).$   We will show that $GL(M_{n+1}(\td A))$ is dense 
in $M_{n+1}(A).$

Let $x\in M_{n+1}(A).$  Put $p:=\diag(1_{\td A},\overbrace{0,0,...,0}^{{n}}).$ 
Let $a=pxp,$ $b=(1-p)x(1-p),$ $c=px(1-p)$ and $d=(1-p)xp.$
Hence we may write  
\beq
x=\begin{pmatrix} a & c\\ d & b\end{pmatrix}.
\eneq
Let $\ep>0.$ By the inductive assumption, there is $b'\in GL(M_n(\td A))$
such that $\|b-b'\|<\ep.$  Note 
\beq
c(b')^{-1}d=px(1-p) (b')^{-1} (1-p)xp\in p(M_n(A))p\, (=A).
\eneq
Therefore (since $A$ has almost stable rank one) there is 
$z\in GL(\td A)$ such that $\|z-(a-c(b')^{-1}d)\|<\ep.$
Set $a'=z+c(b')^{-1} d\in \td A.$  Then
\beq
\|a-a'\|=\|a-c(b')^{-1}d-z\|<\ep.
\eneq
Moreover,
%\beq
$a'-c(b')^{-1}d=z.$
%\eneq
Put 
$$
y=a' +b'+c+d=\begin{pmatrix} a' & c\\ d & b'\end{pmatrix}.
$$
Then $y\in M_{n+1}(\td A)$ and $\|x-y\|<\ep.$
It follows from Lemma 3.1.5  of \cite{Lnbook} (see the proof of Lemma 3.4 of \cite{Rf1}) that 
$y$ is invertible. 

For the last part, let $a\in A\otimes {\cal K}$ and $1>\ep>0.$ 
Viewing $M_n(A)$ as a  \SCA\, of $A\otimes {\cal K},$  we may assume that $a\in M_n(A)$ 
for some large $n.$ By what has been proved, we have an invertible element
$b\in M_n(\td A)$ such that $\|b-a\|<\ep.$ Write $b=(c_{i,j})_{n\times n}$ 
with $c_{i,j}=\af_{i,j}+a_{i,j},$ where $\af_{i,j}\in \C$ and $a_{i,j}\in A.$ Let $E_n$ be 
the identity of $M_n(\td A).$ Put $x:=b+\ep\cdot (1_{(A\otimes {\cal K})^\sim}-E_n).$ 
Then $x\in GL((\td A\otimes {\cal K})^\sim)$ and $\|a-x\|<\ep.$
\end{proof}

\begin{df}\label{Dher}
Let $A$ be a \CA.  
Denote by $A^{\bf 1}$   the unit ball of $A.$
Let  $a\in A_+.$  Denote by  ${\rm Her}(a)$ the hereditary \SCA\, $\overline{aAa}.$
If $a, b\in A_+,$ we write $a\lesssim b,$
($a$ is Cuntz smaller than $b$),
if there exists a sequence of $x_n\in A$ such that $a=\lim_{n\to\infty} x_n^*x_n$
and $x_nx_n^*\in {\rm Her}(b).$  {{If  $a\lesssim b$ and $b\lesssim a$, then we say $a$ is Cuntz equivalent to $b$
and write $a\sim b.$}} The Cuntz equivalence class represented by $a$ will
be denoted by $[a].$  So we write $[a]\le [b],$ if $a\lesssim b.$ Also $[a]\ll [b]$ means 
that, if for any increasing sequence $\{x_n\}$  such that  $[b]\le \sup_n x_n,$ then $[a]\le x_n$ for some $n.$
It is well known that, for any $0<\ep<\|a\|,$ $[(a-\ep)_+]\ll [a]$ (see the middle of the proof of Lemma 2.1.1 of \cite{Rl} and 
Theorem 1 of \cite{CEI}).
Denote by  ${\rm Cu}(A)$  the Cuntz semigroup of $A$ (equivalence classes in $A\otimes {\cal K}$).
An element $x\in {\rm Cu}(A)$ is {\it compact}, if $x\ll x.$
In what follows, we will also use the augmented semigroup ${\rm Cu}^\sim(A)$
introduced in \cite{Rl} and the revised version in \cite{RS}. We refer the reader to
\cite{Rl} and \cite{RS} for details of the definition of ${\rm Cu}^\sim$ and the related terminologies. 

\end{df}

\begin{df}\label{Dqtr}
Let $A$ be a \CA. Denote by $QT(A)$ the set of 2-quasitraces of $A$ with norm 1, 
and by $T(A)$ the tracial state space of $A.$  Both could be empty  in general.

For any (non-unital) separable  \CA\, $A,$ denote by ${\rm Ped}(A)$ the Pedersen ideal of $A.$  
Suppose that $B$ is a full hereditary \SCA\, of $A$ such that $B\subset {\rm Ped}(A).$ 
%Let $QT(B)$ be the set of 2-quasitraces with norm 1. 
If $\tau\in QT(B),$
we will continue to write $\tau$ for $\tau\otimes {\rm Tr},$ where ${\rm Tr}$ is the densely defined 
trace on ${\cal K}.$ We write $QT_0(B)$ for the set of all 2-quasitraces of $B$
with the norm at most one.
Since $A$ is stably isomorphic to $B,$ $\tau\in QT_0(B)$ gives 
a densely defined 2-quasitrace of $A.$ Denote by ${\widetilde{QT}}(A)$ the set of all densely defined 2-quasitraces on $A$
%with point-wise convergence on ${\rm Ped}(A)$ as 
with the topology given in \cite{ESR-Cuntz} (see the paragraph above Theorem 4.4 of \cite{ESR-Cuntz}).
In most cases, we will consider only those \CA s with the property that every 
2-quasitrace is a trace, for example, $A$ is exact. 

If $\tau\in {\widetilde{QT}}(A),$ we will also continue to write $\tau$  on 
%$M_r(A)=A\otimes M_r$ 
$A\otimes {\cal K}$ for $\tau\otimes {\rm Tr},$ where ${\rm Tr}$ is the standard (densely defined)
trace on ${\cal K}.$  So we will view  ${\widetilde{QT}}(A)$ the set of   densely defined 
2-quasitraces on $A\otimes {\cal K}.$

%In what follows $T(A)$ will be the set of all tracial states of $A.$
\end{df}

\begin{df}\label{Daff}
Let  $S$ be a convex subset of a  convex  topological cone (which has zero)
(such as $\widetilde{QT}(A)$).
%of a topological vector space
%(we assume that cone has zero). 
%(of a convex topological cone with Choquet simplex as a base).
Let $\Aff(S)$ be the set of all 
%$C(S, \R)$ be the set of 
real valued continuous affine functions on $S$ with the property that, 
if $0\in S,$ then $f(0)=0.$
%Let $\Aff(S)$ be the set of
%\subset {\tilde T}(A)$  be a convex subset.
% Let $\Aff({\tilde T}(A))$ be
%the set of all real linear continuous functions on ${\tilde T}(A).$
Define (see \cite{Rl}) 
\beq
%\Aff(S)_+&=&\{f: C(S, \R)_+: f \,\, %{\rm linear},
%{{\rm affine}}, f(\tau)\ge 0\},\\
\Aff_+(S)&=&\{f: \Aff(S):  \, %{\rm linear},
f(\tau)>0\,\,{\rm for}\,\,\tau\not=0\}\cup \{0\},\\
%{\rm LAff}_f(S)_+&=&\{f:S\to [0,\infty): \exists \{f_n\}, f_n\nearrow f,\,\,
% f_n\in \Aff(S)_+\},\\
%{\rm LAff}_{f,+}(S)&=&\{f:S\to [0,\infty): \exists \{f_n\}, f_n\nearrow f,\,\,
 %f_n\in \Aff_+(S)\},\\
%{\rm LAff}(S)_+&=&\{f:S\to [0,\infty]: \exists \{f_n\}, f_n\nearrow f,\,\,
% f_n\in \Aff(S)_+\},\\
{\rm LAff}_+(S)&=&\{f:S\to [0,\infty]: \exists \{f_n\}, f_n\nearrow f,\,\,
 f_n\in \Aff_+(S)\}\andeqn\\
 {\rm LAff}^{\sim}_+(S) &=&\{f_1-f_2: f_1\in {\rm LAff}_+(S)\andeqn f_2\in
 {\rm Aff}_+(S)\}.
 \eneq
%\beq
%\Aff({\tilde T}(A))_+&=&\{f: C({\tilde{T}}(A), \R)_+: f \,\, {\rm linear}, f(\tau)\ge 0\},\\
%\Aff_+({\tilde{T}}(A))&=&\{f: C({\tilde{T}}(A), \R)_+: f \,\, {\rm linear}, f(\tau)>0\,\,{\rm for}\,\,\tau\not=0\}\cup \{0\},\\
%{\rm LAff}_f({\tilde{T}}(A))_+&=&\{f:{\tilde{T}}(A)\to [0,\infty): \exists \{f_n\}, f_n\nearrow f,\,\,
% f_n\in \Aff({\tilde{T}}(A))_+\},\\
%{\rm LAff}_{f,+}({\tilde{T}}(A))&=&\{f:{\tilde{T}}(A)\to [0,\infty): \exists \{f_n\}, f_n\nearrow f,\,\,
 %f_n\in \Aff_+({\tilde{T}}(A))\},\\
%{\rm LAff}({\tilde{T}}(A))_+&=&\{f:{\tilde{T}}(A)\to [0,\infty]: \exists \{f_n\}, f_n\nearrow f,\,\,
% f_n\in \Aff({\tilde{T}}(A))_+\},\\
%{\rm LAff}_+({\tilde{T}}(A))&=&\{f:{\tilde{T}}(A)\to [0,\infty]: \exists \{f_n\}, f_n\nearrow f,\,\,
% f_n\in \Aff_+({\tilde{T}}(A))\}\andeqn\\
 %{\rm LAff}^{\sim}({\tilde T}(A)) &=&\{f_1-f_2: f_1\in {\rm LAff}_+({\tilde T}(A))\andeqn f_2\in
% {\rm Aff}_+({\tilde T}(A))\}.
% \eneq
%{\red{In the definition of $\LAff_+(S),$ $f_n\nearrow f$ means that $f_n(s)\le f_{n+1}(s)$ for all 
%$n$ and $s\in S,$ and $f_n\to f$ point-wisely on $S.$ In particular, since $0,$ the zero function, 
%is in $\Aff_+(S),$ $0\in \LAff_+(S).$}}
 Note that $0\in \LAff_+(S).$ For the most part of this paper, 
 %$S={\tilde T}(A)$ {{or}} 
 $S=T(A),$ or $S={\widetilde{QT}}(A)$
in the above definition will be used.  
In particular, if $S={\widetilde{QT}}(A)$ and $f\in \LAff_+(S),$ then $f(0)=0.$

%Moreover, ${\rm LAff}_{b,+}({{\tilde{T}}}(A))$ is the subset of those bounded functions
% in ${\rm LAff}_{f,+}({\tilde{T}}(A)).$ %Moreover, ${\rm LAff}_{b+}(T(A))$ is the subset of those bounded functions
 %in ${\rm LAff}_{f+}(T(A)).$
%Recall $0\in {\tilde T}(A)$ and if $f\in \LAff(\td T(A)),$ then $f(0)=0.$

%{\red{Finally  set $T_0(A):=\{t\in \td T(A): \|\tau\|\le 1\}.$ It is a compact convex subset of $\td T(A).$}}

\end{df}

\begin{df}\label{Dfep}
For any $\ep>0,$ define $f_\ep\in C([0,\infty))_+$ by
$f_\ep(t)=0$ if $t\in [0, \ep/2],$ $f_\ep(t)=1$ if $t\in [\ep, \infty)$ and
$f_\ep(t)$ is linear in $(\ep/2, \ep).$

Let $A$ be a \CA\, and $\tau$ be in ${\widetilde{QT}(A)}.$
For each $a\in A_+$ define $d_\tau(a)=\lim_{\ep\to 0} \tau(f_\ep(a)).$
Note that $f_\ep(a)\in {\rm Ped}(A)$ for all $a\in A_+.$
Recall that $A$ is said to have the Blackadar strict comparison for positive elements, 
if  $a, \, b\in (A\otimes {\cal K})_+,$ then $a\lesssim b$ whenever $d_\tau(a)<d_\tau(b)$ for all 
non-zero
$\tau\in {\widetilde{QT}}(A).$ 
%The affine function $d_\tau(a)$ on 2-quasitraces will be denoted by 
%$\widetilde{[a]}(\tau)=d_\tau(a).$  We will also use $\hat{a}(\tau)$ for $\tau(a).$

Let $A$ be a separable  stably finite simple \CA.
%and let ${\rm Cu}_+(A)$ be the sub-semigroup 
%of those elements in ${\rm Cu}(A)$ which cannot be represented by non-zero projections.
There is an 
%canonical 
order preserving \hm\, 
$\iota: {\rm Cu}(A)\to \LAff_+(\widetilde{QT}(A))$ defined by
$\iota([a])=d_\tau(a)$ for all $\tau\in  \widetilde{QT}(A)$ and for all $a\in (A\otimes {\cal K})_+.$
Let ${\rm Cu}_+(A)$ be the sub-semigroup 
of those elements in ${\rm Cu}(A)$ which cannot be represented by non-zero projections (see Proposition 6.4
of \cite{ESR-Cuntz}) and let 
$V(A)$ be the Murray-von Neumann equivalence classes of projections 
in $A\otimes {\cal K}.$   
%Consider the case that   $A$ has the strict comparison, $\iota$ is surjective and $\iota|_{{\rm Cu}_+(A)}$ is an order isomorphism.
%One may define an addition operation on the disjoint union 
%$V(A)\sqcup \LAff_+(\widetilde{QT(A)})$ as  in page 5 of \cite{BT}.
If $A$ has the strict comparison, $\iota|_{{\rm Cu}_+(A)}$ is surjective  and is an order isomorphism,
following Corollary 6.8 of \cite{ESR-Cuntz},
%\cite{BT},  for example,  
%in this case, 
we write ${\rm Cu}(A)=(V(A)\setminus \{0\})\sqcup \LAff_+(\widetilde{QT}(A)),$ 
%if $\iota|_{{\rm Cu}_+(A)}$ is surjective  and is an order isomorphism,
%(a disjoint union),
% if 
% ${\rm Cu}(A)$ is the disjoint union of $V(A)$ and $\iota(\LAff_+(\widetilde{QT(A)})),$ 
 % $\iota$ is surjective,
 %and if 
%$\iota|_{{\rm Cu}_+(A)}$ is an order isomorphism, 
where  
the mixed addition and the order   are defined 
in  the paragraph 
above Corollary 6.8 of \cite{ESR-Cuntz}
(see also page 10 of \cite{TT}).
% as well as
% the paragraph 
%above Corollary 6.8 of \cite{ESR-Cuntz}).
%(see also Corollary I.1.4 of \cite{Alf}). 
% in page 7 of \cite{BT}).
% (see also subsection 6.1 and 
%the middle of the proof of Proposition 6.1.1 of \cite{Rl}). 
In particular,  if $x\in V(A)\setminus \{0\}$ and $y\in  \LAff_+(\widetilde{QT}(A)),$ 
then $x+y=x$ if $y=0,$ and 
$x+y=\iota(x)+y,$ if $y\not=0,$ and,
$x\le y,$ if $\iota(x)(t)< y(t)$ for all $t\not=0,$   and $y\le x,$ if $y\le \iota(x).$

%Note also, in this case,  if $x, y\in V(A)$ and $\iota(x)(t)<\iota(y)(t)$ for all $t\not=0,$   then, as above, $x\le \iota(y)\le y.$
\end{df}

\begin{df}\label{Dregular}
A  separable simple \CA\, $A$ is said to  be {\it regular},  if $A$ is purely infinite, or
if $A$ has almost stable rank one and ${\rm Cu}(A)=(V(A)\setminus \{0\})\sqcup \LAff_+({\widetilde{QT}}(A))$ 
(see \ref{Dfep} above).  By \cite{Br},  for any non-zero hereditary \SCA\, $B$ of $A,$
$B\otimes {\cal K}\cong A\otimes {\cal K}.$ Therefore ${\rm Cu}(B)={\rm Cu}(A).$ 
 Hence, if $A$ is a separable regular simple \CA,  then every non-zero
hereditary \SCA\, of $A$ is regular (see the last paragraph of \ref{Dqtr}). 
%Throughout this paper,  
Except 
%the statements of  
\ref{C2} and \ref{C3},  we only consider the case that $A$ is finite.
By \cite{Rlz} (also Theorem \ref{astrk1z} above)  and Theorem 6.6 of \cite{ESR-Cuntz},  if $A$ is a  separable simple ${\cal Z}$-stable \CA, then $A$ is regular.  Recall that a separable simple \CA\, is said to be pure  (introduced by Winter in \cite{Winter-Z-stable-02}
with  non-unital version in subsection 6.3 of \cite{RS}) if ${\rm Cu}(A)$ is almost unperforated and 
 almost divisible. 
 While it is not used in this paper, we would like to mention that 
 a finite regular simple \CA\, is pure, and,
a separable simple \CA\,  which has almost stable rank one is regular if and only if it is pure
as shown in subsection 6.3  of \cite{RS} (see also Theorem 6.2 of \cite{TT} and Corollary 5.8 of \cite{ESR-Cuntz}, and 
I.1.4 of \cite{Alf}). 
%We will not use these two statements in this paper. 
We use the term ``regular" only for the convenience here. 
\end{df}

We would like to state the following version of a result of R\o rdam. Note 
that we do not assume that $M_n(A)$ has almost stable rank one.

\begin{lem}\label{LRordam1}
Let $A$ be a \CA\, which has almost stable rank one,
$a$ and $b\in M_n(A)_+$  for some integer $n\ge 1$
(or in $(A\otimes {\cal K})_+$).  Suppose 
that $a\lesssim b,$ then, for any $\ep>0,$ there is a unitary $U\in M_n(\td A)$ 
(or $U\in (\td A\otimes {\cal K}\widetilde)$)
such that 
\beq
U^*f_\ep(a)U\in {\rm Her}(b).
\eneq
\end{lem}

\begin{proof}
By Proposition \ref{Pmatrixsr1}, $GL(M_n(\td A))$ is dense in $M_n(A)$
(or $GL((\td A\otimes {\cal K})\widetilde)$ is dense in $A\otimes {\cal K}$).
Then the proof of  (iv) $\Rightarrow$ (v) in Proposition 2.4 of \cite{Rr2} (applying Theorem 5 of \cite{Pedjot87})
works here.
\end{proof}

The following is taken from the proof of 1.5 of \cite{LinHilbert}.
But it is also known (see \cite{Rlz}). 
\begin{lem}\label{LuniH}
Let $A$ be a \CA\,
% such that every hereditary \CA\, $B$ of $A$ 
which has almost stable rank one.
Suppose that $a\in (A\otimes {\cal K})_+$ (or $a\in  A_+$),  $b\in A_+,$ and $a\lesssim b$ in ${\rm Cu}(A).$ 
Suppose that $1/4>\ep>0$ and $f_{\ep/4}(a)\in {\rm Her}(b).$
Then, for any $0<\eta<\ep/4,$ there is a unitary $u\in  (\td A\otimes {\cal K}\widetilde)$ (or $u\in \td A$) such 
that $uf_{\eta}(a)u^*\in {\rm Her}(b)$ and $uf_\ep(a)=f_\ep(a).$ 
Moreover, 
 there is a partial isometry $v\in  (A\otimes {\cal K})^{**}$ (or $v\in A^{**}$) such that $vc, cv^*\in A\otimes {\cal K}$
 (or in $A$) for all $c\in {\rm Her}(a),$
$vav^*\in {\rm Her}(b)$ and $vf_{\ep}(a)=f_{\ep}(a).$  

Furthermore,  without assuming $f_{\ep/4}(a)\in {\rm Her}(b),$ 
%in general, 
%$a\lesssim b$ implies 
%that 
there is also a partial isometry $v\in  (A\otimes {\cal K})^{**}$ (or $v\in A^{**}$) such that
$vc, cv^*\in A\otimes {\cal K}$ (or in $A$),  $v^*vc=c=cv^*v$ and $vcv^*\in {\rm Her}(b)$ for all $c\in {\rm Her}(a).$ 
\end{lem}

\begin{proof}
There is a unitary $w_1\in (\td A\otimes {\cal K}\widetilde)$ (or $w_1\in {\td A}$) such that $b_1:=w_1f_{\eta/8}(a)w_1^*\in {\rm Her}(b).$
%(see the proof of (iv)$\Rightarrow$(v) of Proposition 2.4 of \cite{Rr2}, this also follows from 
By Lemma \ref{LRordam1}.  Denote $a_1:=w_1f_{\ep/4}(a)w_1^*\in {\rm Her}(b).$
Note that $a_1b_1=a_1.$   Therefore
\beq
\|(b_1-1)w_1f_{\ep/4}(a)\|=\|(b_1-1)w_1f_{\ep/4}(a)w_1^*\|=0.
\eneq
In other words,
$
b_1w_1f_{\ep/4}(a)^{1/2}=w_1f_{\ep/4}(a)^{1/2}.
$
It follows that $y_1:=w_1f_{\ep/4}(a)^{1/2}\in {\rm Her}(b).$   Moreover,
\beq
y_1^*y_1=f_{\ep/4}(a)\andeqn y_1y_1^*=w_1f_{\ep/4}(a)w_1^*.
\eneq
In what follows, for any $d\in A_+^{\bf 1}$ and $1>\dt>0,$ $e_\dt(d)$ denotes 
the open spectral projection of $d$ associated with the interval $(\dt, 1].$
Since ${\rm Her}(b)$ has almost stable rank one, by   Theorem 5 of \cite{Pedjot87},
 there is a unitary $z_1\in {\widetilde{{\rm Her}(b)}}$  such that
\beq\label{1010-8}
z_1e_{1/4}(|y_1|)=w_1e_{1/4}(|y_1|)=w_1(f_{\ep/4}(a)^{1/2}).
\eneq
Note that 
\beq\label{1010-9}
e_{1/4}(|y_1|)=e_{1/4}(f_{\ep/4}(a)^{1/2})=e_{\dt_1}(a)
\eneq
for some $\dt_1\in (\ep/8, \ep/4).$   By \eqref{1010-8} and \eqref{1010-9},
\beq\label{1010-10}
z_1^*w_1e_{\dt_1}(a)=z_1^*(z_1e_{1/4}(|y_1|))=e_{1/4}(|y_1|)=e_{\dt_1}(a).
\eneq
Write $z_1=\af\cdot 1_{\td {\rm Her}(b)}+b'$ for some $b'\in {\rm Her}(b).$ 
Replacing $z_1$ by $\af\cdot 1+b',$ we may view $z_1$ as a unitary 
in $(\td A\otimes {\cal K}\widetilde)$ (or in $\td A$).
%Viewing $z_1^*$ as a unitary in $\td A,$ put 
Put
$u_1:=z_1^*w_1\in (\td A\otimes {\cal K}\widetilde)$ (or $u_1\in \td A$).
 Then, for any $x\in \overline{f_{2\dt_1}(a)(A\otimes {\cal K})},$  by \eqref{1010-10},
%\beq
$u_1x=u_1e_{\dt_1}(|y_1|)x=z_1^*w_1e_{\dt_1}(a)x=e_{\dt_1}(a)x=x.$
%\eneq
In particular, $u_1f_{\ep}(a)=f_{\ep}(a).$  We also have,  since $z_1\in {\rm Her}(b)^\sim,$
\beq
u_1f_{\eta}(a)u_1^*=z_1^*(w_1f_{\eta}(a)w_1^*)z_1\le z_1^*b z_1\in {\rm Her}(b).
\eneq
This proves the first part of the lemma. 

To see the second  part of the lemma,  let $\eta_n=\ep/4^{n+1}.$
By  virtue of the first part of the lemma, we obtain a sequence 
of unitaries $\{u_n\}\subset  (\td A\otimes {\cal K}\widetilde)$ (or in $\td A$) such that
\beq
u_nb_{n-1}u_n^*\in {\rm Her}(b), u_nx=x\rforal x\in {\rm Her}(b_{n-1}),
\eneq
where $b_0=f_\ep(a),$ $b_n=u_nf_{\eta_n}(b_{n-1})u_n^*$ for $n=1,2,....$
Note 
\beq
\|u_{n+1}(u_n \cdots u_1f_{\eta_n}(a)-(u_n \cdots u_1f_{\eta_n}(a)))\|
=\|(u_{n+1}-1)(u_n \cdots u_1f_{\eta_n}(b)\|\\
=\|(u_{n+1}-1)(u_n \cdots u_1)f_{\eta_n}(b)(u_1^*\cdots  u_n^*)\|=\|(u_{n+1}-1)b_n\|=0.
\eneq
In other words,
%\beq
$u_{n+1}u_n\cdots u_1f_{\eta_n}(a)=u_n\cdots u_1f_{\eta_n}(a).$
%\eneq
Moreover, $u_{n+1}u_n\cdots u_1f_{\ep}(a)=f_{\ep}(a)$ for all $n.$
It follows   that $\lim_{n\to\infty} u_{n+1}u_n\cdots u_1 x$ converges in norm for all $x\in {\rm Her}(a)$
and $\lim_{n\to\infty} u_{n+1}u_n\cdots u_1xu_1^*\cdots u_n^*u_{n+1}^*$ converges in 
norm to an element in ${\rm Her}(b).$
Choose  a strictly positive $x$ of ${\rm Her}(a)_+$ with  $\|x\|=1$ and $xf_\ep(a)=f_\ep(a).$
Let $z=\lim_{n\to\infty} u_{n+1}u_n\cdots u_1x\in A.$ 
Then $zz^*=\lim_{n\to\infty} u_{n+1}u_n\cdots u_1x^2u_1^*\cdots u_n^*u_{n+1}^*\in {\rm Her}(b).$
Let $z=vx^{1/2}$ be the polar decomposition in $(A\otimes {\cal K})^{**}$ (or in 
$A^{**}$). Then $v$ is a partial isometry
and, since $x$ is a strictly positive element of ${\rm Her}(a),$ 
%
%Thus, we obtain a partial isometry $U\in (A\otimes {\cal K})^{**}$ (or $U\in A^{**}$) 
%such that 
$vc, cv^*\in A,$   $v^*vc=c=cv^*v,$ $vcv^*\in {\rm Her}(b)$  for all $c\in {\rm Her}(a),$ and
 \beq\nonumber
 vf_\ep(a)=vx^{1/2}f_\ep(a)=\lim_{n\to\infty} u_{n+1}\cdots u_1 x^{1/2}f_\ep(a)
 = \lim_{n\to\infty} u_{n+1}\cdots u_1f_\ep(a)=f_\ep(a).
 \eneq
One also notices that the third part of the lemma holds from the proof above as we may replace $a$ by $u_1au_1^*$
with $u_1f_{\ep/4}(a)u_1^*=f_{\ep/4}(u_1au_1^*)\in {\rm Her}(b).$

\end{proof}

\begin{cor}\label{Pxxxx}
Let $A$ be a \CA\, which has almost stable rank one, and $a\in (A\otimes {\cal K})_+$
(or $a\in A_+$) and $b\in A_+.$ Then 
$a\lesssim b$ if and only if there is $x\in A\otimes {\cal K}$ (or $x\in A$) such that 
$x^*x=a$ and $xx^*\in {\rm Her}(b).$

Moreover, if $a_1, a_2,...,a_n$ are mutually orthogonal elements in $A_+$ 
such that $a_i\sim a_1$ in ${\rm Cu}(A)$ for  $i=1,2,...,n,$ and 
$a:=\sum_{i=1}^n a_i\lesssim b,$  then 
there is a hereditary \SCA\, $A_1\subset {\rm Her}(b)$
such that  there is an isomorphism $\phi: M_n(A_2)\to A_1$ where $A_2={\rm Her}(d)$ for some 
$d\in {\rm Her}(b)$ such that  $\phi^{-1}(d)=d$ and there is $z\in A$ such that $z^*z=a_1$ and $zz^*=d.$
%Moreover, $[\phi(a)]=[a]$  in ${\rm Cu}
\end{cor}

\begin{proof}
The first part follows from Lemma \ref{LuniH}.
In fact,  in the second part of Lemma \ref{LuniH}, we choose $x=va^{1/2}.$ 
Then $x^*x=a^{1/2}v^*va^{1/2}=a$ and $xx^*=vav^*\in {\rm Her}(b).$

By Lemma \ref{LuniH}, there is $v\in A^{**}$ such that ${\bar a}:=vav^*\in {\rm Her}(b)$ 
and $vc,cv\in {\rm Her}(b)$ and $v^*vc=c$ for all $c\in {\rm  Her}(a).$
Let $y_0=va_1v^*.$   Then, by the first part of this lemma, there is $z\in A$ such 
that $z^*z=a_1$ and $b_1:=zz^*\in {\rm Her}(y_1).$ 
Note that $b_1\sim va_iv^*$ in ${\rm Her}({\bar a})\subset {\rm Her}(b),$  $i=1,2,...,n.$ 
Thus we have $x_i\in {\rm Her}({\bar a})$ such that $x_i^*x_i=b_1$ and $x_ix_i^*\in {\rm Her}(va_iv^*),$
$i=1, 2,..., n.$  Note that $x_1x_1^*=x_1^*x_1=b_1.$  Put $A_1={\rm Her}(\sum_{i=1}^n x_ix_i^*).$
One then checks that $A_1=M_n(A_2),$ where $A_2={\rm Her}(b_1).$   The corollary follows.

\end{proof}

We would like to end this section with the following folklore.

\begin{lem}\label{Lfolk}
Let $A$ be a \CA\, and $0\le a\le b\le 1$ be elements in $A.$
Then, for any $0<\ep<\ep'<\|a\|,$   there exists $z\in A$ such that
\beq
(a-\ep)_+\lesssim (b-\ep)_+,\,\,\, (a-\ep')_+\le z^*z \tand zz^*\in {\rm Her}((b-\ep)_+).
\eneq
\end{lem}

\begin{proof}
Choose $0<\ep<\ep'<\ep''<\|a\|$ and define $g\in C_0((0,1])$ such 
that $g(t)=1$ for $t\in [\ep'', 1]$ and $(t-\ep')_+\le g(t)\le 1$ for $t\in (\ep', \ep''),$  $g(t)=0$ if $t\in (0, \ep').$
%$ and  
%$(\ep', \ep'')$ and $g(t)\ge (t-\ep')_+.$ 
Then  $(a-\ep')_+\le g(a)$ and 
\beq
(\ep') g(a) \le g(a)^{1/2} a g(a)^{1/2}\le g(a)^{1/2} b g(a)^{1/2}.
\eneq
It follows that
\beq
g(a)^{1/2}((b-\ep)_+) g(a)^{1/2}\ge g(a)^{1/2}(b-\ep)g(a)^{1/2}\\
=g(a)^{1/2}bg(a)^{1/2} -\ep g(a)\ge (\ep'-\ep) g(a).
%\ge (\ep/4)(a-\ep)_+.
\eneq
Thus
\beq
(a-\ep')_+\le g(a)\le (1/(\ep'-\ep))g(a)^{1/2}(b-\ep)_+ g(a)^{1/2}\lesssim (b-\ep)_+.
\eneq
Since the above holds for any $0<\ep<\ep',$ $(a-\ep)_+\lesssim (b-\ep)_+.$
Let\\ $z= (1/(\ep'-\ep))^{1/2}(b-\ep)_+^{1/2}g(a)^{1/2}.$ 
Then 
\beq
g(a)\le z^*z\andeqn zz^*= (1/(\ep'-\ep))(b-\ep)_+^{1/2} g(a)(b-\ep)_+^{1/2}\in {\rm Her}((b-\ep)_+).
\eneq
\end{proof}

\section{Unitary groups}

The main purpose of this section  is to present a $K_1$-cancellation result for separable 
regular simple \CA s.

\begin{df}\label{Dpi}
Let $A$ be a \CA. Denote by $\td A$  the \CA\, generated by  $A$ and $\C\cdot 1_{\td A},$ where 
$1_{\td A}$ is not in $A.$  Denote by $\pi_\C^A: \td A\to \C\cdot 1_{\td A}=\C$ the quotient map.
We also write $\pi_\C^A$ for the extension from $M_n(\td A)$ to $M_n.$
\end{df}

\begin{df}\label{Du0}
Let $A$ be a unital \CA. Denote by $U(A)$ the unitary group of $A$  and by $U_0(A)$
the path connected component of $U(A)$ containing $1_A.$
\end{df}

\begin{prop}\label{Lmatrix}
Let $A$ be a \CA\, and $u\in U_0(M_n(\td A))$ be a unitary 
with the form $u=\af\cdot 1_{M_n(\td A)}+a$ for some $\af\in \T$ and $a\in M_n(A).$
Then  $u\in U_0(M_n(A)^\sim).$
% there is a continuous path of unitaries $u(t)\in M_n(A)^\sim$
%such that $u(0)=u$ and $u(1)=1_n.$ 
In particular, if $\af=1,$ then 
$u=\exp(ib_1)\exp(ib_2)\cdots \exp(ib_m)$
for some $b_1, b_2,...,b_m\in M_n(A)_{s.a.}.$
\end{prop}

\begin{proof}
Replacing $u$ by $u{\bar \af},$ we may assume that $\pi_\C^{A}(u)=1_n:=1_{M_n}.$
Let $u=\exp(i h_1)\exp(ih_2)\cdots \exp(i h_k),$ where $h_j\in M_n(\td A)_{s.a.}.$
For each $h_j,$ there is a scalar self-adjoint matrix $a_j\in M_n(\C\cdot 1_{\td A})$
such that $\pi_\C^{A}(h_j)=\pi_\C^{A}(a_j).$ 
Note that, since $\pi_\C^{A}(u)=1_n,$ 
$$
\exp(ia_1)\exp(ia_2)\cdots \exp(ia_k)=1_n.
$$
Define, for $t\in [0,1],$
\beq\nonumber
u(t)=\exp(i th_1)\exp(i th_2)\cdots \exp(ith_k)\exp(-ita_k)\exp(-ita_{k-1})\cdots \exp(-ita_1).
\eneq
Then $u(1)=u(\exp(ia_1)\exp(ia_2)\cdots \exp(ia_k))^*=u$ and $u(0)=1_n.$
However,
\beq\nonumber
\pi_\C^{A}(u((t)))=\exp(i ta_1)\exp(ita_2)\cdots \exp(ita_k)\exp(-ita_k)\exp(-it a_{k-1})\cdots \exp(-ita_1)=1_n.
\eneq
Therefore $u(t)\in M_n(A)^\sim$ for all $t\in [0,1].$

Suppose that $\af=1.$ Since now $u\in U_0(M_n(A)^\sim),$ $u=\exp(ih_1)\exp(i h_2)\cdots \exp(i h_m)$
for some $h_1,h_2,...,h_m\in M_n(A)^\sim.$ 
Let $\pi_\C^A(h_j)=\lambda_j\cdot 1_{M_n},$ where $\lambda_i\in \T,$ $j=1,2,...,m.$
Then $\sum_{j=1}^m \lambda_j=2k\pi$ for some integer $k.$ 
Choose $b_j=h_j-\lambda_j\,(=h_j-\lambda_j 1_{M_n}),$ $j=1,2,...,m.$ Then $b_j\in M_n(A).$ Note 
$\lambda_j\cdot 1_{M_n}$
is in the center of $M_n(A)^\sim.$ 
Then
\beq\nonumber
\exp(ib_1)\exp(ib_2)\cdots \exp(i b_m)=\exp(i h_1)\exp(ih_2)\cdots \exp(i h_m)\exp(i\sum_{j=1}^m -\lambda_j)
=u.
\eneq
\end{proof}

Note, in the following statement,  that  the unital  \CA\,  $\td A$ is not divisible in any sense.

\begin{lem}\label{Lshrink}
Let $A$ be a finite regular simple \CA\, which has no nonzero projections,  $u\in U(\td A),$  and $a_1,a_2,..., a_m\in A_{s.a.}.$ Then, for any $a\in A_+\setminus \{0\},$ 
any $\ep>0,$ there is an integer $n_0\ge 2$ such that, for any integer $n\ge n_0,$ there is a hereditary \SCA\, $B\subset A,$ and 
a unitary $v\in \C\cdot 1_{\td A}+B,$ $b_1, b_2,...,b_m\in B$ such that
%$B=U(M_n({\rm Her}(c))U^*$
$B=U^*(M_n({\rm Her}((c-\eta)_+)){{)}}U$
for some unitary $U\in M_n(\td A)$ and 
for some 
$0<\eta<\|c\|,$ where 
$c\in {\rm Her}(a)_+$
such that 
\beq
\|v-u\|<\ep\tand \|a_j-b_j\|<\ep/2(m+1),\,\,1\le j\le m.
\eneq
Moreover, we may assume that $4[c]\le [a].$

(Note that here we identify $A$ with the first corner of $M_n(A).$)
\end{lem}

\begin{proof}
Fix a strictly positive element of $e_A$ of $A$ with $\|e_A\|=1.$
Write $u=\af\cdot 1_{\td A}+x$ for some $x\in A$ and $\af\in \T.$
Let $1/2>\ep>0.$ Choose $1/2>\dt>0$ such 
that 
\beq
\|f_\dt(e_A)xf_\dt(e_A)-x\|<\ep/4\andeqn \|f_\dt(e_A)a_jf_\dt(e_A)-a_j\|<\ep/2(m+1),\,\,j=1,2,...,m.
\eneq
Choose $b_j:=f_\dt(e_A)a_jf_\dt(e_A),$ $1\le j\le m.$ 
Let $A_1=\C\cdot 1_{\td A}+{\rm Her}(f_\dt(e_A)).$ 
It is standard to find  a unitary $v\in A_1$ such that
\beq\label{LLinjU-1}
\|v-u\|<\ep.
\eneq
Let $D={\rm Her}(f_{\dt/2}(e_A)).$ Note that $D\subset {\rm Ped}(A).$ 
%We may also  assume 
%\beq\label{LLinjU-2}
%0<d_\tau(f_{\dt/2}(e_A))<1-\ep/2\rforal \tau\in QT(A).
%\eneq 
Choose $0<\dt_0<\dt/2$ such that $f_{\dt_0}(a)\not=0.$
Since both $f_{\dt_0}(a)$ and $f_{\dt/16}(e_A)$ are in ${\rm Ped}(A),$ 
there is an integer $k>2$
such that
\beq
(k-1)[a]\ge (k-1)[f_{\dt_0}(a)]\ge [f_{\dt/16}(e_A)].
\eneq

%Fix any $0<\sigma<1/2.$ 
%Suppose
%that (extend $QT(D)$ to $(D\otimes {\cal K})_+$)
%$$
%\inf\{\tau(a):\tau\in QT(D)\}\ge 1/(n_0-1)
%$$
%for some large integer $n_0>4.$ 
Choose $n_0=4k.$
Let $n\ge n_0.$   Let $c_0\in A\otimes {\cal K}$ with $0\le c_0\le 1$ such that $d_\tau(c_0)=(1/n)d_\tau(f_{\dt/8}(e_A))$   for all $\tau\in {\widetilde{QT}}(A).$ 
Thus
\beq
4d_\tau(c_0)<d_\tau(a)\rforal \tau\in {\widetilde{QT}}(A)\setminus \{0\}.
\eneq
It follows that $4[c_0]\le [a]$ in  ${\rm Cu}(A).$
Since $A$ has almost stale rank one, by  the first  part of Lemma \ref{LuniH}, there is $c\in {\rm Her}(a)_+$
such that $c\sim c_0$ and $d_\tau(c)=(1/n)\tau(f_{\dt/8}(e_A))$   for all $\tau\in {\widetilde{QT}}(A).$ 
%Choose $c\in {\rm Her}(a)_+$ with $\|c\|=1$
%such that $d_\tau(c)=(1/n)\tau(f_{\dt/2}(e_A))$   for all $\tau\in QT_0(D).$ 
Since 
%$A$ has strict comparison, 
${\rm Cu}(A)=(V(A)\setminus\{0\})\sqcup \LAff_+(\widetilde{QT}(A))$ and $A$ has no non-zero 
projection,
\beq\label{LLinjU-10}
4[c]\le  [a]\andeqn [f_{\dt/8}(e_A)]=n[c].
\eneq

%    Choose $0<\eta<1/2$ such 
%that
%\beq\label{LLinjU-10}
%nd_\tau((c-\eta)_+)>1\ge d_\tau(f_{\dt/2}(e_A))\rforal \tau\in QT(D). 
%\eneq
We now view $A$ as a \SCA\, of $M_n(A)$ (as the first corner of $M_n(A)$).
Let 
$$
c_1:=\diag(\overbrace{c, c,...,c}^{n}).
$$
%
%$$
%c_1:=\diag(\overbrace{(c-\eta)_+, (c-\eta)_+,...,(c-\eta)_+}^{n}).
%$$ 
Then, by \eqref{LLinjU-10}, $f_{\dt/2}(e_A)\ll f_{\dt/8}(e_A)\lesssim c_1.$  Therefore there is $0<\eta_0<1$
such that
\beq\label{LLinjU-18}
f_{\dt/2}(e_A)\lesssim f_{\eta_0}(c_1).
\eneq
%Since $A$ has almost stable rank one, by Lemma \ref{LRordam1}, there 
%is a unitary $U_1\in M_n(\td A)$ such 
%that 
%\beq\label{LLinjU-19}
%c_1:=U_1^*f_{\eta_0/8}(c)U_1\in {\rm Her}(a).
%\eneq
%%Then $4[c]\le [a].$
Choose $0<\eta<\eta_0/2.$ 
%such 
%that
%\beq
%(f_{\eta_0/8}(c_1)-\eta)_+\ge f_{\eta_0/2}(c_1).
%\eneq
%
%Note $(c-\eta)_+=
%U_1^*(f_{\eta_0/8}(c_1)-\eta)_+U_1.$
Since $A$ has almost stable rank one, from the last part of \eqref{LLinjU-10}, by  Lemma \ref{LRordam1},
%by Lemma 3.2 of \cite{eglnp}, 
there is 
a unitary $U_1\in M_n(\td A)$ such that
\beq
c_2:=U_1(c_1-\eta)_+U_1^*\in A.
\eneq
By \eqref{LLinjU-18}, since $A$ has almost stable rank one, applying Lemma \ref{LRordam1} again,
there is a unitary $U_2\in \td A$ such that 
\beq
U_2^*f_\dt(e_A)U_2\in {\rm Her}(c_2).
\eneq
Put $c_3=U_2c_2U_2^*.$   Put $U=\diag(\overbrace{U_2, 1_{\td U},...,1_{\td U}}^{n-1})U_1.$
Then $f_\dt(e_A)\in {\rm Her}(c_3).$ Moreover, 
$B:={\rm Her}(c_3)=U^*M_n({\rm Her}((c-\eta)_+))U^*.$  Then $v\in \C\cdot 1_{\td A}+{\rm Her}(c_3).$

\end{proof}

\begin{lem}\label{C1}
Let $A$ be a  finite separable regular simple \CA\,  and let $u\in \td A$ be a unitary.
 If $\diag(u,1)\in U_0(M_2(\td A)),$ then 
$u\in U_0(\td A).$

%(2) For any $a\in A_+\setminus \{0\},$ there is a unitary $v\in \C\cdot 1_{\td A}+{\rm Her}(a)$ such
%that $uv^*\in U_0(\td A).$ 

\end{lem}

\begin{proof}
Note that, if $A$ has a nonzero projection, then, by Proposition \ref{P1}, 
$A$ has stable rank one. Then the lemma follows from  Theorem 2.10 of \cite{Rf2}. 
So we now assume that $A$ has no nonzero projection.

We may assume that $\pi_\C^{A}(\diag(u,1))=1_2.$   By the second part of Proposition \ref{Lmatrix},
we may write $u=\exp(ib_1)\exp(i b_2)\cdots \exp(i b_m)$
for some $b_j\in M_2(A)_{s.a.},$ $1\le j\le m.$
Let $1/2>\ep>0.$
By virtue of Lemma \ref{Lshrink}, without loss of generality, we may assume that 
$u\in 1_{\td A}+B$ and there are $a_1, a_2,...,a_m\in M_2(B)_{s.a.}$
such that 
\beq
\|\diag(u,1)-\exp(ia_1)\exp(ia_2)\cdots \exp(ia_m)\|<\ep,
\eneq
where $B=U^*M_n({\rm Her}(c))U\subset A$ 
for some 
$c\in A_+,$ 
%$j=1,2,..,m,$  and 
$n\ge 4,$ 
and where $U\in M_n(\td A)$ (recall that we identify $A$ with the first corner of $M_n(A)$).
 Put $C=U^*{\rm Her}(c)U.$

Write $u=1_{\td A}+b$ for some $b\in B.$ Let $u_1:=1_{\td B}+b\in \td B$ and 
\beq
v_1:=(1_{\td B}+\sum_{n=1}^\infty {ia_1^n\over{n!}})\cdot (1_{\td B}+\sum_{n=1}^\infty {ia_2^n\over{n!}})\cdot \cdots 
\cdot (1_{\td B}+\sum_{n=1}^\infty{ia_m^n\over{n!}}).
\eneq
Hence
\beq
\|\diag(u_1, 1_{\td B})-v_1\|<\ep.
\eneq
Thus $\diag(u_1,1_{\td B})\in U_0(M_2(B)^\sim).$
Recall that $A$ has almost stable rank one. 
Thus the set of  invertible elements of $\td C$ is dense in $C=U^*{\rm Her}(c)U,$  $C$ has stable rank at most 2 (see the proof of 
 Theorem 6.13 of  \cite{RS}),
by Theorem 2.10 of \cite{Rf2},
%Theorem 10.12 of \cite{Rf1}, 
the map from $U(M_n(\td C))/U_0(M_n(\td C))$
to $U(M_{2n}(\td C))/U_0(M_{2n}(\td C))$ is injective. It follows that 
$u_1\in U_0(M_n(\td C)).$ By Lemma \ref{Lmatrix}, $u_1\in U_0(M_n(C)^\sim )=U_0(\td B).$
It follows that $u\in U_0(\td A).$

\end{proof}

\begin{thm}\label{TK1}
Let $A$ be a separable finite regular simple \CA\, and let $u\in U(\td A).$

(1) For any $a\in A_+\setminus \{0\},$ there is a unitary $v\in \C\cdot 1_{\td A}+{\rm Her}(a)$
such that  $uv^*\in U_0(\td A).$

(2) If $u=\af\cdot 1_{\td A}+x$ for some $\af\in \T$ and $x\in D$ for some hereditary \SCA\, $D$ of $A$
and $u\in U_0(\td A),$ then $v=\af\cdot 1_{\td D}+x\in U_0(\td D).$
\end{thm}

\begin{proof}
 If $A$ has stable rank one, the theorem is well known and follows from 
 the fact  (\cite{Br}) that every nonzero (full) hereditary \SCA\, $D$  of $A$ is stably isomorphic 
 %(see also Corollary 2.10 of \cite{Br}) 
 to $A$ and the  inclusion $\iota: D\to A$ induces an isomorphism on $K_1(D),$ and then apply
 Theorem 2.10 of \cite{Rf2}.
 
 We will prove the case that $A$ is not assumed to have stable rank one. Therefore we assume 
$A$ has no nonzero projection (see Proposition \ref{P1}).
 
For (1),  by Lemma \ref{Lshrink}, \wilog, we may assume 
$u=1_{\td A}+b$ for some $b\in B,$ where $B=U^*M_n({\rm Her}((c-\eta)_+))U\subset A$  
for some $0<\eta<\|c\|,$ and $c\in {\rm Her}(a)_+,$
$n>8$ and $4[c]\le [a],$   and where $U\in U(M_n(\td A)).$

Put $C=U^*{\rm Her}((c-\eta)_+)U$ and $u_1:=1_{\td B}+b.$
 Since $GL(\td C)$ is dense in 
$C,$ by (the proof of)   Theorem 6.13 of  \cite{RS}, $C$ has stable rank at most 2.
It follows from Proposition 5.3 of \cite{Rf2} that there exists a unitary $v_0\in M_2(\td C)$
such that $u_1v_1^*\in U_0(M_n(\td C)),$ where 
$v_1:=\diag(v_0,\overbrace{1_{\td C},1_{\td C},...,1_{\td C}}^{n-2}).$

Let $w\in M_2(\C\cdot 1_{\td C})$ be the scalar matrix 
such that $\pi_\C^{C}(v_0)=\pi_\C^{C}(w).$ By replacing
$v_0$ by $v_0w^*,$ we may assume that $\pi_\C^{C}(v_0)=1_{M_2(\td C)}.$ 
%In other words, 
Hence $v_0\in M_2(C)^\sim.$   Write $v_0=1_{M_2(\td C)}+y$ for some $y\in M_2(C).$
It follows that $\pi_\C^{B}(u_1v_1^*)=1.$ Then, by Lemma \ref{Lmatrix}, 
$u_1v_1^*\in U_0(\td B).$   Let $v_2:=1_{\td A}+y.$
Then 
$
uv_2^*\in U_0(\td A). 
$
Since $4[c]\le [a]$ and  $A$ has almost stable rank one,  by Lemma 3.2 of \cite{eglnp}, there is a unitary $V\in \td A$
such that  (note that $\eta>0$)
\beq
V^*M_2(C)V\subset {\rm Her}(a).
\eneq
Then
%\beq
$V^*yV\in {\rm Her}(a).$
%\eneq 
Define $v=V^*v_2V.$ Then $v$ has the form described in the lemma.
Put $W:=V^*uv_2^*V.$  Since $uv_2^*\in U_0(\td A), $ one has 
\beq
\begin{pmatrix} W & 0\\ 0 &1\end{pmatrix}\in U_0(M_2(\td A)).
\eneq
Applying Lemma \ref{C1}, one concludes $W\in U_0(\td A).$ 
Thus 
\beq\label{TK1-10}
(V^*uV)v^*\in U_0(\td A).
\eneq
There exists a continuous path of unitaries $\{H(t):t\in [0,1]\}\subset U(M_2(\td A))$
such that 
\beq
H(0)=\begin{pmatrix} V^*uV &  0\\
                                       1 & 0\end{pmatrix}\begin{pmatrix} v^* & 0 \\ 0 & 1\end{pmatrix}
                                       \andeqn
                                       H(1)=\begin{pmatrix} u & 0\\ 1& 0\end{pmatrix}\begin{pmatrix} v^* & 0 \\ 0 & 1
                                       \end{pmatrix}=\begin{pmatrix} uv^* & 0 \\ 0 & 1\end{pmatrix}.
                                       \eneq
By \eqref{TK1-10}, $H(0)\in U_0(M_2(\td A)).$  Therefore $\diag(uv^*,1)\in U_0(M_2(\td A)).$
Applying Lemma \ref{C1}  again,  one obtains $uv^*\in U_0(\td A).$

To see part (2), we may assume that $\af=1.$ 
Let $\iota: D\to A$ be the inclusion map. Since $D$ is a full hereditary \CA\, and $A$ is separable,
it follows that 
$D$ is stably isomorphic to $A$ and $\iota_{*1}: K_1(D)\to K_1(A)$ is an isomorphism (see, for example, Corollary 2.10 of \cite{Br}).
Let $u_1=1_{\td D}+x.$ Then $\iota_{*1}([u_1])=[u]$ is zero in $K_1(A)$ from the assumption that  $u
\in U_0(\td A).$ Thus 
$[u_1]$ is zero in $K_1(D).$
Therefore, for some integer $n\ge 1,$
\beq
\diag(u_1, \overbrace{1_{\td D},..., 1_{\td D}}^{2n+1})\in U_0(M_{2n}(\td D)).
\eneq
Since $D$ is a finite separable regular simple \CA, by repeatedly applying Lemma \ref{C1},
we conclude that $u_1\in U_0(\td D).$

\end{proof}

\begin{cor}\label{C2}
Let $A$ be a separable regular simple \CA. Then 
the map  
\beq
U(M_n(\td A))/U_0(M_n(\td A))\to U(M_{n+1}(\td A))/U_0(M_{n+1}(\td A))
\eneq
is an isomorphism. In particular, $U(\td A)/U_0(\td A)=K_1(A).$
\end{cor}

\begin{proof}
The finite case follows immediately from Theorem \ref{TK1}.
Suppose that $A$ is  a purely infinite simple \CA. By the comment before Remark 3.1
of \cite{BP95}, $eAe$ is extremally rich for any projection $e\in A.$  Applying Proposition 5.4
of \cite{BP95}, one concludes that $A$ is extremally rich.  Since $\C$ has stable rank one,
by Proposition 6.8 of \cite{BP95}, $\td A$ is extremally  rich.  By \cite{Z}, $A$  has real rank zero,
and, hence,
$\td A$ has real rank zero.   By theorem 6.10 of \cite{BP2}, the corollary follows (when $A$ 
%for the case that $A$ 
is a purely infinite 
simple \CA).
%%%%%%%
\iffalse
, and $u\in U(M_n(\td A))$ for some $n>1.$
By multiplying a scalar unitary matrix, we may assume that 
$u\in 1+ M_n(A).$ Since $A$ is purely infinite,  by \cite{Z}, it has real rank zero, 
therefore, for any $1>\ep>0,$ there is a projection $p$ of the form $\diag(e,e,...,e)$
($e$ repeats $n$ times). 
for some projection $e\in A,$ and a unitary $v\in pM_n(A)p$
such that $\|u-z\|<\ep,$ where $z=(1_{M_n}-p)+v.$ Since $pM_n(A)p$ is a unital purely infinite 
simple \CA\, there is a continuous path $\{v(t): t\in [0,1]\}\subset pM_n(A)p$ 
such that $v(0)=v$ and $v(1)=\diag(v_1,e,....,e),$ where $v_1\in U(eAe).$ 
This implies that from $U(A)/U_0(A)$ to $U(M_n(A))/U_0(M_n(A))$ is surjective. 
The injectivity follow from  Theorem 1.9 of \cite{Cu} and \cite{BP}. 
\fi
%%%%%%%%%%%%%%%%%%%%%%%%%
\end{proof}

\begin{cor}\label{C3}
Let $A$ be a separable   simple ${\cal Z}$-stable \CA. Then 
the map  
\beq
U(M_n(\td A))/U_0(M_n(\td A))\to U(M_{n+1}(\td A))/U_0(M_{n+1}(\td A))
\eneq
is an isomorphism. In particular, $U(\td A)/U_0(\td A)=K_1(A).$
\end{cor}

\section{Comparison in $\td B$}

The main purpose of this section is to present Theorem \ref{LcomparisonU}
and Theorem \ref{TBtildC}.

\begin{NN}\label{41}
Let $A$ be a separable simple \CA\, and let ${\tilde{T}}(A)$ be the cone of densely defined
positive lower semi-continuous traces on $A$ equipped with the topology
of point-wise convergence on elements of the Pedersen ideal  ${\rm Ped}(A)$ of $A.$
%Then $\widetilde{T}(A)$ is a topological cone with a 
By Proposition 3.4 of \cite{TT}, $\widetilde{T}(A)$ has a  Choquet simplex $T_e$ as its base. Let $f$ be a  lower semicontinuous affine function 
on $\widetilde{T}(A)$ such that $f(t)>0$ for all $t\in \widetilde{T}(A)\setminus \{0\}.$
   Then, a standard compactness argument shows that  ${\rm inf}\{f(t): t\in T_e\}>0.$
  By I.1.4 of \cite{Alf}, together with a standard compactness  argument, 
  one obtains an increasing sequence $f_n\in \Aff_+(\widetilde{T}(A))$ such 
  that $\lim_{n\to\infty}f_n(t)=f(t)$ for all $t\in \widetilde{T}(A).$   In other words,
  $f\in \LAff_+(\widetilde{T}(A)).$ 

Now suppose that $A$ is a finite separable  regular simple \CA.
%such that $A\otimes {\cal K}$ has almost stable rank one.
It follows that $M_n(A)$ has almost stable rank one, for all $n\in \N.$ 
%of point-wise convergence on elements of the Pedersen ideal  ${\rm Ped}(A)$ of $A.$
Let us assume that every densely defined 2-quasitrace is a trace.
%that $QT(A)=T(A).$   
Then $\LAff_+({\widetilde{QT}}(A))=\LAff_+(\tilde{T}(A)).$
Let $a\in {\rm Ped}(A)_+\setminus \{0\}.$ Then $C={\rm Her}(a)$ is algebraically simple.  
 %$B$ is also regular (see \ref{Dregular}).   
 %Note that $\tau\mapsto d_\tau(a)$ is a function in $\LAff_+({\tilde T}(A)).$
Choose $f\in \Aff_+(\tilde{T}(A))\setminus \{0\}$ such that $f(\tau)<d_\tau(a)$  for all $\tau\in \tilde{T}(A)\setminus \{0\}.$
Then there is $c\in (A\otimes {\cal K})_+$ such that $d_\tau(c)=f(\tau)$ for all $\tau\in {\tilde T}(A),$  and
$c\lesssim a.$  Since $A$
%\otimes {\cal K}$ 
has almost stable rank one, 
by  \ref{Pxxxx}, 
%Lemma 3.2 of \cite{eglnp},  
there exists $x\in A\otimes {\cal K}$ such that 
$xx^*=c$ and $b:=x^*x\in C_+.$ Note that   $d_\tau(b)=f(\tau)$ for all $\tau\in {\tilde T}(A).$  By Theorem 5.3 of \cite{eglnp}, ${\rm Her}(b)$ has continuous scale.  Note also that
${\rm Her}(b)\otimes {\cal K}\cong A\otimes {\cal K}.$

 In the case that $QT(A)=T(A)$ and $T(A)$ is compact, 
 the map $f\mapsto f|_{T(A)}$  is affine and continuous, and an  order isomorphism 
 from $\LAff_+(\td T(A))$ onto  $\LAff_+(T(A))$ as ${\tilde T}(A)$  is  a convex topological  cone with the metrizable Choquet simplex $T(A)$ as its base (note $0\in \Aff_+(T(A))$--see \ref{Daff}). Therefore, since $A$ is regular (see \ref{Dregular}), in this  case, ${\rm Cu}(A)=(V(A)\setminus\{0\})\sqcup \LAff_+(T(A)).$ 
 %(see also \cite{BT} and Theorem 6.2 of \cite{TT} and the paragraph above it). 
 
%we may replace $\LAff_+({\tilde{T}}(A))$ by $\LAff_+(T(A))$ (they are not the same set)
 %as ${\tilde T}(A)$  is  a convex topological  cone with the metrizable Choquet simplex $T(A)$ as its base.
 
%By  the proof of Theorem 6.13 of \cite{RS}, if $A$ has almost stable rank one,  $A$ has stable rank at most 2.
%By \cite{Rlz}, if $A$ is a  separable simple ${\cal Z}$-stable \CA, then $A$ is regular
%(see also Theorem 6.5 of \cite{ESR-Cuntz}), 
\end{NN}

\begin{NN}\label{151}

Throughout this section, $B$ is, unless otherwise stated,  a  finite separable  stably projectionless simple \CA\, 
with continuous scale such that
$M_n(B)$ is regular for each integer $n\ge 1,$ and $QT(B)=T(B)$
(for example, $B$ is an exact   finite separable simple stably projectionless ${\cal Z}$-stable \CA\, with continuous scale
-- see \ref{Dregular}).
%{astrk1z}).  
%Note that $B\otimes {\cal K}$ has almost stable rank one.
%\CA\, which has almost stable rank one and has  the property
%that $QT(B_1)=T(B_1)$ for every hereditary \SCA\,  $B_1:={\rm Her}(a)$ with $a\in {\rm Ped}(A)$
%which has continuous scale and
%${\rm Cu}(B_1)={\rm LAff}_{+}(T(B_1)).$

Note that, by (the proof of) Theorem 6.13 of \cite{RS}, $B$ has stable rank at most two.
Also, since $B$ has continuous scale, $T(B)$ is compact (see Theorem 5.3 of \cite{eglnp}). 
%Let $\tilde{T}(B)$ be the set of all (densely defined) traces. Then $T(B)$ is a base for $\tilde{T}(B).$ 
We also 
have, as $B$ is stably projectionless,  ${\rm Cu}(B)=\LAff_+(T(B)).$ 
%(Theorem 6.6 of \cite{ESR-Cuntz}).

If $a\in (\td B\otimes {\cal K})_+\setminus \{0\},$  $\hat{a}(\tau):=\tau(a)$ for all $\tau\in T(B)$ is a function  
in  $\LAff_+(T(B))$ 
(or for all $\tau\in T(\td B)$ as a function in  $\LAff_+(\td B)$)
and $\widehat{[a]}(\tau):=d_\tau(a)$ for all $\tau\in T(B)$ is a function $\LAff_+(T(B))$ 
 (or for $\tau\in T(\td B)$  as a function in $\LAff_+(T(\td B))$). 
 Note that $\widehat{[a]}$ is a lower semicontinuous 
affine functions in $\LAff_+(T(B))$ with values in $(0, \infty].$
%If $b\in M_n(\td B)_+,$ then we also use $\widehat{[b]}(\tau)=d_\tau(b)$ for $\tau\in T(B).$ 

Note that, if $a, b\in (B\otimes {\cal K})_+$ and $d_\tau(a)\le d_\tau(b)$ for all $\tau\in T(B),$
then $a\lesssim b$ (recall that $B$ is stably projectionless).  In particular, $B$ has strict comparison
for positive elements.

Moreover, if $a, b\in  (B\otimes {\cal K})_+$ and $[a]\le [b]$ in ${\rm Cu}^\sim(B),$ 
then, as $B$ has stable rank at most 2, by Corollary 4.10 of \cite{RS}, 
\beq
[a]+2[1_{\td B}]\le [b]+2[1_{\td B}] \,\,\,{\rm in}\,\, {\rm Cu}(\td B).
\eneq
It follows that
%\beq
$d_\tau(a)\le d_\tau(b)\rforal \tau\in T(B).$
%\eneq
Therefore $a\lesssim b,$ or $[a]\le [b]$ in ${\rm Cu}(B).$   This also implies that ${\rm Cu}(B)$ is 
orderly embedded into ${\rm Cu}^\sim(B).$ 

These facts will be repeatedly used.

\end{NN}

Note that $\td B$ is unital. 
Suppose that 
% we do not assume that 
$B\not={\rm Ped}(B).$   Let $a=d+b,$ where 
 $d\in M_r(\C\cdot 1_{\td B})_+\setminus\{0\}$ and $b\in {\rm Ped}(B)_+.$
Then
% let $a=m+b\in M_r(\td B)_+$ where $m\in M_r(\C)_+$ and $b\in 
%$a \in M_r(\td B)_+\setminus M_r(B),$  
$\tau(a)=\infty$  for those $\tau\in {\tilde T}(B)$ which is not bounded (see the last part of 4.9
of \cite{EGLN}).
  If $B$ is stable, then $\tau(a)=\infty$
for all $\tau\in {\tilde T}(B).$
Therefore, it is more than convenient 
to consider a hereditary \SCA\,  of $B$ which has continuous scale (see \ref{41}).
%{Dregular}).
% (which is stably isomorphic to $B$). 

 \begin{df}\label{Dcu=}
 Let $A$ be a unital \CA\, with stable rank at most $m$ ($m\ge 1$).
 Denote by ${\rm Cu}(A)^{\circeq}$ the set of equivalence classes 
 of elements in ${\rm Cu}(A)$ with the following equivalence relation: 
 $x\circeq y$ if and only if $x+m[1_A]=y+m[1_A]$ in ${\rm Cu}(A).$
 %
 %If $A$ is unital, then 
 The the map $x\to (x, 0)$ gives an order embedding 
 from ${\rm Cu}(A)^{\circeq}$ to ${\rm Cu}^\sim (A)$ (see 3.1 of \cite{Rl}, Subsection 
 4.2 and  Corollary 4.10 of \cite{RS}).   So, in  this unital case,  we may view ${\rm Cu}(A)^{\circeq}\subset {\rm Cu}^\sim(A).$
 
 Let $B$ be a non-unital stably finite \CA\, with continuous scale and let $\tau_\C$ be the tracial state of $\td B$ 
 that vanishes on $B.$
 Define
 \beq
 \LAff_+(T(\td B))^\diamond=\{f\in \LAff_+(T(\td B)): f(\tau_\C)\in \{0\}\cup \N\cup \{\infty\}\}.
 \eneq
 %%%%%%%%%%%%%%%%%%%%%%%%%%%%%%%%
 %
 %
 \iffalse
 Recall (see A.5 of \cite{eglnkk0}) that $S(\td B)$ is the sub-semigroup of ${\rm Cu}(\td B)$ of those 
 elements $x\in {\rm Cu}(\td B)$ which is a supremum of an increasing sequence 
 $\{[a_n]\},$   where $\widehat{[a_n]}\in \Aff_+(T(B))$ and $[\pi_\C^B(a_n)]<\infty$ and 
 $[a_n]$ are not represented by projections (see also property O4 in  page 5 of \cite{Rl}). 
 By Theorem A.6 of \cite{eglnkk0}, when $A$ satisfies the assumption of \ref{151},
 the natural map $[x]\mapsto \widehat{[x]}$ is surjective from $S(\td B)$ onto $\LAff_+(\td B)^\diamond.$
 %It should be noted that if $\widehat{x}$ is continuous on $T(B),$ then $x\in S(\td B).$
 \fi
 %
 %%%%%%%%%%%%%%%%%
 \end{df}

 \begin{lem}[Theorem A.6 of \cite{eglnkk0} and Theorem 6.11 of \cite{RS}]\label{TcomparisonintdA}
  Let $B$ be in \ref{151}.
  %a 
 % separable simple stably projectionless simple \CA\, 
 %with continuous scale such 
 %that $M_m(A)$ has almost stable rank one for all $m\ge 1.$
% Suppose also $QT(A)=T(A)$ and ${\rm Cu}(A)=\LAff_+(T(A)).$ 
 Then ${\rm Cu}(\td B)^{\circeq}=(K_0(\td B)_+\setminus \{0\})\sqcup \LAff_+(T(\td B))^\diamond$
 (see lines above  Theorem 6.11  of \cite{RS} --also at the end of \ref{Dfep}).
 
 %Moreover, if $x, y\in {\rm Cu}(\td A)$ such that $\hat{x}<\hat{y},$ 
 %the  $x\lesssim y,$ If $x\not\in V(A)$ and $\hat{x}\le \hat{y},$ then $x\lesssim y.$
  
 \end{lem}
 
 \begin{proof}
 By the assumption, applying Theorem 6.11 of \cite{RS}, one obtains 
 ${\rm Cu}^\sim(B)=K_0(B)\sqcup \LAff_+^\sim(T(B))$ (see also 
 I.1.4 of \cite{Alf}  and the first part of \ref{41}).  Note that, as in the proof of Theorem 6.13 of \cite{RS},
 $B$ has stable rank at most 2.   So, the definition of ${\rm Cu}^\sim(B)$ in \cite{RS} coincides  with
 that in \cite{Rl} (see subsection 4.2 of \cite{RS}). 
 
 Let $x, y\in {\rm Cu}(\td B)$ which are  not represented by projections and are represented 
 by elements $a, b\in (\td B\otimes {\cal K})_+$  such that $[\pi_\C^{B}(a)]=n[1_{\td B}]$
 and $[\pi_\C^B(b)]=m[1_{\td B}]$ 
 for some integers $n, m\ge 0,$ respectively.  Suppose that $d_\tau(a)=d_\tau(b)$
 for all $\tau\in T(\td B).$  We will show that $x\circeq y.$ Let $\tau_\C$ be the tracial state of $T(\td B)$
 which vanishes on $B.$ The condition $d_{\tau_\C}(a)=d_{\tau_\C}(b)$ implies that $n=m.$
 It then follows from Theorem 6.11 of \cite{RS} that there exists $k$ ($=2$)
 such that  (in ${\rm Cu}(\td B)$)
 \beq
 [a]+n[1_{\td B}]+k[1_{\td B}]=[b]+n[1_{\td B}]+k[1_{\td B}].
 \eneq
 Thus $x\circeq y$ (see also Corollary  4.10 of \cite{RS}). 
 
 Now consider the case that  $[\pi_\C^B(a)]=\infty=[\pi_\C^B(b)].$ 
 Then, for any $1>\ep>0,$ $f_\ep(a)\in {\rm Ped}(B\otimes {\cal K}).$
 Hence 
 $[\pi_\C^B(f_\ep(a))]<\infty.$ Also, there is $0<\eta<\ep,$  as 
 $[a]$ is not represented by a projection and $B$ is simple, 
 \beq
 d_\tau(f_\ep(a))<\tau(f_{\eta}(a))\rforal \tau\in T(B).
 \eneq
 
 For $\pi_\C^B(b),$  since $0$ is the only non-isolated point of the spectrum of $\pi_\C^B(b),$
 one may find $g\in C_0((0,\|b\|])_+$ such 
 that $[\pi_\C^B(g(b))]=[\pi_\C^B(f_\ep(a))].$  Let $m:=[\pi_\C^B(f_\ep(a))]<\infty.$
 Note that $T(B)$ is compact. One then can find $0<\dt<\eta/2$ such that $[\pi_\C^B(f_\dt(b))]= [\pi_\C^B(f_\eta(a))]$ and 
  \beq
 d_\tau(f_\ep(a)) < \tau(f_\eta(a))<d_\tau(f_\dt(b))\rforal \tau\in T(B).
  \eneq
  Consider $C:=\overline{f_{\dt/4}(b)(B\otimes {\cal K})f_{\dt/4}(b)}$ and let $\{e_n\}$ be 
  an approximate identity for $B\otimes {\cal K}.$ 
  Then 
  \beq
  \tau(f_{\dt/4}(b)^{1/2}e_nf_{\dt/4}(b)^{1/2})\nearrow \tau(f_{\dt/4}(b))\rforal \tau\in T(B).
  \eneq
  It follows that (recall $T(B)$ is compact) there is $n_0\ge 1$ such 
  that, for all $n\ge n_0,$ 
  \beq\label{TcomparisonintdA-1}
  \tau(f_{\dt/4}(b)^{1/2}e_nf_{\dt/4}(b)^{1/2})>\tau(f_{\dt/2}(b))\ge d_\tau(f_\dt(b))> d_\tau(f_\ep(a))
  \rforal \tau\in T(B).
  \eneq
  Choose $b'=f_{\dt/4}(b)^{1/2}e_{n_0+1}f_{\dt/4}(b)^{1/2}+g(b).$
  Then $[b']\le [b]$ in ${\rm Cu}(\td B).$  
  We also have $[\pi_\C^B(b')]=[\pi_\C^B(g(b)]=[\pi_\C^B(f_\ep(a))]=m.$ 
It follows from \eqref{TcomparisonintdA-1} that 
\beq
d_\tau(b')>d_\tau(f_\ep(a))\rforal \tau\in T(B).
\eneq
 It follows that
 \beq
 d_\tau(b')-m\ge d_\tau(f_\ep(a))-m\rforal \tau\in T(B).
 \eneq
 It follows from  Theorem 6.11 of \cite{RS} that in ${\rm Cu}^\sim(B),$
 \beq
 (b', m)\ge (f_\ep(a),m)
 \eneq
 which means  that, for some integer $k\ge 1,$  in ${\rm Cu}(\td B),$ 
 \beq
 [b]+m[1_{\td B}]+k[1_{\td B}]\ge [b']+m[1_{\td B}]+k[1_{\td B}]\ge [f_\ep(a)]+m[1_{\td B}]+k[1_{\td B}].
 \eneq
 Since $B$ has stable rank at most two, by Corollary 4.10 of \cite{RS},
% \beq
 $[b]+2[1_{\td B}]\ge  [f_\ep(a)]+2[1_{\td B}].$
% \eneq
 Since the above holds for all $0<\ep<1,$ one concludes that
 \beq
 [b]+2[1_{\td B}]\ge [a]+2[1_{\td B}].
 \eneq
 The same argument also shows that
 \beq
 [a]+2[1_{\td B}]\ge [b]+2[1_{\td B}].
 \eneq
 It follows that  $[a]\circeq [b].$ 
 This shows that the map from ${\rm Cu}(\td B)^\circeq$ to $(K_0(\td B)_+\setminus \{0\})\sqcup \LAff_+(T(\td B))^\diamond$
 is an order embedding. 
 
 The map is surjective follows from the first part of Theorem A.6 (and Def.~A.5) of \cite{eglnkk0}. 
 %
 %
 %%%%%%%%%%%%%%%%%%%%%%%%%%%%%%%%%%%%%%%%%%%%%
 %

 %%%%%%%%%%%%%%%%%%%
 \iffalse
 %
 Now let $f\in \LAff_+(T(\td B))^\diamond$ be such that $f(\tau_\C)=n$ for some integer $n\ge 1.$
 Set $g:=(f-n)|_{T(B)}.$ Then $g\in \LAff_+^\sim(T(B)).$ 
 It follows from Theorem 6.11 of \cite{RS}  that there exists $a\in (\td B\otimes {\cal K})_+$
 with $[\pi_\C^B(a)]=m<\infty$ such that $\widehat{[a]}-m=g.$

 Then $\hat{a}\in \LAff_+(T(B))$ and there are $a_k\in (\td B\otimes {\cal K})_+$ with 
 $[\pi_\C^B(a_k)]=n[1_{\td B}]$ such that $\hat{a_k}\in \Aff_+(T(B))$ and 
 $\hat{a}=\sup\{\hat{a_k}: k \}.$ 
 It then follows Theorem 6.11 of \cite{RS} such that $x+m[1_{\td B}]=\sup\{[a_n]\}+m[1_{\td B}]$ 
 (see also the second paragraph of \cite{RS}), where $m=2.$
 Thus there is an element $y\in S(\td B),$   where $S(\td B)$ is the sub-semigroup 
 of ${\rm Cu}(\td B)$ generated by elements in ${\rm Cu}(B)$ and 
 those $z\in {\rm Cu}(\td B)$ which is a supremum of an increasing sequence 
 $\{y_n\}$ such that $\hat{y_n}\in \Aff_+(T(B)),$ 
  such that 
 $y\circeq x$ (see Theorem A.6 of \cite{eglnkk0}).   Since every element $x\in {\rm Cu}(\td B)$ which is not represented 
 by a projection 
 is a supremum of elements $\{x_n\}$ with $[\pi_C^{B}(x_n)]<\infty,$ we conclude 
 that, for every $x\in {\rm Cu}^\sim(\td B),$ there is $y\in S(\td B)$
 such that $x\circeq y.$  Then the conclusion follows from Theorem A.6 of \cite{eglnkk0}.
 \fi
 %
 %T%%%%%%%%%%%%%%%%%%%%%%

 \end{proof}

 \begin{lem}\label{Linvdense}
Let $A$ be a \CA\, 
%such that $M_r(A)$ 
which  has almost stable rank one.
% for some integer $r\ge 1.$ 
Suppose 
that $a\in M_r(A)$ and $b\in M_r(\td A)$ (for some $r\ge 1$).
Then ${\rm dist}(x, LG(M_{2r}(\td A))=0,$ where $LG(M_{2r}(\td A))$ is the set of invertible 
elements in $M_{2r}(\td A)$ and $x=\begin{pmatrix} a & b \\ 0 & 0\end{pmatrix}.$

 \end{lem}
 
 \begin{proof}
 For any $\ep>0,$ by  Proposition \ref{Pmatrixsr1}, there is an invertible element  $y\in M_r(\td A)$ with 
 the inverse $y^{-1}$ such that $\|a-y\|<\ep.$
 Put 
 \beq
 z:=\begin{pmatrix} y & b\\ 0  & \ep\end{pmatrix}\andeqn w:=\begin{pmatrix} y^{-1} &0\\ 0 & 1/\ep\end{pmatrix}.
 \eneq
 Then $\|x-z\|<\ep$ and $w$ is invertible and 
 \beq
 zw=\begin{pmatrix} 1 & b/\ep\\0  & 1\end{pmatrix}.
 \eneq
 Since $\begin{pmatrix} 0 & b/\ep\\ 0 &0\end{pmatrix}$ is a nilpotent, $zw$ is invertible.
 As $w$ is invertible,  $z$ is invertible.

 \end{proof}

 \begin{lem}\label{Linvt}
 Let $B$ be as in \ref{151}.   Suppose that $a,b\in M_r(B)_+$ and $c, d\in M_r(\td B)_+$
 such that $a\lesssim b$ (in $M_r(B)$) and $c\lesssim d$ (in $M_r(\td B)$)
 for some integer $r\ge 1.$ Suppose also
 that $b\perp d.$
 
 Then, for any $\eta>0$ and $\ep>0,$  there is a  unitary $U\in M_{2r}(\td B),$  $\dt>0,$ 
 and $h\in {\rm Her}(f_{\dt}(b+d))_+$ with $\|h\|\le 1$
 %sequence of  unitaries  $U_n\in M_{2r}(\td B)$ 
% and a sequence of elements $h_n\in {\rm Her}(b+d)_+$ 
such that
 \beq
%U^*f_\eta(a+c)U\in {\rm Her}(f_\dt(b+d)).
%\lim_{n\to\infty}
\| U^*f_\eta(a+c)U-h\|<\ep.
 \eneq
 %
 %Moreover, if there is $x_2\in M_r(\td B)$ such that $x_2^*x_2=c$ and $x_2x_2^*\in {\rm Her}(d),$
% then one may choose $U$ so that
% \beq
% \|U^*(a+c)U-h\|<\ep\andeqn h\in  {\rm Her}(f_{\dt}(b+d))_+.
% \eneq
  \end{lem}
  (We identify $M_r(\td B)$  with the hereditary \SCA\, 
  $\{ \begin{pmatrix} b & 0\\
                                0 & 0\end{pmatrix}\in M_{2r}(\td B): b\in M_r(\td B)\}.$)
 
 \begin{proof}
 Fix $1>\eta>0.$ There are $x_1\in M_r(B)$ and $x_2\in M_r(\td B)$  
 such that (see \ref{Pxxxx})
 \beq\label{invt-100}
 %\|x_1^*x_1-a\|<\eta/4,\,\, 
 x_1^*x_1=a,\,\,\, x_1x_1^*\in {\rm Her}(b),\,\, \|x_2^*x_2-c\|<\eta/4,\andeqn 
 x_2x_2^*\in {\rm Her}(d).
 \eneq 
  Put ${x:=\begin{pmatrix} x_1 &   0\\ x_2 & 0\end{pmatrix}}.$ Then 
  ${x^*x=\begin{pmatrix}{{x_1^*x_1}}+x_2^*x_2 & 0\\ 0 &0\end{pmatrix}}.$
       By Proposition 2.2 of \cite{Rr2}, there is $r\in M_r(\td B)$
  such that $r^*(x_1^*x_1+x_2^*x_2)r=f_{\eta/2}(a+c).$ 
  Let $y=\begin{pmatrix} x_1r &  0\\ x_2 r & 0\end{pmatrix}.$
  Then 
  %\beq
 ${y^*y=\begin{pmatrix}f_{\eta/2}(a+c) & 0\\ 0 &0\end{pmatrix}}.$
  %\eneq
  
  Let $y=v|y|$ be the polar decomposition of $y$ in $M_{2r}(\td B)^{**}.$ By applying Lemma \ref{Linvdense} above
and Theorem 5 of \cite{Pedjot87}, there is, for any $\sigma>0,$  a unitary $W\in M_{2r}(\td B)$ such that
 \beq
Wf_{\sigma}(|y|)=vf_{\sigma}(|y|).
 \eneq
 We choose a sufficiently small $\sigma$ so that
\beq
 Wf_\eta(a+c)^{1/2}=vf_\eta(a+c)^{1/2}.
 \eneq
 Then
\beq
 Wf_{\eta}(a+c)W^*= vf_\eta(a+c)v^*\le f_{\eta/4}(yy^*).
 \eneq
 Note that
 (see I.1.11 of \cite{BH})
 \beq
 yy^*=\begin{pmatrix} x_1rr^*x_1^* &  x_1rr^*x_2^*\\ x_2rr^*x_1^* & x_2rr^*x_2^*\end{pmatrix}\le 2\begin{pmatrix} 
  x_1rr^*x_1^* & 0\\ 0 & x_2rr^*x_2^*\end{pmatrix}.
 \eneq
Thus $yy^*\in {\rm Her}( f),$ where $f:=\begin{pmatrix} b^{1/2} & 0 \\  0 & d^{1/2}\end{pmatrix}.$
 %
% Let $y=u|y|$  be the polar decomposition of $y$ in $M_{2r}(\td B)^{**}.$ 
% For any $0<\ep<\eta/16,$ by Lemma \ref{Linvdense}  and Theorem 5 of \cite{Pedjot87} again, there is a unitary $W_1\in M_{2r}(\td B)$ such that 
 %$W_1f_\ep((y^*y)^{1/2})=uf_\ep((y^*y)^{1/2}).$ It follows that
% \beq
% W_1	f_\ep(|y|)W_1^*=f_\ep((yy^*)^{1/2})\in {\rm Her}(b+d).
% \eneq
%
%
%%%%%%%%%%%%%%%

%

%%%%%%%%%%%%%%      %%%%%%%%%%%%%
 
 %
 Put $z:=\begin{pmatrix} b^{1/2} & d^{1/2} \\ 0 & 0\end{pmatrix}.$ 
 Then (recall $b\perp d$)
 \beq
 zz^*=b+d\andeqn z^*z=\begin{pmatrix} b & 0 \\ 0 & d\end{pmatrix}.
 %\begin{pmatrix} a & ab\\ ba & b\end{pmatrix}\le 2\begin{pmatrix} a &0\\ 0 & b\end{pmatrix}.
 \eneq
 For any $1>\ep>0,$ 
 choose $0<\dt<1/2$ such that
 \beq
 \|f_{\dt}(|z|)Wf_\eta(a+c)W^*f_{\dt}(|z|)-Wf_\eta(a+c)W^*\|<\ep.
 \eneq
 Let $z=u|z|$ be the polar decomposition of $z$ in $M_{2r}(\td B)^{**}.$ 
 %For any $\ep>0,$ 
 By Lemma \ref{Linvdense}   above and Theorem 5 of \cite{Pedjot87} again, there is a unitary $W_1\in M_{2r}(\td B)$ such that 
 $W_1f_\dt(|z|)=uf_\dt(|z|).$ It follows that
 \beq
 W_1	f_{\dt}(|z|)W_1^*=f_{\dt}(b+d).
 %\in {\rm Her}(b+d).
 \eneq
% Note that, for any $g\in {\rm Her}(f),$
 %$\lim_{\ep\to 0}\|f_\ep(|z|)gf_\ep(|z|)-g\|=0.$ Thus we obtain a sequence of unitary $W_n\in M_{2r}(\td B)$
 %and a sequence of element
Let $h=W_1f_{\dt}(|z|)Wf_\eta(a+c)W^*f_{\dt}(|z|)W_1^*\in {\rm Her}(f_{\dt}(b+d))_+.$ Then  $\|h\|\le 1$ and 
 %$h_n\in {\rm Her}(b+d)_+$ 
 %such that
 \beq\nonumber
 %W_nW f_\eta(a+c)W^*W_n\in M_r(\td B)\andeqn
  %\lim_{n\to\infty}
  \|W_1W f_\eta(a+c)W^*W_1-h\|<\ep.
 \eneq
 %
% For the last part of the statement,  in \eqref{invt-100}, we have $x_2^*x_2=c.$ 
% Then define $y:=\begin{pmatrix} x_1&  0\\ x_2 & 0\end{pmatrix}.$
 %Then $y^*y=a+c.$
 %%%%%%%%
 
 %
 %
 %
 %%%%%%%%%%%%%%%%%%%%%%%%%%%%%%%%%%%%%%%%%%%
 \end{proof}
 %
 %%%%%%%%%%%%%%%%%%%%%%%%%%%%%%%%%%%
 
 %
 %
 %

 %%%%%%%%%%%%%%%%%%%%%%%%%%%%%%%%%%%%%%%%%
 %\iffalse
 %
 %
 \begin{df}[A.1 of \cite{eglnkk0}]\label{Domega}
 Let $B$ be a separable \CA\, with compact $T(B)\not=\emptyset.$
 Let $a\in M_n(\td B)_+$ define 
 \beq
 \omega([a])=\inf \{\sup\{d_\tau(a)-\tau(c): \tau\in T(B)\}:   0\le c\le 1\andeqn c\in \overline{aM_n(\td B)a}\}.
 \eneq
 Note that $\widehat{[a]}$ is continuous on $T(B)$ if and only 
 if $\omega([a])=0,$ and if $a\sim b,$ then $\omega([a])=\omega([b]).$ 
 (see A.1 of \cite{eglnkk0}).  If $B$ has continuous scale and $p\in M_n(\td B)$  is a projection,
 then  $\hat{p}$ and $\widehat{1_{M_n(\td B)}-p}$ are both lower semicontinuous. Thus both 
 are continuous.
 %
 %
 %%%%%%%%%%%%
 \iffalse
  Let $p\in M_n(\td B)$  be a projection.   If $a\in pM_n(\td B)p$ and $b\in (1-p)M_n(\td B)(1-p),$ 
  then  $\omega([a+b])\le \omega([a])+\omega([b]).$
  
  Let $\ep>0.$ There is $c\in {\rm Her}(a+b)_+$ with $0\le c\le 1$
  such that 
  \beq
  \tau(c)>d_\tau(a+b)-\omega([a+b])-\ep\rforal \tau\in T(B).
  \eneq
  Put $c_1=pcp+(1-p)c(1-p).$  One checks that $c_1\in {\rm Her}(a+b).$ Moreover,
  $\tau(c)=\tau(pcp)+\tau((1-p)c(1-p))$ for all $\tau\in T(B).$
   Then, for all $\tau\in T(B),$
  \beq
  \omega([a+b])+\ep>d_\tau(a+b)-\tau(c)=d_\tau(a)+d_\tau(b)-\tau(pcp)-\tau((1-p)c(1-p))\\
  =(d_\tau(a)-\tau(pcp))+(d_\tau(b)-\tau((1-p)c(1-p)).
  %\ge \omega(a)+\omega(b).
  \eneq
  Thus $\omega([a+b])+\ep\ge \sup_{\tau\in T(B)}\{d_\tau(a)-\tau(pcp)\}.$
  We conclude that $\omega([a+b])\ge \omega([a]).$
  \fi
  %%%%%%%%%%%%%%%%%%%%%%%%%%%%%%%%%%%%%
 \end{df}
% \fi
 %%%%%%%%%%%%%%%%%%%%%

 %
 %
 %
 %%%%%%%%%%%%%%%%%%%%%%%%%%%%%%%%%%%%%%%
   \begin{lem}\label{Lorthog}
  Let $B$ be a nonunital simple \CA\, and $a\in M_n(\td B)_+$  with $0\le a\le 1$
  such that $0$ is not an isolated  point of\, ${\rm sp}(a).$ 
  Then, for any $1/2>\dt_0>0,$ there exists $0<\dt<\dt_0$ such that 
  there is    an element 
  $b\in {\rm Her}(f_{\dt}(a))_+\setminus \{0\}$
  such that $b\perp f_{\dt_0}(a)$ and there is a nonzero element $c\in M_n(B)_+\cap {\rm Her}(b)_+.$
  
  Moreover, if $T(B)$ is a nonempty  compact set,  then 
  \beq\label{Lorthog-1}
  \inf\{\tau(c): \tau\in T(B)\}>0.
  \eneq
  
  \end{lem}
   \begin{proof} 
  The existence of $b$ 
follows from the spectral theory immediately. 
  For the existence of $c,$ note, since $b\not=0,$
  $\overline{bM_n(B)b}\not=\{0\}.$ Choose $c\in \overline{bM_n(B)b}_+\setminus \{0\}.$
  By the simplicity of $M_n(B),$ $\tau(c)>0$ for all $\tau\in T(B).$ Since 
  we also assume that $T(B)$ is compact, inequality  \eqref{Lorthog-1}  holds.
  
  \end{proof}

 % It is also important that we have the unitaries $U_n.$ 
  \begin{lem}[Compare Lemma A.3 of \cite{eglnkk0}]\label{LLcomparisonU}
  %{LLcomparisonU}
 Let $B$ be as in \ref{151} and 
 $a, b\in M_r(\td B)_+$ (where $r\ge 1$ is an integer). 
 %such that 
 %$\pi_\C^{B}(b)$ has finite rank  and 
 %$\widehat{[b]}$ is continuous on $T(B).$
 Suppose that $\pi_\C^B(a)\lesssim \pi_\C^B(b)$ and 
 \beq
d_\tau(a)
+4\omega([b])
<d_\tau(b)\rforal \tau\in T(B).
 \eneq
 Then, for any $1>\eta>0,$ there exists a sequence of unitaries 
 $U_n\in M_{2r}(\td B)$  and a sequence of elements $h_n\in {\rm Her}(b)_+$ with $\|h_n\|\le 1$
 such that
 \beq\label{T48eq}
 %U_n^*f_\eta(a)U_n\in M_r(\td B)\tand 
 \lim_{n\to\infty}\|U_n^*f_\eta(a)U_n-h_n\|=0.
 \eneq
 In particular, $a\lesssim b.$
  \end{lem}

\begin{proof}

Let us assume that $0\le a, \, b\le 1.$ 
It suffices to prove that \eqref{T48eq} holds for $f_{\eta_1}(a)$ in place of $a$ for any $0<\eta_1<1.$ 
If $[a]$ is represented by a projection, 
then   $d_\tau(a)$ is continuous. So
\beq
\inf\{d_\tau(b)-d_\tau(a): \tau\in T(B)\}>4\omega([b]).
\eneq
%Since, $f_\ep(a)\lesssim b$ for all 
%$\ep>0$ implies $a\lesssim b,$ \wilog, we may assume 
%that $a\in \mathrm{Ped}(A).$ In particular, $d_\tau(a)$ is bonded. 
%Note that, for each $\eta>0,$ $f_\eta(b)\in \mathrm{Ped}({\widetilde A}\otimes {\cal K}).$ So 
%$d_\tau(f_\eta(b))$ is finite. By a standard compact argument, we may assume 
%that $b\in \mathrm{Ped}(A).$
%
Otherwise, 
for  any fixed  $0<\eta_1<1/2,$ 
%By applying 7.1 of \cite{eglnp}, 
there exist 
$\eta_1>\eta_1/2>\eta_2>\eta_2/2>\eta_3>0$   such that  %(recall that $B$ is stably projectionless and ${\rm sp}(a)=[0,1]$)
%and a continuous function $f: T(B) \to \R^+$ such that
\beq
d_\tau((a-\eta_1)_+)<\tau(f_{\eta_2}(a))<d_\tau((a-\eta_3)_+) <d_\tau(a)\rforal \tau\in T(B).
\eneq
Then 
\beq
\inf\{d_\tau(b)-d_\tau(f_{\eta_1}(a)): \tau\in T(B)\}> 4\omega([b]).
%+\inf\{d_\tau(b)-d_\tau(f_{\eta_1}(a)): \tau\in T(B)\}.
\eneq
%Since $\eta$ in the lemma is arbitrarily in $(0,1),$ we may 
%By replacing $a$ by 
%$f_{\eta_1}(a)$ 
Thus, in both cases,  we may assume, \wilog\, (replacing $a$ by $f_{\eta_1}(a)$)  that 
\beq\label{1131-1}
\inf\{d_\tau(b):\tau\in T(B)\}>d=\inf\{d_\tau(b)-d_\tau(a): \tau\in T(B)\}>4\omega(b).
\eneq

By applying  Lemma A.2 of \cite{eglnkk0},
%\ref{Ldecomp}, 
one obtains non-zero  elements
%and mutually orthogonal elements 
$b_0\in M_r(B)_+$ and $b_1, {{b'}}\in M_r(\td B)_+$ with $b_0\perp b_1$ such that 
\beq
&&b_0+b_1\le b',\,\,\, [b'] =[b], \,\,\pi_\C^B(b_1)=\pi_\C^B(b'),\\\label{Ldecomp-n1}
&&2\omega([b])<d_\tau(b_0)<d/2,\,\,\, d_\tau(b_1)>d_\tau(b)-d/2\rforal \tau\in T(B),
\eneq
and, for any $c_n'\in M_r(B)_+$ with $c_n'\in \overline{b_1M_r({\widetilde B})b_1}$ and $d_\tau(c_n')\nearrow d_\tau(b_1)$ on $T(B),$
there exists $n_0\ge 1$ such that
\beq\label{Notef1-10}
d_\tau(b_1)-d_\tau(c_n')<\omega([b])+(1/64)\inf\{\tau(b_0):\tau\in T(B)\} \rforal \tau\in T(B).
\eneq
In fact, the proof of Lemma A.2 of \cite{eglnkk0} states that $b'=g_{1, \eta_1}(b)$ for some strictly positive function 
$g_{1,\eta_1}$ on $(0, \|b\|]$  as in the proof of Lemma A.2 of \cite{eglnkk0} (we replace $a$ by $b$ and $a'$ by $b'$).
Recall from A.1 of \cite{eglnkk0} that $\omega([b'])=\omega([b]).$
Moreover, $[\pi_\C^B(b_1)]=[\pi_\C^B(b)].$
Replacing $b$ by $b',$ \wilog, we may assume that $b_0+b_1\le b.$ 
We may also assume $0\le b_0+b_1\le b\le 1.$
Note, for any integer $m\ge 1,$ that $b_0+b_1^{1/m}\le b_0^{1/m}+b_1^{1/m}=(b_0+b_1)^{1/m}\le b^{1/m}.$ 
By choosing large $m,$ we may assume that $\pi_\C^B(b_1^{1/m})=\pi_\C^B(b^{1/m})=p_1$ is a projection. 
Replacing $b_1$ by $b_1^{1/m}$ and $b$ by $b^{1/m},$
we may further assume that $\pi_\C^B(b_1)=\pi_\C^B(b)=p_1.$ 
Similarly, we may assume that $\pi_\C^B(a):=p_2$ is also a projection. Since $\pi_\C^B(a)\lesssim \pi_\C^B(b),$
there is a scalar matrix $U_0\in M_r(\C \cdot 1_{\td B})$
 such that $\pi_\C^{B}(U_0^*aU_0)\le p_1.$    Hence we may also assume that $p_2\le p_1.$
% Note, by A.1 of \cite{eglnkk0}, 
% $\omega([b])=\omega([b^{1/m}]).$   

 %\Wlog, we may  assume that $p_2\le p_1.$  We may also assume $\|b\|\le 1.$ 
 We may further assume  that there are integers $m_2\le m_1$ such that
 \beq
 p_i=\diag(\overbrace{1,1,...,1}^{m_i},0,\cdots, 0),\,\,\,i=1,2.
 \eneq
 Let $P_i=\diag(\overbrace{1_{\td B}, 1_{\td B},...,1_{\td B}}^{m_i},0,\cdots 0),$ $i=1,2.$

Put $d_0=\inf\{\tau(b_0):\tau\in T(B)\}.$  Note that the above 
holds for the case that $\omega([b])=0.$
%
%There exists an invertible matrix $y_1\in {\mathrm{M}}_r(\C)_+$
%such that $y_1^{1/2}\pi(b_1)y_1^{1/2}=p_1$ is a projection.
%Let $Y_1\in {\mathrm{M}}_r({\widetilde A})$ denote the scalar matrix such that $\pi(Y_1)=y_1.$
%Note that $\la Y_1^{1/2}b_1Y^{1/2}\ra =\la b_1\ra,$ 
%$Y_1^{1/2}c_nY^{1/2}\le Y_1^{1/2}b_1Y^{1/2}$ and $d_\tau(Y^{1/2}c_nY^{1/2})=d_\tau(c_n).$
%So, 
%\wilog, 
%replacing $b_1$ by $Y_1^{1/2}b_1Y^{1/2},$ we may assume that 
%$\pi(b_1)=p_1.$ 
%By replacing $a$ by $Y_1^{-1}b_1Y_1,$ \wilog, we may assume that $\pi(a)=p_1.$
%Similarly, we may assume that $\pi(a)=p_2$ is also a projection. 
%There is a scalar matrix $U\in {\mathrm{M}}_r({\widetilde A})$ such that 
%$\pi(U^*aU)\le p_2.$ \Wlog, we may assume that $p_2\le p_1.$ 
%
Note  that $(b_1-1/n)_+\le b_1$ and $d_\tau((b_1-1/n)_+)\nearrow d_\tau(b_1),$ 
so
by \eqref{Notef1-10},
for some $\dt_1>0,$
\beq\label{Notef1-11+}
d_\tau(b_1)-d_\tau(f_{\dt}(b_1))<\omega([b])+d_0/64\rforal \tau\in T(B)
\eneq
and 
all
$0<\dt<\dt_1.$ 
We also assume $\pi_\C^B(f_\dt(b_1))=p_1$ ($0<\dt\le \dt_1$).
\iffalse
Since $f_\dt(b_1)f_{\dt/2}(b_1)=f_\dt(b_1),$
this also implies that 
\beq\label{Notef1-11}
d_\tau(b_1)-\tau(f_{\dt}(b_1))<\omega(b)+d_0/64\rforal \tau\in T(B)
\eneq
for all $0<\dt<\dt_1/2.$ 
\fi
\iffalse
Put  $g(t)=t$ for $t\in [0,1]$ and $g(t)=1$ for $t\ge 1.$
Consider $g(a)$ and $g(b).$ We note that $0$ is still a limit point of the spectrum of $g(a)$ and $g(b).$
Now $\|g(a)\|, \|g(b)\|\le 1.$ Note that $\pi(g(a))=\pi(a)$ and $\pi(g(b))=\pi(b).$ 

Let 
\beq
d=\inf\{d_\tau(b)-f(\tau): \tau\in T(B)\}>0.
\eneq

Since $d_\tau(b)$ is continuous on $T(B),$ there is $\dt_0>0$ 
such that
\beq\label{Notef1-3}
0<d_\tau(b)-\tau(f_\dt(b))<d/16\rforal \tau\in T(B).
\eneq

By applying Lemma , 
\fi 
%
Let $\{e_n\}$ be an approximate identity for $B$
such that $e_ne_{n+1}=e_{n+1}e_n=e_n,$ $n=1,2,....$
Put
\beq\label{48-eq1202-1}
E_n=%\d
%{Lomega-23}
\diag(e_n, e_n,...,e_n)\in M_r(B),\,\,\, n=1,2,....
\eneq
Then $\{E_n\}$ is an approximate identity for $M_r(B),$ and 
for all $i$ and $n,$ 
 \beq\label{LcomparisonU-10}
 E_nP_i=P_iE_n\andeqn E_n(1-E_k)=0=(1-E_k)E_n, \,{\rm if}\,\, \,\,k\ge n+1.
 \eneq

We have $b_1^{1/2}E_nb_1^{1/2}\nearrow b_1$ (in the strict topology). Let $c_n= E_n^{1/2}b_1E_n^{1/2},$ $n=1,2,....$ 
It follows that $d_\tau(c_n)\nearrow d_\tau(b_1)$ on $T(B).$
By the construction of $b_1,$ there exists $n_0\ge 1$ such that
\beq\label{Notef1-8n}
d_\tau(b_1)-d_\tau(b_1^{1/2}E_nb_1^{1/2})=d_\tau(b_1)-d_\tau(c_n)<\omega([b])+d_0/64
\eneq
for all $\tau\in T(B)$ and for all $n\ge n_0.$

One then computes, by \eqref{Notef1-8n}, \eqref{Ldecomp-n1} and \eqref{1131-1},  that, for $n\ge n_0,$ 
for all $\tau\in T(B),$
\beq\nonumber
&&\hspace{-0.4in}d_\tau(c_n)>d_\tau(b_1)-\omega([b])-d_0/64>d_\tau(b)-d/2-\omega([b])-d_0/64\\\label{Notef1-8}
&&>d_\tau(a)+d/2-\omega([b])-d_0/64>d_\tau(a)+d/4-d_0/64>d_\tau(a).
%d_\tau(a)<d_\tau(c_n)\rforal \tau\in T(B).
%andeqn\, n\ge n_0.
\eneq

Since $0\le a\le 1$ and $\pi_\C^B(a)=\pi_\C^{B}(P_2),$ for any $0<\eta<1/2,$ 
%To prove the lemma, it suffices to show \eqref{T48eq} holds 
%for any small $\eta>0.$ Therefore we may only consider those $\eta>0$ such that
$\pi_\C^{B}(f_{\eta/2}(a))=\pi_\C^B(a)=p_2.$
%=\pi_\C^{B}(P_2).$
%
Put $a_k=E_k f_{\eta/2}(a)E_k,$ $k=1,2,....$  Then, by \eqref{Notef1-8}, 
 $a_k\lesssim c_n$ for any $k\ge 1$ and $n\ge n_0,$
 as $B$ has the strict comparison. 

On the other hand,  
 since $\pi_\C^{B}(f_{\eta/2}(a))=\pi_\C^B(a)=\pi_\C^{B}(P_2)$ and 
 $\pi_\C^{B}(b_1)=\pi_\C^{B}(P_1),$ 
 $$ 
 b_1=P_1+b_{00},\andeqn f_{\eta/2}(a)=P_2+a_{00}
 $$
 for some  $b_{00}, a_{00}\in M_r(B)_{s.a.}.$ 
 For any $\ep>0,$ there is $k_{00}\ge 1$ such that, if $k\ge k_{00},$
 $$
(1-E_k)b_1\approx_{\ep} (1-E_k)P_1\andeqn E_k^{1/2}b_{00}\approx_\ep b_{00} \approx_\ep b_{00}E_k^{1/2}\approx_{\ep} E_k^{1/2}b_{00}E_k^{1/2}
 $$ 
 Thus, by also \eqref{LcomparisonU-10}, $E_k(P_1+b_{00})=E_k^{1/2}P_1E_k^{	1/2}+ E_kb_{00}\approx_{3\ep} E_k^{1/2}b_1E_k^{1/2}.$
 Therefore,  (with a similar consideration for $P_2+a_{00}$)
 %also  by \eqref{LcomparisonU-10},
 \beq\label{LcomparisonU-11}
&& \lim_{k\to\infty}\|(E_k^{1/2}b_1E_k^{1/2}+(1-E_k)^{1/2}P_1(1-E_k)^{1/2})-b_1\|=0\andeqn \\\label{LcomparisonU-12}
&& \lim_{k\to\infty}\|(E_k^{1/2}f_{\eta/2}(a)E_k^{1/2}+(1-E_k)^{1/2}P_2(1-E_k)^{1/2})-f_{\eta/2}(a)\|=0.
%\andeqn\\
%&&\lim_{k\to\infty}\|(E_kaE_k+(1-E_k)P_2(1-E_k))-a\|=0.
 \eneq
 Put $x_k:=E_k^{1/2}f_{\eta/2}(a)E_k^{1/2}+(1-E_k)^{1/2}P_2(1-E_k)^{1/2}$ and $y_k:=E_k^{1/2}b_1E_k^{1/2}+(1-E_k)^{1/2}P_1(1-E_k)^{1/2},$ $k=1,2,....$
Since $y_n\to b_1,$ we may also assume  (by Proposition 2.2 of \cite{Rr2})
that, for all $n\ge n_0,$ 
\beq
f_{\dt_1/8}(y_n)\lesssim b_1.
\eneq
Since, for any fixed $\dt_0>0,$ 
\beq
\lim_{k\to\infty}\|f_{\dt_0}(y_k)-f_{\dt_0}(b_1)\|=0,
\eneq
we may assume, \wilog,  for all $k\ge 1,$   $\pi_\C^B(f_{\dt_1/2}(y_k))=p_1=\pi_\C^B(f_{\dt_1/2}(b_1))$ and
\beq
\tau(f_{\dt_1/2}(y_k))\ge \tau(f_{\dt_1/2}(b_1))-d_0/64 \rforal  \tau\in T(B).
\eneq
It follows 
%from
by
 \eqref{Notef1-11+} (with $\dt=\dt_1/2$) that
\beq\label{Notef1-19}
\tau(f_{\dt_1/2}(y_k))>d_\tau(b_1)-\omega([b])-3d_0/64\rforal \tau\in T(B).
\eneq
Since $M_r(B)$ has continuous scale, there is $k_0\ge n_0$ such that
\beq\label{Notef1-20}
d_\tau(1-E_{n})\le \tau(1-E_{n-1})<{{d_0/64}} \rforal \tau\in T(B)\andeqn \rforal n\ge k_0.
\eneq
It follows that, for $k\ge k_0,$ 
\beq\label{Notef1-21}
&&\tau(f_{\dt_1/2}(y_k))\le d_\tau(y_k)\le d_\tau(c_k)+d_0/64\\
&&=d_\tau(b_1^{1/2}E_kb_1^{1/2})+d_0/64\le d_\tau(b_1)+d_0/64 \rforal \tau\in T(B).
\eneq
Let $g_{\dt_1}\in C_0((0,1])_+$ with $1\ge g(t)>0$ for all $t\in (0, \dt_1/4),$
$g_{\dt_1}(t)\ge t$ for $t\in (0, \dt_1/16),$ $g_{\dt_1}(t)=1$ for $t\in (\dt_1/16, \dt_1/8)$ 
and $g_{\dt_1}(t)=0$ if $t\ge \dt_1/4.$ 

Since $g_{\dt_1}(y_k)f_{\dt_1/2}(y_k)=0,$ by 
%combining 
\eqref{Notef1-21},
we conclude that, for $k\ge k_0,$ 
\beq\label{Notef1-22}
d_\tau(g_{\dt_1}(y_k))+\tau(f_{\dt_1/2}(y_k))\le d_\tau(y_k)\le d_\tau(b_1)+d_0/64\rforal \tau\in T(B).
\eneq
Then, by \eqref{Notef1-19} and \eqref{Ldecomp-n1},  for all $k\ge k_0,$ 
\beq
d_\tau(g_{\dt_1}(y_k))&\le&(d_\tau(b_1)-\tau(f_{\dt_1/2}(y_k)))+d_0/64\\\label{Notef1-22+}
&\le &  \omega([b])+3d_0/64+d_0/64=\omega([b])+d_0/16<d_0
\eneq
for all $ \tau\in T(B).$
%and for all $k\ge k_0.$ 
Moreover, since $\pi_\C^B(y_k)=p_1=\pi_\C^B(f_{\dt_1/2}(y_k))$ for 
all $k,$
\beq\label{Notef1-19+}
g_{\dt_1}(y_k)\in M_r(B).
\eneq
It should be noted and will be used later that, for any $0\le x\le 1,$ 
\beq\label{Notef1-5}
x\le f_{\dt}(x)+g_{\dt_1}(x)\rforal 0<\dt<\dt_1/8.
\eneq
Note, for all $k>n+1\ge n>n_0,$  that  $c_n\perp (1-E_k)^{1/2}P_1(1-E_k)^{1/2},$ 
%since $B$ has strict comparison,
 \beq
a_{k}\lesssim c_n\andeqn  (1-E_k)^{1/2}P_2(1-E_k)^{1/2}\lesssim (1-E_k)^{1/2}P_1(1-E_k)^{1/2}. 
%c_n\perp (1-E_k)^{1/2}P_1(1-E_k)^{1/2}.
 \eneq
 Put $y_{k,n}':=c_n+(1-E_k)^{1/2}P_1(1-E_k)^{1/2},$ $k=1,2,....$
 
  There exists a function $\chi\in C_0((0,\|a\|)$ with $0\le \chi\le 1$ such 
  that $\chi(f_{\eta/2})=f_\eta.$
 For any $\ep>0,$ there exists $\dt_2>0$ such that, if $0\le e_1,e_2\le 1$
 be elements in a \CA\, with $\|e_1-e_2\|<\dt_2,$ then 
 $\|\chi(e_1)-\chi(e_2)\|<\ep/32.$
 By Lemma \ref{Linvt},   for any fixed 
 $k\ge n+1>n\ge n_0,$
  there  are $\dt_k>0$  and a unitary 
  %exists a sequence of unitaries 
  $V\in M_{2r}(\td B)$ and 
  %a sequence of elements 
  $h_k\in {\rm Her}(f_{\dt_k}(y_{k,n}'))_+$  with $\|h_k\|\le 1$ 
  such that
  \beq\label{LcomparisonU-30}
 %V_m^*f_\ep(x_k)V_m\in M_r(\td B)\andeqn 
% \lim_{m\to\infty}
 \|V^*f_{\eta/2}(x_k)V-h_k\|<\min\{\ep/32, \dt_2\}.
  \eneq
   By \eqref{LcomparisonU-12}, choose $k_{m,1}\ge k_0$ such that, for all $k\ge k_{m,1},$ 
  \beq\label{LcomparisonU-31}
  \|\chi(f_{\eta/2}(x_k))-\chi(f_{\eta/2}(a))\|<\ep/32.
  \eneq
Thus we have
\beq\label{1130-eq-1}
 \|V^*f_{\eta}(a)V-\chi(h_k)\|<3\ep/32.
\eneq

%%%%%%%%%%%%%%%
\iffalse

Fix an $\eta>0.$ 
Then there exists $k_1\ge k_0+2$ such that, since $\lim_{k\to\infty}\|y_k-a\|=0,$
\beq\label{Notef1-26}
(a-\eta)_+\lesssim y_k=E_kaE_k+(1-E_k)P_2(1-E_k).
\eneq
Note that this holds regardless  of whether $\la a\ra$  represented by a projection or not.
Fix any $n\ge k_0\ge n_0,$.
%  by 
By
\eqref{Notef1-8},
\beq
d_\tau(E_kaE_k)&=&d_\tau(a^{1/2}E_k^2a^{1/2})\le d_\tau(a)
<d_\tau(c_n)\rforal \tau\in T(B)
\eneq
{{and}} for any $k.$ 
Since $B$ has strict comparison,
\beq
E_kaE_k{{\lesssim}} c_n
\eneq
for any $n\ge k_0$ and any $k.$
Choose $k\ge \max\{k_1, n\}+2.$ In particular, $E_n$ and $(1-E_k)$ are mutually 
orthogonal. 
Then  
\beq\label{Notef1-30}
(a-\eta)_+ &\lesssim  &y_k \lesssim  E_kaE_k+ (1-E_k)P_2(1-E_k)\\
                 &\lesssim & c_n+(1-E_k)P_2(1-E_k)   \le c_n+(1-E_k)P_1(1-E_k) \\
                 &=& c_n+P_1(1-E_k)^2P_1\le c_n +P_1(1-E_n)^2P_1\\
                 &=& c_n+(1-E_n)P_1(1-E_n)=x_n.
                  \eneq
     %
     %
     %
     %%%%%%%%%
     \fi
     %
     %

     %%%%%%%%%%%%%%%%%%%%%%%%%%%%             
                  
%In other words, 
%\beq\label{Notef1-31}
%(a-\eta)_+\le x_n\rforal n\ge k_0.
%\eneq  
Recall (see \eqref{LcomparisonU-10}), for  $n>n_0$ and $k\ge \max\{k_{m,1}, k_0, n+1\},$
\beq
&&\hspace{-0.4in}y_k'=E_n^{1/2}b_1E_n^{1/2}+(1-E_k)^{1/2}P_1(1-E_k)^{1/2}=E_n^{1/2}b_1E_n^{1/2}+P_1(1-E_k)P_1\\
&&\hspace{-0.2in}\le E_n^{1/2}b_1E_n^{1/2}+P_1(1-E_n)P_1=E_n^{1/2}b_1E_n^{1/2}+(1-E_n)^{1/2}P_1(1-E_n)^{1/2}=y_n.
\eneq
By \eqref{Notef1-5},
\beq\label{Notef1-33}
y_n\le f_{\dt_1/8}(y_n)+g_{\dt_1}(y_n):={\bar y}_n
%\,\, f_{\dt_1/{\red{8}}}(x_n)\lesssim b_1, 
%\la x_n\ra  &\le & \la f_{\dt_1/{\red{8}}}(x_n)+g_{\dt_1}(x_n)\ra \le \la f_{\dt_1/{\red{8}}}(x_n)\ra +\la g_{\dt_1}(x_n)\ra\\
%&\le& \la b_1\ra +\la g_{\dt_1}(x_n)\ra.
\eneq
Thus $h_k\in {\rm Her}({\bar y}_n).$  Choose $\ep'>0$ such 
that 
\beq\label{1130-eq2}
\|f_{\ep'}({\bar y}_n)\chi(h_k)f_{\ep'}({\bar y}_n)-\chi(h_k)\|<\ep/16.
\eneq
%Choose $n\ge k_0$ such that (note that, by \eqref{A200117-n1}, $y_n\to b_1$ as $n\to\infty$)
%\beq\label{Notef1-32}
%\eneq
%By \eqref{Notef1-5},
%\beq\label{Notef1-33}
%y_n\le f_{\dt_1/{\red{8}}}(y_n)+g_{\dt_1}(y_n), 
%\,\, f_{\dt_1/{\red{8}}}(x_n)\lesssim b_1, 
%\la x_n\ra  &\le & \la f_{\dt_1/{\red{8}}}(x_n)+g_{\dt_1}(x_n)\ra \le \la f_{\dt_1/{\red{8}}}(x_n)\ra +\la g_{\dt_1}(x_n)\ra\\
%&\le& \la b_1\ra +\la g_{\dt_1}(x_n)\ra.
%\eneq
By \eqref{Notef1-22+}  
%and \eqref{Ldecomp-n1} 
and the strict comparison of $B,$ 
%\beq\label{Notef1-34}
$g_{\dt_1}(y_n) \lesssim b_0.$
%\eneq
Recall $b_1\perp b_0$ and  $f_{\dt_1/8}(y_n)\lesssim  b_1.$
By applying Lemma \ref{Linvt} again, we obtain a unitary $W\in M_{2r}(\td B)$ and 
${\bar h}\in {\rm Her}(b_1+b_0)\subset {\rm Her}(b)_+$
such that
\beq
\|W^*f_{\ep'}({\bar y}_n)W-{\bar h}\|<\ep/8.
\eneq
Let $U=VW.$ Then, by \eqref{1130-eq-1}, \eqref{1130-eq2}, 
\beq
U^*f_\eta(a)U\approx_{3\ep/32} W^*\chi(h_k)W\approx_{\ep/16} W^*f_{\ep'}({\bar y}_n)\chi(h_k)f_{\ep'}({\bar y}_n)W\\
=W^*f_{\ep'}({\bar y}_n)WW^*\chi(h_k)WW^*f_{\ep'}({\bar y}_n)W\approx_{\ep/4} {\bar h}(W^*\chi(h_k)W){\bar h}.
\eneq
Note that ${\bar h}(W^*\chi(h_k)W){\bar h}\in {\rm Her}(b).$  This proves the first part of the lemma.
To see the last part, let $0<\sigma<1/2,$ the first part and Proposition of 2.2 of \cite{Rr2} imply that, for large $n,$ 
\beq
f_\sigma(U_n^*f_\eta(a)U_n)\lesssim b.
\eneq 
It follows that $f_\sigma(f_\eta(a))\sim U_n^*f_\sigma(f_\eta(a))U_n=f_\sigma(U_n^*f_\eta(a)U_n)\lesssim b$
for all $0<\sigma<1/2.$   Hence $f_\eta(a)\lesssim b$ (for all $0\le \eta<1/2$) which implies $a\lesssim b.$
\end{proof}

\begin{lem}\label{1202-1}
Let $B$ be as in \ref{151} and let $a\in M_r(\td B)_+$ with $0\le a\le 1.$
Then there exists a sequence $0\le a_n\le 1$ in ${\rm Her}(a)$ 
such that $[a_n]\le [a_{n+1}],$ $a=\sup\{[a_n]: n\in \N\}$ 
and $\lim_{n\to\infty}\omega([a_n])=0.$

\end{lem}
\begin{proof}
If there is a sequence $t_n\in (0,1)$ such that $t_{n+1}<t_n$ and $\lim_{n\to\infty}t_n=0$ 
and $t_n\not\in {\rm sp}(a)$ for all $n,$ then one obtains an increasing 
sequence of projections $\{p_n\}$ such that $p_n\in {\rm Her}(a),$ and  such that
for any $0<\ep<1,$ $f_{{\ep}}(a)\le p_n$ for all sufficiently large $n.$ Let $a_n:=p_n.$
Then $\omega([p_n])=0$ and $[a]=\sup\{[a_n]: n\in  \N\}.$ 
Thus we assume that $[0,\eta_0]\subset {\rm sp}(a)$ for some $\eta_0\in (0,1].$ 

As in the proof of  Lemma \ref{LLcomparisonU}, we may assume 
that, for some integer $m\ge 1,$ 
%$\pi_\C^B(a)=p$ is a projection of finite rank $m\ge 1.$ 
%Moreover, 
\beq
\pi_\C^B(a)=\diag(\overbrace{1,1,...,1}^m,0,\cdots, 0):=p.
\eneq
Let $P:=\diag(\overbrace{1_{\td B}, 1_{\td B},...,1_{\td B}}^m, 0,...,0).$
Let $\{e_n\}$ and $\{E_n\}$ be as in the proof  \ref{LLcomparisonU} of (see \eqref{48-eq1202-1}). 
Note that $E_kP=PE_k$ for all $k.$
As in the proof of \ref{LLcomparisonU}, if $0<\eta<\eta_0/16,$  then
\beq
\lim_{k\to\infty}\|(E_k^{1/2}aE_k^{1/2}+(1-E_k)^{1/2}P(1-E_k)^{1/2}-a\|=0\andeqn\\\label{Lomega-9}
\lim_{k\to\infty}\|(E_k^{1/2}f_{\eta}(a)E_k^{1/2}+(1-E_k)^{1/2}P(1-E_k)^{1/2}-f_{\eta}(a)\|=0.
\eneq
%Since $M_r(B)$ has continuous scale, $(1-E_k)^{\widehat{}}\searrow 0$ uniformly on $T(B).$
%
Note $\pi_\C^B((1-E_k)^{1/2}P(1-E_k)^{1/2})=p=\pi_\C^B(f_\eta(a)).$
Note also that, since $f_\eta(a)^{1/2}E_kf_\eta(a)^{1/2}\nearrow f_\eta(a)$ (in the strict topology),
$(E_k^{1/2}f_\eta(a)E_k^{1/2})^{\widehat{}}\nearrow  \widehat{f_\eta(a)}$ uniformly on $T(B)$ (by Dini's theorem). 
We may  therefore assume that, if $k\ge k_\eta$ (for some $k_\eta\ge 1$),
\beq\label{Lomega-20}
(E_k^{1/2}f_{\eta}(a)E_k^{1/2})^{\widehat{}}(\tau)>\widehat{f_{\eta}(a)}(\tau)-\sigma(\eta)/16 \rforal \tau\in T(B),
\eneq
 where $\sigma(\eta)=\min\{\inf \{\tau(f_{\eta}(a))-d_\tau(f_{4\eta}(a)):\tau\in T(B)\}, \eta/16\}>0$ 
 (recall $[0,\eta_0]\subset {\rm sp}(a)$).
 Moreover, since $M_r(B)$ has continuous scale, we may assume,  for all $k\ge k_\eta,$
 \beq\label{Lomega-20-1}
 {[1-E_{k+1}]}^{\widehat{}}\le (1-E_k)^{\widehat{}}<\sigma(\eta)/16.
 \eneq
 Put $a_{\eta, k}:=E_k^{1/2}f_\eta(a)E_k^{1/2}.$
 Choose $0<\dt(\eta)<\sigma(\eta)/16r.$  Then, by \eqref{Lomega-20}, for any $k\ge k_{\eta, 1}$
 for some $k_{\eta, 1}\ge k_\eta+1,$
 \beq\nonumber 
&&\hspace{-0.4in}\tau( f_{\dt(\eta)}(a_{\eta,k}))\ge \tau((a_{\eta,k}-\dt(\eta))_+)
>\tau(a_{\eta,k}-\dt(\eta))\\ \label{Lomega-21}
&&=\tau(a_{\eta,k})-r\dt(\eta)>
 \tau(f_{\eta}(a))-\sigma(\eta)/8\rforal \tau\in T(B).
 \eneq
By \eqref{Lomega-9}  and Proposition 2.2 of \cite{Rr2}, 
%for any fixed  $\dt(\eta)>0,$ 
there is $k_{\eta,2}\ge k_{\eta,1}$ such that, for all $k\ge k_{\eta,2},$ 
there is 
% $x_{\dt/2,\eta}, 
$x_{\dt/8,\eta}\in {\rm Her}(f_{\eta}(a))$ such that 
\beq\label{newnumber411}
%f_{\dt(\eta)/2}((1-E_k)P(1-E_k)+E_kf_{\eta}(a)E_k)\sim x_{\dt/2,\eta}  \andeqn\\
f_{\dt(\eta)/8}((1-E_k)^{1/2}P(1-E_k)^{1/2}+E_k^{1/2}f_{\eta}(a)E_k^{1/2})\sim x_{\dt/8,\eta}.
\eneq

 Since $B$ is stably projectionless,  for any  {{nonzero}} $0\le b\le 1$ in $M_r(B),$
 ${\rm sp}(b)=[0,1].$ Thus
 \beq\label{Lomega-8}
 d_\tau(f_{\dt(\eta)}(a_{\eta,k})))<\tau(f(a_{\eta,k}))<d_\tau(f_{\dt(\eta)/2}(a_{\eta, k}))\rforal \tau\in T(B),
 \eneq
 where $0\le f\le 1$ is in $C_0((0,1])$ such that
 $f(t)=1$ for $t\in [\dt(\eta)/2, 1],$  $ f(t)=0$ for $t\in (0, \dt(\eta)/4].$ 
 Since ${\rm Cu}(B)=\LAff_+(T(B)),$ there is 
  $c_{k,\eta,\dt(\eta)}\in M_r(B)_+$  with $\|c_{k, \eta, \dt(\eta)}\|\le 1$ 
 such that, for all $\tau\in T(B),$  $d_\tau(c_{n,\eta, \dt(\eta)})=\tau(f(a_{\eta,k}))$  which is continuous on $T(B).$ 
 Since $B$ has strict comparison, by \eqref{Lomega-8},  $c_{k, \eta, \dt(\eta)}\lesssim f_{\dt(\eta)/2}(a_{\eta, k}).$
 Since
  $M_r(B)$ has almost stable rank one, by  Lemma \ref{Pxxxx}, we may assume that 
 $c_{k, \eta, \dt(\eta)}\in {\rm Her}(f_{\dt(\eta)/2}(a_{\eta, k})).$
 
 Note that $(1-E_{k+1})E_k=0$ and $P(1-E_k)^{1/2}=(1-E_k)^{1/2}P$ for all $k.$ 
 In  particular, \\
 $P(1-E_{k+1})P=(1-E_{k+1})^{1/2}P(1-E_{k+1})^{1/2}\perp a_{\eta,k}.$
  By Lemma \ref{Lfolk},  there exists $z\in M_r(\td B)$ such that  
  (see also \eqref{newnumber411})
 \beq\nonumber
&&\hspace{-0.7in} f_{\dt(\eta)/2}((1-E_{k+1})^{1/2}P(1-E_{k+1})^{1/2})+c_{k, \eta, \dt(\eta)}\le  f_{\dt(\eta)/4}(P(1-E_{k+1})P) +f_{\dt(\eta)/4}(a_{\eta,k})\\
&&= f_{\dt(\eta)/4}(P(1-E_{k+1})P) +a_{\eta,k})
\sim  z^*z,\andeqn\\
&&z^*z\lesssim f_{\dt(\eta)/8}(P(1-E_k)P +a_{\eta,k})\\
%E_k^{1/2}f_{\eta}(a)E_k^{1/2})\\
&&= f_{\dt(\eta)/8}((1-E_k)^{1/2}P(1-E_k)^{1/2} +E_k^{1/2}f_{\eta}(a)E_k^{1/2})
\sim x_{\dt/8, \eta}\in {\rm Her}(f_\eta(a)).
 \eneq
 Define $b_{k, \eta, \dt(\eta)}:= f_{\dt(\eta)/2}((1-E_{k+1})^{1/2}P(1-E_{k+1})^{1/2})+c_{k, \eta, \dt(\eta)}$ for $k\ge k_{\eta, 2}.$
 From the displays above, there is $y_{k,\eta, \dt(\eta)}\in {\rm Her}(f_\eta(a))$ such that $b_{k, \eta, \dt}\sim y_{k, \eta, \dt(\eta)}.$
 By \eqref{Lomega-8} and \eqref{Lomega-21},  we have, for $k\ge k_{\eta,2},$ and for all $\tau\in  T(B),$ 
 \beq
&&\hspace{-0.4in}d_\tau(y_{k, \eta, \dt(\eta)})=d_\tau(b_{k, \eta, \dt(\eta)})>d_\tau( c_{k, \eta, \dt(\eta)})\\\label{Lomega-22}
&&\ge {[f_{\dt(\eta)}(E_k^{1/2}f_\eta(a)E_k^{1/2})]}^{\widehat{}}(\tau)\ge  \tau(f_{\dt(\eta)}(a_{\eta,k}))>
 \tau(f_{\eta}(a))-\sigma(\eta)/8.
 \eneq
 Since ${[c_{k,\eta, \dt(\eta)}]}^{\widehat{}}$ is continuous on $T(B),$ for $k\ge k_{\eta,\dt},$ 
 by \eqref{Lomega-20-1}\\ (recall $f_{\dt(\eta)/2}((1-E_{k+1})^{1/2}P(1-E_{k+1})^{1/2})\perp c_{k, \eta, \dt(\eta)}$),
 \beq\label{Lomega-23}
 \omega([y_{k, \eta, \dt(\eta)}])<\sigma(\eta)/16\le \eta/32.
 \eneq
 Combing (recall the definition of $\sigma(\eta)$) \eqref{Lomega-23} and \eqref{Lomega-22}, for all $\tau\in T(B),$
 \beq
 &&\hspace{-0.8in}d_\tau(f_{8\eta}(a)) +4\omega([y_{k, \eta, \dt(\eta)}])
 %\le \tau(f_{4\eta}(a))
 \le \tau(f_{4\eta}(a))+4\omega([y_{k, \eta, \dt(\eta)}])
 <\tau(f_\eta(a))-\sigma(\eta)+4\omega([y_{k, \eta, \dt(\eta)}])\\
 &&<\tau(f_\eta(a))-5\sigma(\eta)/16<d_\tau(y_{k, \eta, \dt(\eta)}).
 \eneq
 We also have 
 $
 [\pi_\C^B(f_{2\eta}(a))]\le [\pi_\C^B(f_\eta(a))]=p=[\pi_\C^B((1-E_{k_{\eta,2}})P)]=
 [\pi_\C^B(y_{k_{\eta,2}, \eta, \dt(\eta)})].
 $
 By Lemma \ref{LLcomparisonU},  
 \beq\label{Lomega-40}
 f_{8\eta}(a)\lesssim y_{k_{\eta,2}, \eta, \dt(\eta)}.
 \eneq
 For each fixed $0<\eta<1/8,$ there exists $\mu_\eta>0$ 
 such that (recall $[0, \eta_0]\subset {\rm sp}(a)$)
 \beq\label{Lomega-24}
 \tau(f_{\eta/4}(a))>d_\tau(f_\eta(a))+\mu_\eta\rforal \tau\in T(B).
 \eneq
 Choose $0<\eta'<\eta/16$ such that $\eta'<\mu_\eta/4.$ 
 %$0<\dt(\eta')<\sigma(\eta').$
 Then, for $k\ge k_{\eta',  2},$ and for all $\tau\in T(B),$  by \eqref{Lomega-22}, \eqref{Lomega-24}, 
 (and recall the definition of $\sigma(\eta')$ and $y_{k_{\eta, 2}, \eta, \dt(\eta)}\in {\rm Her}(f_{\eta}(a))$), and \eqref{Lomega-23},
 \beq
&&\hspace{-0.6in}{[y_{k_{\eta', 2}, \eta', \dt(\eta')}]}^{\widehat{}}(\tau)\ge \tau(f_{\eta'}(a))-\sigma(\eta')/8\ge d_\tau(f_{\eta}(a))+\mu_\eta/2
\ge {[y_{k_{\eta, 2}, \eta, \dt(\eta)}]}^{\widehat{}}(\tau)+\mu_\eta/2\\
&&>{[y_{k_{\eta, 2}, \eta, \dt(\eta)}]}^{\widehat{}}(\tau)+\eta'
\ge {[y_{k_{\eta, 2}, \eta, \dt(\eta)}]}^{\widehat{}}(\tau)+4\omega([y_{k_{\eta',2}, \eta', \dt(\eta')}])\rforal \tau\in T(B).
 \eneq
 Recall $\pi_\C^B(f_{\dt(\eta)}((1-E_k)^{1/2}P(1-E_k)^{1/2}))=p$ for all $k$ and 
 for all $\dt(\eta)<1/2.$
 It follows from Lemma \ref{LLcomparisonU} (or from Lemma A.3 of \cite{eglnkk0})
 \beq
 f_{8\eta}(a)\lesssim y_{k_{\eta, 2}, \eta, \dt(\eta)}\lesssim y_{k_{\eta', 2}, \eta', \dt(\eta')}.
 \eneq
Thus, we obtain a sequence $\{c_n\}$ which is a subsequence of $\{ y_{k_{\eta, 2}, \eta, \dt(\eta)}\}\in {\rm Her}(a)$
(with $\eta\to 0$)
such that 
%$\widehat{[c_n]}$ is continuous on $T(B),$ 
\beq
[c_n]\le [c_{n+1}]\andeqn \lim_{n\to\infty}\omega([c_n])=0\,{\rm(see} \,\,\eqref{Lomega-23}{\rm).}
\eneq
Put $x=\sup\{[c_n]: n\in \N\}$ (see Theorem 1 of \cite{CEI}). Then, by \eqref{Lomega-40},
$[f_{8\eta}(a)]\le x$ for all $0<\eta<\min\{\eta_0, 1/16\}.$   It follows that $[a]\le x.$
 Since each $c_n\in {\rm Her}(a),$
 $x\le [a].$  It follows that $x=[a].$
 \end{proof}

%
%%%%%%%%%%%%%%%%%%%%%%%%%%%%

%

\iffalse
%%%%%%%%%%%%%%%%
%
%
%
 \begin{lem}\label{Lomegaapp}
 Let $B$ be a separable simple \CA\, with nonempty compact $T(B)$ and 
 let $a\in M_m(\td B)_+$ with $0\le a\le 1$ (for some integer $m\ge 1$).
 %
 Then, for any $1/2>\ep>0,$ 
 %and $\sigma>0,$ 
 there is $1/2>\dt_0>0$ such that, for all $0<\dt_1<\dt\le \dt_0,$ 
 \beq
 \omega([f_\dt(a]))<\omega([a])+\ep\tand d_\tau(f_{\dt_1}(a))-\tau(f_{\dt}(a))<\omega([a])+\ep
 \tforal \tau\in T(B).
 \eneq
 
% (2):  If $c_n=E_naE_n\le 1,$   where $\{E_n\}$ is an approximate identity 
% for $M_m(B),$ 
 %are elements in $C:=\rm Her}(a)\cap M_n(B)$ such
% that 
 %$\lim_{n\to\infty}\|c_n-e\|=0,$  where $e_C$ is a strictly positive 
% then, for any $\ep>0,$ there exists an integer $n_0\ge 1$ such that, for all $n\ge n_0,$
% \beq
% d_\tau(c_n)>d_\tau(a)-\omega([a])-\ep\rforal \tau\in T(B).
% \eneq
   \end{lem}

\begin{proof}
Since $\{f_{1/n}(a)\}$ forms an approximate identity of ${\rm Her}(a),$   there is an integer $N\ge 1$ such that, if $n\ge N,$
\beq\label{Lomegaapp-1}
0<d_\tau(a)-\tau(f_{1/n}(a))<\omega([a])+\ep/2\rforal \tau\in T(B).
\eneq
It follows that
\beq
0<d_\tau(f_{1/n}(a))-\tau(f_{1/n}(a))\le d_\tau(a)-\tau(f_{1/n}(a))<\omega([a])+\ep/2\rforal \tau\in T(B).
\eneq
Choose $\dt_0=1/N.$ Then, for any $0<\dt\le \dt_0,$ 
%\beq
$\omega([f_\dt(a)])<\omega([a])+\ep.$
%\eneq
By the  virtue  of \eqref{Lomegaapp-1}, we also have, if $0<\dt_1<\dt\le \dt_0,$
\beq\nonumber
0\le d_\tau(f_{\dt_1}(a))-\tau(f_{\dt}(a))<d_\tau(a)-\tau(f_\dt(a))<\omega([a])+\ep.
\eneq
\end{proof}
\fi%%%%%%%%%%%%%%%%%%%%%%%%%%%%%%%%%%%%%

 In the following statement, it should be noted that
 % warned that we do {\it not} assume  
% that $\td B$ has strict comparison, 
 we  do not assume that $\td B$ has almost stable rank one.  
 %Recall that if $a\in M_r(\td B)_+$ such that $\widehat{a}$ is continuous on $T(B),$   then
 % $[a]\in S(\td B)$ in the following statement.
 One of the features of the following statement is the existence of the unitaries  $U_n$ 
 %in the statement
 which compensates  the absence of the cancellation for our late purposes.

   \begin{thm}\label{LcomparisonU}
 Let $B$ be as in \ref{151} and 
 $a, b\in M_r(\td B)_+$ (where $r\ge 1$ is an integer). 
 %such that 
 %$\pi_\C^{B}(b)$ has finite rank  and 
 %$\widehat{[b]}$ is continuous on $T(B).$
 Suppose that $\pi_\C^B(a)\lesssim \pi_\C^B(b),$ and 
 \beq
d_\tau(a) <d_\tau(b)\rforal \tau\in T(B).
 \eneq
% and  $[b]\in S(\td B).$ 
 Then, for any $1>\eta>0,$ there exists a sequence of unitaries 
 $U_n\in M_{2r}(\td B)$  and a sequence of elements $h_n\in {\rm Her}(b)_+$
 such that
 \beq
 %U_n^*f_\eta(a)U_n\in M_r(\td B)\tand 
 \lim_{n\to\infty}\|U_n^*f_\eta(a)U_n-h_n\|=0.
 \eneq
  \end{thm}

 \begin{proof}
First consider the case that $[a]=[p]$ for some projection $p\in M_r(\td B).$ 
Then $d_\tau(a)=\tau(p)$ is continuous on $T(B).$ 
Put
\beq
\sigma:=(1/2)\inf\{d_\tau(b)-\tau(p):\tau\in T(B)\}>0.
\eneq
Since $\tau(f_{1/2^n}(b))\nearrow d_\tau(b),$ as $n\to\infty,$ there exists $n_0\ge 1$ such that,
for all $n\ge n_0,$ 
\beq
d_\tau(a)+\sigma<\tau(f_{1/2^n}(b))\rforal \tau\in T(B)\andeqn [\pi_\C^B(f_{1/2^n}(b))]=[\pi_\C^B(b)].
\eneq
%Since $[b]\in S(\td B),$  by Proposition \ref{Pbadingood}, 
By Lemma \ref{1202-1}, 
%It follows from Lemma \ref{Lapproxcont} that 
there exists a sequence of elements $b_n\in {\rm Her}(b)_+$ with $0\le b_n\le 1$  and an integer $N\ge 1$ 
such that, for all $n\ge N$ (as  $[f_{1/2^{n_0+1}}(b)]\ll  [b]$),
\beq
 f_{1/2^{n_0+1}}(b)\lesssim b_n  \andeqn \lim_{n\to\infty}\omega([b_n])=0.
\eneq
Thus,  there exists $n_1\ge N+n_0,$ for all $n\ge n_1$
\beq\label{LcomparisonU-50}
d_\tau(a)+4\omega([b_n])<d_\tau(b_n)\rforal \tau\in T(B)\andeqn [\pi_\C^B(a)]\le [\pi_\C^B(b_n)].
\eneq
Applying Lemma \ref{LLcomparisonU}, for any $\eta>0,$ there exist
a  sequence of unitaries $U_n\in M_{2r}(\td B)$ and a sequence of elements 
$h_n\in  {\rm Her}(b_{n_1})_+$ such that
\beq
%U_n^*f_\eta(a)U_n\in M_r(\td B)\andeqn 
\lim_{n\to\infty}\|U_n^*f_\eta(a)U_n-h_n\|=0.
\eneq
Note that $h_n\in {\rm Her}(b_{n_1})_+\subset {\rm Her}(b)_+.$

Next consider the case that $[a]$ cannot be represented by a projection. 
It follows that $0$ is not an isolated point. 

Fix $0<\eta<1.$ Choose $0<\ep<\eta/4,$ by Lemma 
\ref{Lorthog},  there exists $\sigma_0>0$ such that
\beq
d_\tau(f_{\ep/2}(a))+\sigma_0<d_\tau(f_{\eta/4}(a))<d_\tau(b)\rforal \tau\in T(B).
\eneq
Choose $b_n\in {\rm Her}(b)_+$   above.
Then, there exists $n_2\ge 1$ such that, for all $n\ge n_2,$ 
\beq\label{LcomparisonU-50+}
d_\tau(f_{\ep/2}(a))+4\omega([b_n])<d_\tau(b_n)\andeqn [\pi_\C^B(f_{\ep/2}(a))]\le [\pi_\C^B(b_n)].
\eneq
Applying Lemma \ref{LLcomparisonU}, one obtains a sequence of unitaries $U_n\in M_{2r}(\td B)$
and a sequence of elements $h_n\in {\rm Her}(b_{n_2})_+\subset {\rm Her}(b)_+$ such
that
\beq
%U_n^*f_\eta(a)U_n\in M_r(\td B)\andeqn 
\lim_{n\to\infty}\|U_n^*f_\eta(a)U_n-h_n\|=0.
\eneq
Theorem follows.
\end{proof}

We now arrive at the following theorem (see Theorem A.6 of \cite{eglnkk0}).
%Note that, if $A$ is a separable exact simple ${\cal Z}$-stable \CA, 
%then there is a hereditary \SCA\, $B\subset B$ (see ) such that
%$B$ has continuous scale. Note also $B\otimes {\cal K}\cong A\otimes {\cal K}.$ 
%The following also implies that 

 \begin{thm}\label{TBtildC}
 Let $B$ be as in \ref{151}. Then, for any $a, b\in (\td B\otimes {\cal K})_+,$
 if   $[\pi_\C^B(a)]\le [\pi_\C^B(b)]$ and 
 \beq
 d_\tau(a)<d_\tau(b)\rforal \tau\in T(B),
 \eneq
 then 
 $
 a\lesssim b.
 $
 Moreover, if $[a]$ is not represented by a projection, then 
 $d_\tau(a)\le d_\tau(b)$ for all $\tau\in T(\td B)$ implies that $a\lesssim b.$
 
 \end{thm}
 
 \begin{proof}
 For the first part, we note that,
 for any $\ep>0,$ there exists $\dt>0$ such 
 that 
 \beq
 d_\tau(f_\ep(a))<d_\tau(f_\dt(b))\rforal \tau\in T(B)\andeqn [\pi_\C^B(f_\ep(a))]\le [\pi_\C^B(f_\dt(b))].
 \eneq
 With this observation, we reduce the general case to the case 
 that $a, b\in M_r(\td B)_+$ with $0\le a, \, b\le 1.$   
 
 For this case, for any $0<\eta<1/2,$ by Lemma \ref{LcomparisonU}, there is $h\in {\rm Her}(b)_+$ 
and a unitary $U\in M_{2r}(\td B)$ such that
\beq
\|U^*f_{\eta/4}(a)U-h\|<\eta/8.
\eneq
By Proposition 2.2 of \cite{Rr2}, this implies that 
\beq
f_{\eta/4}(f_{\eta/4}(a))\sim U^*f_{\eta/4}((f_{\eta/4}(a))U=f_{\eta/4}(U^*f_{\eta/4}(a)U)\lesssim h\lesssim b.
\eneq
Since this holds for all $0<\eta<1/2,$  one has $a\lesssim b.$

% This case follows from \ref{CCCC}.
 
 %The proof of Theorem \ref{LcomparisonU} shows that (see either \eqref{LcomparisonU-50}
% or \eqref{LcomparisonU-50+}),
% for any $\eta>0,$ 
%there is $b_0\in {\rm Her}(b)$ with $0\le b_0\le 1$ such that 
%\beq
%d_\tau(f_{\eta}(a))+4\omega([b_0])<d_\tau(b_0)\rforal \tau\in T(B),\andeqn [\pi_\C^B(f_\eta(a))]\le [\pi_\C^B(b_0)].
%\eneq
%Applying Lemma A.3 of \cite{eglnkk0}, one concludes 
%\beq
%f_\eta(a)\lesssim b_0\lesssim b
%\eneq
%which implies that $a\lesssim b.$

Now suppose that
\beq
d_\tau(a)\le d_\tau(b)\rforal \tau\in T(\td B).
\eneq
If $[a]$ is not represented by a projection, then, by Lemma \ref{Lorthog}, 
for any $1>\ep>0,$ 
\beq
d_\tau(f_\ep(a))<d_\tau(b)\rforal \tau\in T(B)\andeqn [\pi_\C^B(f_\ep(a))]\le [\pi_\C^B(b)].
\eneq
By what has been proved above, 
$f_\ep(a)\lesssim b$ for all $1>\ep>0.$ Therefore $a\lesssim b.$

 \end{proof}
 
 Combining Theorem \ref{TBtildC} and Lemma \ref{TcomparisonintdA}, we have the following description 
 of the ${\rm Cu}(\td B).$  Note that all {{finite}} exact separable simple stably projectionless ${\cal Z}$-stable 
 \CA s satisfy the assumption of the corollary below.
 
 \begin{cor}\label{CtdBcomp}
 Let $B$ be a  separable stably projectionless simple \CA\, with continuous scale  such that $M_n(A)$ has almost stable rank one
 (for all $n\in \N$), and  such that $QT(B)=T(B)$ and ${\rm Cu}(B)=\LAff_+(T(B)).$ 
 Then, ${\rm Cu}(\td B)=(V(\td B)\setminus \{0\})\sqcup \LAff_+(T(\td B))^\diamond.$
  \end{cor}

 \begin{rem}\label{Rcomparison}
 If both $x$ and $y$ are not compact in ${\rm Cu}(\td B)$ and 
 $x\circeq y,$ or equivalently $x+k[1_{\td B}]=y+k[1_{\td B}]$ in ${\rm Cu}(\td B)$ for some integer $k,$ 
 then, by Theorem \ref{TBtildC}, $x=y.$ So ${\rm Cu}(\td B)$ has the weak version of cancellation. 
 However, we still do not have the cancellation for projections. In other words, if 
 $p\oplus e\sim q\oplus e$ for some nonzero projection $e,$ we do not know that $p\sim q.$ 
 Nevertheless, if $[p\oplus e]+x\le [q\oplus e]$ for some $x\in {\rm Cu}(\td A)_+\setminus \{0\},$ then $[p]\le [q],$ by Theorem \ref{TBtildC}.
 
 \end{rem}

\section{Approximation}

In this section we will present Lemma \ref{Llimintmaps} (see also the last part of Remark \ref{Rcountexample}).

 \begin{df}\label{DappCu}
 Let $A$ and $B$ be \CA s and $\lambda: {\rm Cu}^\sim(A)\to {\rm Cu}^\sim (B)$ be a morphism in ${\bf Cu}$
 (see \cite{Rl}).
 Suppose that $\phi_n: A\to B$ is a sequence of \hm s.
 %Denote by 
 %${\rm Cu}_f^\sim(C)=\{\bar{[a]}-n[1_{\td A}]: [\pi_\C^{\td A}(a)]=n[1_{\td A}]\andeqn a\in M_m\}.$
 We say ${\rm Cu}(\phi_n)$ converges to $\lambda$ and 
 write $\lim_{n\to\infty} {\rm Cu}(\phi_n)=\lambda,$ if, for any finite subset $G\subset {\rm Cu}^\sim(A),$ 
 there exists $N\ge 1$ such that, for all $n\ge N,$ 
 \beq
 {\rm Cu}^\sim(\phi_n)(x)\le \lambda(y)\andeqn \lambda(x)\le {\rm Cu}^\sim(\phi_n)(y),
 \eneq
 whenever $x, y\in G$ and $x\ll y.$
 
 Let $G_0\subset  K_0(A)\subset {\rm Cu}^\sim(A)$ (see 6.1 of \cite{RS}) be a finite subset.  Then $\lim_{n\to\infty}{\rm Cu}^\sim(\phi_n)=\lambda$ 
 implies that, there is an integer $n_0\ge 1$ such that, for all $n\ge n_0,$
 \beq
 {\rm Cu}^\sim(\phi_n)(x)=\lambda(x)\rforal x\in G_0
 \eneq
 as $x\ll x$  in ${\rm Cu}^\sim(A).$
 
 We write $\lim_{n\to \infty}^w{\rm Cu}^\sim(\phi_n)=\lambda,$  if 
 for any finite subset $G\subset {\rm Cu}^\sim(A),$ there exists $n_0\ge 1$ such that, for all $n\ge n_0,$ 
 \beq
  {\rm Cu}^\sim(\phi_n)(z)=\lambda(z)\rforal z\in G\cap K_0(A)\andeqn\\
 {\rm Cu}^\sim(\phi_n)(x)\le \lambda(y)\andeqn \lambda(x)\le {\rm Cu}^\sim(\phi_n)(y),
 \eneq
 whenever $x, y\in G$ and $x\ll y$ and both $x$ and $y$ are not compact. 
\end{df}

%\begin{prop}\label{Ptraceapp}
%Let $B$ be a separable simple \CA\, with continuous scale such that ${\rm Cu}(B)=V(B)\sqcup \LAff_+(T(B)).$
%Let $C$ be a separable \CA\, and let $\lambda: {\rm Cu}^\sim(C)\to {\rm Cu}^\sim(B)$ be 
%a morphism in ${\bf Cu}.$  Suppose that $\phi_k: C\to B$ is a sequence of \hm s 
%such that $\lim_{n\to\infty}^w{\rm Cu}^\sim (\phi_k)=\lambda.$
%Then, for any finite subset ${\cal F}\subset A_+,$
% \beq\label{eqWcu-1}
% \lim_{n\to\infty}\sup\{|\widehat{[\phi_n(a)]}(\tau)-\widehat{\lambda([a])}(\tau)|:\tau\in T(B)\}=0
 %\eneq
 %\end{prop}
 
% \begin{proof}
% Claim, if $
% \end{proof}

\begin{lem}\label{LappCum}
Let $C$ be a separable \CA\, of stable rank one
%  with a strictly positive element $e_A$ 
and $B$ be 
a \CA\, with finite  stable rank.
%as in  \ref{151}.
Suppose that $\lambda: {\rm Cu}^\sim(C)\to {\rm Cu}^\sim(B)$ is a morphism in ${\bf Cu}$ and 
there exists  a sequence of \hm s $\phi_n: C\to B$ such 
that 
\beq
\lim_{n\to\infty}{\rm Cu}^\sim(\phi_n)=\lambda.
\eneq
Suppose that 
%$\{\ep_m\}$ be a sequence of decreasing positive numbers in $(0,1)$ such that 
%$\lim_{m\to\infty}\ep_m=0$ and 
 $\psi_n: C\to B$  is a sequence of \hm s  such that 
%for some 
%hereditary \SCA\, $B_0$ such that, for any $C_m={\rm Her}(f_{\ep_m}(e_C)),$ 
\beq\label{LappCum-1}
\lim_{n\to\infty}\|\psi_n(a)-\phi_n(a)\|=0\tforal a\in C.
\eneq
Then 
\beq
\lim_{n\to\infty}{\rm Cu}^\sim(\psi_n)=\lambda.
\eneq
\end{lem}
  
 \begin{proof}
 Let $G\subset {\rm Cu}^\sim(C)$ be a finite subset. 
 Let $S=\{(f, g): f, g\in G, f\ll g\}.$ 
 
 Suppose that $(f,g)\in S.$ 
 %such that $g$ is not compact, i.e., it is not an element in $K_0(C)$
 %(see   Theorem 6.1 of \cite{RS}).   
 We claim, in this case, that
 there is $h\in {\rm Cu}^\sim(C)$ such 
 that
 \beq\label{59}
 f\ll h \ll g.
 %\andeqn h\not=g. 
 \eneq
 Recall that $C$ has stable rank one.
 % (see, for example, 3.5 of \cite{GLN}). 
 We may assume that $f=[a^f]-m_f[1_{\td C}]$ 
 and $g=[a^g]-m_g[1_{\td C}],$ where $a^f, a^g\in M_r((\td C))_+$ 
 with $\|a^f\|\le 1$ and $\|a^g\|\le 1$ for 
 some integer $r\ge 1,$ and, rank of $\pi_\C^{C}(a^f)$ is $m_f\le r,$
 %for all $(f,g)\in S.$
 %<\infty$
 and rank of $\pi_\C^{C}(a^g)$ is $m_g\le r.$    
 %Note that $m_1, m_2\le r.$
 Therefore
 \beq\label{Lunithm-10}
[a^f\oplus 1_{m_g}]\ll [a^g \oplus 1_{m_f}]
\eneq
(in the ${\rm Cu}(\td C)$),
where  $1_{m_f}$ and $1_{m_g}$ are identities of $M_{m_f}(\td C)$ and $M_{m_g}(\td C)$ respectively. 

By  \eqref{Lunithm-10},  there is $1/2>\ep>0$ such that 
%$f_{\ep}(a^g)\oplus 1_{m_f}\ge a^f\oplus 1_{m_g}.$ 
\beq
[a^f\oplus 1_{m_g}]\ll  [f_{\ep}(a^g)\oplus 1_{m_f}]\andeqn
%\eneq
%It is well known 
%\beq
[f_{\ep}(a^g)]\ll [a^g].
\eneq
 %Note that $0$ is not an isolated point in $sp(a).$ So $[f_{\ep}(a)]\not=[a].$ 
 Moreover, by choosing smaller $\ep,$ 
 we may assume that $\pi_\C^{C}(f_\ep(a^g))=f_{\ep}(\pi_\C^{C}(a^g))$ has the same rank as 
 that of $[\pi_\C^{ C}(a^g)]=m_g.$   Put $a^h=f_\ep(a^g)$ and $h=[a^h]-m_g[1_{\td C}].$
 Then 
 \beq\label{512}
 f\ll h\ll g.
 \eneq
 
 Define ${\bar f}=\diag(a^f, 1_{m_g}),$ ${\bar h}=\diag(a^h, 1_{m_f})$ and 
 ${\bar g}=\diag(a^g, 1_{m_f}).$ Note that, in ${\rm Cu}(\td {{C}}),$
 \beq
 {\bar f}\ll {\bar h}\ll {\bar g}.
 \eneq
 Choose $0<\dt<\ep/4$ such that 
 \beq\label{LappCum-5}
 {\bar f}\le f_{\dt}({\bar h})\le {\bar h}\andeqn {\bar h}\le f_\dt({\bar g})\le {\bar g}.
 \eneq

 Let $\phi_n^\sim, \psi_n^\sim: M_r(\td C)\to M_r(\td B)$ be the (unital) extensions of $\phi_n$ and $\psi_n,$
 respectively.
% Since $\lim_{n\to\infty}\|\psi_n(c)-\phi_n(c)\|=0$ for all $c\in C,$
 We claim that, for each $(f,g)\in S,$ there is an integer $N\ge 1$ such that, for all $n\ge N,$
 \beq
 {\rm Cu}^\sim(\psi_n)(f)\le \lambda(g)\andeqn \lambda(f)\le {\rm Cu}^\sim (\psi_n).
 \eneq
 
 Write $\lambda(f)=\lambda(f)_+-m_{\lambda,f}[1_{\td B}]$ and 
 $\lambda(g)=\lambda(g)_+-m_{\lambda, g}[1_{\td B}],$ where $\lambda(f)_+=[a_{\lambda,f}]$
 and $\lambda(g)_+=[a_{\lambda,g}]$ 
 for some $a_{\lambda, f}, a_{\lambda,g}\in M_r(\td B)_+$ (by enlarge $r$ if necessary),
 and $[\pi_\C^B(a_{\lambda, f})]=m_{\lambda, f}$ and $[\pi_\C^B(a_{\lambda, g})]=m_{\lambda, g}$ are 
 integers.

 Note, by \eqref{LappCum-1}, we have
 \beq
 \lim_{n\to\infty}\|\psi_n^\sim(c)-\phi_n^\sim(c)\|=0\rforal c\in M_r(\td C).
 \eneq
Then, by \eqref{LappCum-5} and by repeated application of Proposition 2.2 of \cite{Rr2}, there exists 
an integer $N\ge 1$ such that, if ${{n}}\ge N,$
\beq\label{LappCum-11}
\psi_n^\sim({\bar f})\lesssim \psi_n^\sim(f_{\dt}({\bar h}))\lesssim \phi_n^\sim({\bar h})\andeqn
 \phi_n^\sim ({\bar h})\lesssim \phi_n^\sim (f_{\dt}({\bar g}))\lesssim \psi_n^\sim({\bar g}).
 \eneq
 Assume that $B$ has stable rank $K.$
 Since $\lim_{n\to\infty}{\rm Cu}^\sim (\phi_n)=\lambda,$ we may also assume, if $n\ge N,$
 %for all $n\ge N,$
 \beq\label{LappCum-12}
&&{[}\phi_n^\sim(a^h){]} +(m_{\lambda, g}+K)[1_{\td B}]\le \lambda(g)_++(m_g+K)[1_{\td B}]\andeqn\\\label{LappCum-12+}
&&\lambda(f)_++(m_g+K)[1_{\td B}]\le [\phi_n^\sim(a^h)]+(m_{\lambda, f}+K)[1_{\td B}]
%
%a\lesssim  \psi_n(f_{\dt_1}(a^{g'}))\le \psi(f_{\dt_1}(a^{g'})
%\lesssim h_m(f_{\dt_2}(a^{g'})\le \psi(a^{g})
\eneq
%for all $n\ge N,$  and, 
for all $(f, g)\in S.$  Combining \eqref{LappCum-11}, \eqref{LappCum-12} and \eqref{LappCum-12+}, we obtain
\beq\nonumber
&&\hspace{-0.3in}[\psi_n^\sim(a^f)]+(m_g+m_{\lambda, g}+K)[1_{\td B}]=[\psi_n^\sim({\bar f})]+(m_{\lambda, g}+K)[1_{\td B}]\le
 [\psi_n^\sim(f_{\dt}({\bar h})]+(m_{\lambda, g}+K)[1_{\td B}]\\\nonumber
 &&\le {[}\phi_n^\sim(a^h){]}
+(m_f+m_{\lambda,g}+K)[1_{\td B}]\le \lambda(g)_++(m_g+K)[1_{\td B}]+m_f[1_{\td B}]\andeqn\\\nonumber
&&\hspace{-0.3in}\lambda(f)_+ +(m_f+m_g+K)[1_{\td B}]\le [\phi_n^\sim(a^h)]+(m_{\lambda,f}+m_f+K)[1_{\td B}]=[\phi_n^\sim({\bar h})]+(m_{\lambda,f}+K)[1_{\td B}]\\\nonumber
&&\le [\psi_n^\sim({\bar g})]+(m_{\lambda, f}+K)[1_{\td B}]
=[\psi_n^\sim(a^g)]+(m_f+m_{\lambda, f}+K)[1_{\td B}].
\eneq
Thus, for all $n\ge N,$  and, for all $(f,g)\in S,$ 
\beq\nonumber
{\rm Cu}^\sim(\psi_n)(f)\le \lambda(g)\andeqn \lambda(f)\le {\rm Cu}^\sim(\psi_n)(g).
\eneq
%for all $(f,g)\in S.$
 %It follows that $[f_{\ep}(a)\oplus 1_{m_2}]\not=[a]+m_2.$
%
 \end{proof}

 \begin{lem}\label{Llimintmaps}
 Let $C$ be a separable semiprojective \CA\,
 with a strictly positive element $e_C$ and  $B$ be as in \ref{151}.
 %and 
 
 (1) Let  
 $\lambda: {\rm Cu}^\sim(C)\to {\rm Cu}^\sim(\td B)$
 be  a morphism in ${\bf Cu}$ with $\lambda([e_C])\le [b]$  for some $b\in M_N(\td B)_+$ (and $N\ge 1$), 
 %and  for some integer $N\ge 1,$
 %and $b_0\in M_N(\td B)_+$ with $[b_0]\in S(\td B)$ and $[b_0]\circeq [b],$ 
 %such that $\hat{[b]}$ is continuous 
 %on $T(B),$ 
  and let $\phi_k: C\to M_N(\td B)$
 be a sequence of \hm s such that $\lim_{k\to\infty}{\rm Cu}^\sim(\phi_k)=\lambda,$
 then 
 exists a sequence of \hm s $\psi_k: C\to \overline{bM_N(\td B)b}$
  %\overline{b_0M_N(\td B)b_0}$ 
  such 
 that
 \beq
 \lim_{k\to\infty}{\rm Cu}^\sim(\psi_k)=\lambda,
 \eneq
if, in  addition, 

(i) $\lambda([e_C])^{\widehat{}}(\tau)< \widehat{[b]}(\tau)$ for all $\tau\in T(B),$  or
 
(ii) $\lambda([e_C])$ is not a compact element in ${\rm Cu}^\sim(\td B).$ 
% (i)  if there exists a subsequence $\{n_k\}$ such that, $[\phi_{n_k}(e_C)]$ is compact in 
 %${\rm Cu}^\sim(\td B),$ then there exists a sequence of \hm s 
% $\psi_k: C\to M_N(\td B)$ such that $\psi_k(e_C)\lesssim b$ and 
% $\lim_{k\to\infty}{\rm Cu}^\sim(\psi_k)=\lambda;$ 
 \iffalse
 %
 %
 (i)  If $\lambda([e_C])^{\widehat{}}(\tau)< \widehat{[b]}(\tau)$ for all $\tau\in T(B),$ 
 %and ${\rm Cu}^\sim(\pi_\C^B)\circ \lambda([e_C])\le [\pi_\C^B(b)],$ 
  then exists a sequence of \hm s $\psi_k: C\to \overline{bM_N(\td B)b}$
  %\overline{b_0M_N(\td B)b_0}$ 
  such 
 that
 \beq
 \lim_{k\to\infty}{\rm Cu}^\sim(\psi_k)=\lambda.
 \eneq
 
 (ii) If $\lambda([e_C])$ is not a compact element in ${\rm Cu}^\sim(\td B),$ 
 then conclusion of (i) also holds;
 %
 \fi
 %%%%%%%%%%%
 %(iii) if there exists a subsequence $\{n_k\}$ such that $[\phi_{n_k}(e_C)]$ is not a compact element,
% then the conclusion of (ii) also holds;

(2) If $\lambda: {\rm Cu}^\sim(C)\to {\rm Cu}^\sim(B)$ is a morphism in ${\bf Cu},$ $\lambda([e_C])\le [b]$
for some $b\in M_N(B)_+,$ and there exists a sequence of \hm s 
$\phi_k: C
\to M_N(B)$ such that 
$\lim_{k\to\infty}{\rm Cu}^\sim(\phi_k)=\lambda,$ then 
there exists a sequence of \hm s $\psi_k: C\to \overline{bM_N(B)b}$ such that 
 \beq
 \lim_{n\to\infty}{\rm Cu}^\sim(\psi_n)=\lambda.
 \eneq

 \end{lem}
 
 \begin{proof}
% If $\phi(e_C)\sim p$  for some projection $p,$ 
%then the condition that $\phi(e_C)\lesssim b$ implies that $p\lesssim b.$ Then 
% there is a partial isometry $v\in \td B$ such that $v^*\phi(e_C)v\in {\rm Her}(b).$
% Define $\phi^\sim: C\to \td B$  by $\phi^\sim|_C=\phi$ and $\phi^\sim(1_{\td C})=p.$ 
% Define  $\psi: \td C\to {\rm Her}(b)$ by $v^*\phi^\sim (c)v$ for $c\in C$ and $\psi(1_{\td C})=v^*pv.$ 
 %Then 
 %${\rm Cu}^\sim(\psi|_C)={\rm Cu}^\sim(\phi).$ Define $\phi_n=\psi$ for all $n.$ 
 
 Let us consider case (1) first.
 For any $\ep>0,$ there exists $k(\ep)\ge 1$ such that 
 $[\phi_k(f_{\ep/4}(e_C))]\le \lambda([f_{\ep/16}(e_C)])\le \lambda([e_C])$ for all $k\ge k(\ep).$  Put $a(k,\ep):=\phi_k(f_{\ep/4}(e_C)).$ 
 For case (i), we have 
 \beq
 d_\tau(a(k,\ep))<d_\tau(b)\rforal \tau\in T(B).
 \eneq
 For case (ii), let $e\in (\td B\otimes {\cal K})_+$ be such that $[e]=\lambda([e_C]).$ 
Then $e\not\sim p$ for any projection.  In other words, we may assume 
that $0$ is not an isolated point in ${\rm sp}(e).$  
Moreover, since $\lambda$ is a morphism in ${\bf Cu},$ it maps compact elements to compact elements.
Hence $[e_C]$ cannot be represented by a projection.  It follows 
that 0 is not an isolated point in ${\rm sp}(e_C).$
Choose $\eta>0$ such that
\beq
[\lambda(f_{\ep/16}(e_C))]\le [f_\eta(e)].
\eneq
For any $\eta>0,$ 
there is a nonzero element $c\in  {\rm Her}(e)_+$ such that $c\perp f_{\eta}(e)$
(see Lemma \ref {Lorthog}).
Since $B$ is simple, $\tau(c)>0$ for all $\tau\in T(B).$ 
It follows that
\beq
d_\tau(f_{\eta}(e))<d_\tau(b)\rforal \tau\in T(B).
\eneq
Thus we also have 
\beq
d_\tau(a(k, \ep))<d_\tau(b)\rforal \tau\in T(B).
\eneq
%Since $\widehat{[b]}$ is continuous on $T(B),$ i
Recall that $\lambda([e_C])\le [b]$ implies that ${\rm Cu}^\sim(\pi_\C^B)\circ \lambda([e_C])\le [\pi_\C^B(b)].$
Thus, in both case (i) and (ii), 
 by Theorem \ref{LcomparisonU},   there exist a sequence of unitaries $U_n\in M_{2N}(\td B)$ 
 and a sequence of elements $h_n\in {\rm Her}(b)_+$  with $\|h_n\|\le 1$ such that
 \beq\label{Llimintmaps-1-0}
%U_n^*f_\ep(\phi_{k(\ep)}(e_C))U_n\in \td B\andeqn 
\|U_n^*f_\ep(\phi_{k(\ep)}(e_C))U_n-h_n\|<1/2^{n+1},\,\,\, n=1,2,....
 \eneq
 %(where we identify $\td B$ with the first corner of $M_2(\td B)$).
 %We may assume that $\|h_n\|=1,$ $n=1,2,....$   
 Put  $\ep_n>0$ such 
 that $\lim_{n\to\infty}\ep_n=0.$ 
 One obtains a sequence of elements $e_n\in {\rm Her}(b)_+$ with $\|e_n\|=1$ 
 and a sequence of uniaries $V_n\in M_{2N}(\td B)$ 
 such that
 \beq\label{Llimintmaps-1}
 %&&V_n^*f_{\ep_n}(\phi_{k(\ep_n)}(e_C))V_n\in \td B\andeqn\\\label{Llimintmaps-2}
 && \|e_nV_n^*f_{\ep_n}\phi_{k(\ep_n)}(e_C)V_ne_n-V_n^*f_{\ep_n}(\phi_{k(\ep_n)}(e_C))V_n\|<1/2^n,\,\,n=1,2,....
 \eneq
Put $C_n=\overline{f_{2\ep_n}(e_C)Cf_{\ep_n}(e_C)},$ 
$\Phi_n: C\to  M_{2N}(\td B)$ by $\Phi_n(c)=V_n^*\phi_{k(\ep_n)}(c)V_n,$ and 
 \morp s $L_n: C\to {\rm Her}(b)$ such that $L_n(c)=e_nV_n^*\phi_{k(\ep_n)}(f_{\ep_n}(e_C)cf_{\ep_n}(e_C))V_ne_n$
 for $c\in C.$ Then 
 \beq
 \lim_{n\to\infty}\|L_n(c)L_n(c')-L_n(cc')\|=0\rforal c, c'\in C.
 \eneq 
 Since $C$ is semiprojective, there exists a sequence of \hm s $\psi_n: C\to {\rm Her}(b)$
 such that
 \beq
 \lim_{n\to\infty}\|\psi_n(c)-L_n(c)\|=0\rforal c\in C.
 \eneq
 Let $\Phi_n^\sim, \psi_n^\sim: \td C\to \td B$  be the usual unitization of $\Phi_n$ and $\psi_n,$
 %$\psi_n^\sim,$ 
 respectively.
 Then, by \eqref{Llimintmaps-1}, for a fixed $m,$ on $\C\cdot 1_{\td C}+C_m,$
 \beq
 \lim_{n\to\infty}\|\psi_n^\sim(c)-\Phi_n^\sim(c)\|=0\,\, (\text{for all}\,\, c\in \C\cdot 1_{\td C}+C_m).
 \eneq
 Note that $V_n$ are unitaries in $M_2(\td B).$ Hence ${\rm Cu}^\sim(\Phi_n)={\rm Cu}^\sim(\phi_{k(\ep_n)}).$
 It follows from Lemma \ref{LappCum} that
 \beq
 \lim_{n\to\infty}{\rm Cu}^\sim(\psi_n)=\lambda.
 \eneq

  For case (2),   we work in $B.$ 
  By the end of \ref{151}, 
   $\lambda([e_C])\le [b]$ in ${\rm Cu}(B).$
 % implies that $\lambda([e_C])
  %$\widehat{\lambda([e_C])}\le \widehat{[b]}$ in $\LAff_+(T(B))$ (see also Lemma 6.10).
  %
  Then, instead of  \eqref{Llimintmaps-1-0},
  % and \eqref{Llimintmaps-1}, 
  %and  \eqref{Llimintmaps-2},
   since $B$ has almost stable rank one, 
   % has strict comparison, 
   by  Lemma \ref{LRordam1},
   %Lemma 3.2 of \cite{eglnp}, 
   there  is, for each $k,$  a unitary 
  $U\in \td M_{2N}(\td B)$ such that
  \beq
  U^*f_\ep(\phi_k(e_C))U\in {\rm Her}(b).
  \eneq
  The rest of the proof is similar but simpler. 
 \end{proof}

\begin{rem}\label{Rcountexample}
Let $\phi: C\to M_n(\td B)$ be a \hm\, such that $[\phi(e_C)]\le [b]$ in ${\rm Cu}^\sim(\td B)$
for some $b\in \td B_+,$ where $e_C$ is a strictly positive element of $C.$   
Since we do not know whether ${\rm Cu}^\sim(\td B)$ has the cancellation,  in the case that $[\phi(e_C)]$
is represented by a projection, 
there might not be any $d\in \td B_+$
such that $[d]=[\phi(e_C)]$ in ${\rm Cu}^\sim(\td B).$   In that case, there would not be any
\hm\, $\psi: C\to \td B_+$  such that ${\rm Cu}^\sim(\psi)={\rm Cu}^\sim(\phi).$
Suppose that there is $d\in \td B_+$ such that  $[d]=[\phi(e_C)]$ in ${\rm Cu}^\sim(\td B).$
We still do not know $d\sim \phi(e_C)$ in ${\rm Cu}(\td B)$ without knowing the cancellation in ${\rm Cu}^\sim(\td B).$

Suppose that $[\phi(e_C)]$ is not a compact element. 
In an ideal  situation, say there is $x\in M_n(\td B)$ such that $x^*x=\phi(e_C)$ and $xx^*\in {\rm Her}(d),$ then 
 one obtains a partial isometry $v\in M_n(\td B)^{**}$ such that
$v^*v\phi(c)=\phi(c)v^*v=\phi(c)$ for all $c\in C$ and $v\phi(c)v^*\in {\rm Her}(d).$
Define  $\psi: C\to {\rm Her}(d)$ by $\psi(c)=v\phi(c)v^*$ for all $c\in C.$ Then 
${\rm Cu}(\psi)={\rm Cu}(\phi).$ However, ${\rm Cu}^\sim(\psi)$ may not be the same as ${\rm Cu}^\sim(\phi)$
(see Example \ref{Counexm} below).
It is  crucial  that we have   unitaries $U_n$ in Theorem \ref{LcomparisonU}.

Let us assume  that $\lambda([e_C])$ is compact and $\lambda([e_C])\le [b]$ for some $b\in \td B_+.$ 
Suppose that $[\lambda([e_C])\not=[1_{\td B}].$ Since $B$ is stably projectionless, $\td B$
has only one nonzero projection $1_{\td B}.$ To see this, let $p\in \td B$ be a nonzero 
projection. Then $p\not\in B.$ Therefore $\pi_\C^B(p)=\pi_\C^B(1_{\td B}).$ 
This implies that $1_{\td B}-p\in B.$ Since $B$ is stably projectionless, $p=1_{\td B}.$
Therefore, in this case, $\lambda([e_C])$ cannot be represented 
by an element in $\td B.$ Consequently, there will be no sequence of \hm s $\psi_k: C\to \td B$ such 
that $\lim_{k\to\infty}{\rm Cu}^\sim(\psi_k)=\lambda.$  Even if $\lambda([e_C])=[1_{\td B}]$ in ${\rm Cu}^\sim(\td B)$ 
and $\phi_k: C\to M_N(\td B)$ is a sequence of \hm s such that $\lim_{k\to\infty}{\rm Cu}^\sim(\phi_k)=\lambda,$
and each $\psi_k(e_C)$ is a projection so that $[\psi_k(e_C)]=\lambda([e_C])=[1_{\td B}]$ in 
${\rm Cu}^\sim(\td B),$ one may not have $\psi_k(e_C)\sim 1_{\td B}$ in ${\rm Cu}(\td B).$ It is then
impossible to perturb $\phi_k$ into \hm s  $\psi_k: C\to  \td B$  such that
$\lim_{k\to\infty}{\rm Cu}^\sim(\psi_k)=\lambda.$

\end{rem}

 \section{Unitization}

The following is a result of L. Robert.

\begin{lem}[Lemma 3.2.1 of \cite{Rl}]\label{Lunitalext}
 Let $A$ be a  \CA\, of stable rank one and $B$ be a unital \CA\, with finite stable rank.
 Let $e_A\in A$ be a strictly positive element. Let $\af: {\rm Cu}^\sim(A)\to {\rm Cu}^\sim(B)$
 be a morphism in ${\bf Cu}$ such that $\af([e_A])\le [1_B].$ 
 Then there exists a unique morphism $\af^\sim: {\rm Cu}^\sim(\td A)\to {\rm Cu}^\sim(B)$
 in ${\bf Cu}$ 
 such that $\af([1_{\td A}])=[1_B].$
 \end{lem}
 
 \begin{proof}
 We keep  the same notation in the proof of Lemma 3.2.1 of \cite{Rl}.
 For any $[a]\in W(\td A)$ such that $[\pi(a)]=n<\infty,$ one defines, for any integer $m,$ (in ${\rm Cu}^\sim(B)$)
 \beq	\label{Rl33}
 \af^\sim([a]-m[1_{\td C}])=\af([a]-n[1_{\td C}])+(n-m)[1_{\td B}].
 \eneq
 (Note that, by subsection 4.2 of \cite{RS}, the revised definition of ${\rm Cu}^\sim(B)$ is the same as that 
 defined in \cite{Rl}).
 The exactly the same proof first shows that such $\af^\sim$ is uniquely defined, additive 
 and {{sends positive elements to positive elements.}}  Let $a_1,a_2\in (\td A\otimes {\cal K})_+$  with 
 $[a_1], [a_2]\in W(\td A)$  be such that $[a_1]\le [a_2].$
 {{We also use $\pi: \td A\to \C$ as in the proof of Lemma 3.2.1 of \cite{Rl}.}}
 If $[\pi(a_1)]=[\pi(a_2)],$ as in the proof of Lemma 3.2.1 of \cite{Rl}, $\af^\sim([a_1])\le \af^\sim([a_2]).$
 Consider now the case $[\pi(a_1)]<[\pi(a_2)].$ 
 Let $1>\ep>0.$ Choose $0<\dt<\ep/8$ such that $\pi(f_{2\dt}(a_1))=\pi(a_1)$ and $\pi(f_{2\dt}(a_1))$
 is a projection. We may also assume that $\pi(f_{2\dt}(a_1))< \pi(a_2),$ by 
 replacing $a_2$ with $u^*g(a_2)u$ for some strictly positive functions in $C_0((0, \|a_2\|])$
 and a scalar unitary $u\in {\cal K}^\sim.$
 \Wlog, we may further assume that $f_\dt(a_1)\in {\rm Her}(a_2)$ (see Proposition 2.4 of \cite{Rr2}).
 Choose $a_3\le a_2$ such that $\pi(a_3)\perp \pi(f_{2\dt}(a_1))$ and $[\pi(a_3)]+[\pi(f_{2\dt}(a_1)]=[\pi(a_2)].$
 Put $c=(1_{(A\otimes {\cal K})^\sim}-f_{\dt/2}(a_1))a_3(1_{(A\otimes {\cal K})^\sim}-f_{\dt/2}(a_1)).$
 Then $\pi(c)\perp \pi(f_{2\dt}(a_1))$ and 
 $\pi(c)=\pi(a_3).$ 
 Now $[c]\in W(\td A)$ and 
 \beq\label{Unt-409}
 [(a_1-\dt)_+]+[c]\le [a_2]\andeqn [\pi(a_1-\dt)_+]+[\pi(c)]=[\pi(a_2)].
 \eneq
 Let $n_{1,\dt}=[\pi((a_1-\dt)_+)].$  
 %There are finite  positive scalar matrixes ${\bar a}_{1, \dt}$ and ${\bar c}\in (\C\cdot 1)\otimes {\cal K}$  such that   $[\pi({\bar a}_{1,\dt})]=n_{1,\dt}[1],$ $\pi({\bar c})=\pi(c)$ and 
% $a_{1,\dt}':=(a_1-\dt)_+-{\bar a}_{1,\dt}, c':=c-{\bar c}\in A\otimes {\cal K},$ and $a_{1,\dt}\perp c'.$
 %
 %
 Then,  since $\af^\sim $  maps positive elements to positive elements
 and is additive,  by \eqref{Rl33} and \eqref{Unt-409}, as in  the proof 
 of Lemma 3.2.1 of \cite{Rl}, one computes that
 \beq
 &&\hspace{-0.9in}\af^\sim([(a_1-\dt)_+])\le \af^\sim([(a_1-\dt)_+])+\af^\sim ([c])
 =\af^\sim([(a_1-\dt)_+]+[c])\\
 %=\af([(a_1-\dt)_+]-n_{1,\dt}[1])+n_{1,\dt}[1]\\
   && \le \af([(a_1-\dt)_+]+[c]-n_{1,\dt}[1]-[\pi(c)][1])+ (n_{1, \dt}+[\pi(c)])[1]
   \le \af^\sim ([a_2]).
 \eneq
% One then proceeds the same proof of Lemma 3.2.1 of \cite{Rl} to show that $\af^\sim$ preserves 
 %the order.
 %Then, as in the proof of 3.2.1 of \cite{Rl},   in 
 %\beq
 %\af^\sim([(a_1-\dt)_+])\le \af^\sim([a_2]).
 %\eneq
 As in the proof of Lemma 3.2.1 of \cite{Rl}, it follows that $\af^\sim$ preserves the order. 
 One then proceeds the rest of the proof of 
 Lemma 3.2.1 of \cite{Rl}. 
 \end{proof}

\begin{rem}
In \ref{Counexm}, it will be shown that there are  \hm s $\phi, \psi: A\to B$ such that ${\rm Cu}(\phi)={\rm Cu}(\psi)$ but
${\rm Cu}(\phi^\sim)\not={\rm Cu}(\psi^\sim).$ It may be worth noticing
that Lemma \ref{Lunitalext} deals with a  different situation. 
%This might be 
%a good time to see the importance of the augmented Cuntz semigroup 
%${\rm Cu}^\sim(A)$ which carries also the information of $K_0(A)$ (in particular, 
%when $K_0(A)_+=\{0\}$ but $K_0(A)\not=\{0\}$).

\end{rem}

\begin{df}

Let $F_1$ and $F_2$ be two finite dimensional \CA s.
Suppose that there are  (not necessary unital)  \hm s
$\phi_0, \phi_1: F_1\to F_2.$
Define
\beq\label{dd63}
A=A(F_1, F_2,\phi_0, \phi_1)
=\{(f,g)\in  C([0,1], F_2)\oplus F_1: f(0)=\phi_0(g)\andeqn f(1)=\phi_1(g)\}.
\eneq
Denote by ${\cal C}$ the class of all  \CA s of the form $A=A(F_1, F_2, \phi_0, \phi_1).$
These \CA s {{are called  Elliott-Thomsen building blocks as well as
one  dimensional non-commutative CW complexes  (see  \cite{ELP1} and  \cite{point-line}).}}

Denote by ${\cal I}_0$  the subclass of \CA s  $C$ in ${\cal C}$
such that $K_1(C)=\{0\}.$  

All \CA s in ${\cal C}$ have stable rank one (see, for example, Lemma 3.3 of \cite{GLN}) and are semiprojective (see 
Theorem 6.22 of \cite{ELP1}). 
\end{df}

\begin{lem}\label{LheredC0}
Let $A, C\in {\cal I}_0$ be  \CA s such that  there is an isomorphism
$\phi:A\otimes {\cal K}\cong C\otimes {\cal K}.$
%where $C_0=C_0((0,1]).$
Then  there exists an integer $n\ge 1$ and an injective  \hm\, $\iota: \phi(A)\to M_n(C)$
such that $\iota\circ \phi(A)$ is a full  \SCA\, of $M_n(C)$
and ${\rm Cu}^\sim(\iota)={\rm Cu}^\sim(\id_{\phi(A)}).$ 
\end{lem}
(Note that we identify $A$ with  the first corner $A\otimes e_{1,1}$ of $A\otimes {\cal K}.$)
\begin{proof}
Let $D$ be a liminal \CA.
%\, whose irreducible presentations are finite dimensional, or 
Denote by ${\rm Irr}(D)$ the set of irreducible representations of $D.$
If $d\in D_+$ and $\xi\in {\rm Irr}(D),$  let us denote $r_\xi(d)$ 
the rank of $\xi(d)$ (with value in $\{0\}\cup \N\cup \{\infty\}$).
 
Let $e_A\in A_+^{\bf 1}$ be a strictly positive element of $A.$
%Since $A$ has the form in  \eqref{dd63}, 
There is $N\ge 1$ such that 
$$
1\le \inf\{ r_\xi(e_A): \xi\in {\rm Irr}(A\otimes {\cal K})\}\le \sup\{r_\xi(e_A): \xi\in {\rm Irr}(A\otimes {\cal K})\}\le N,
$$ 
viewing $A$ as a hereditary \SCA\, of $A\otimes {\cal K}.$
%%%%%%%%%%%%%%%%%%%
%
%
\iffalse
$d_{\tau\otimes {\rm Tr}}(e_A)\le N$
for all $\tau\in T(A),$ where ${\rm Tr}$ is the densely defined trace on ${\cal K},$ viewing 
$A$ as a hereditary \SCA\, of $A\otimes {\cal K}.$
\fi
%%%%%%%%%%%%%%%%%%
%
Put $A_1=\phi(A).$  Then $A_1$ is a full hereditary \SCA\, of $C\otimes {\cal K}$
as isomorphisms preserve the full hereditary \SCA s.
Hence $A_1={\rm Her}(\phi(e_A)).$  Note that,  since $\phi$ is an isomorphism,  
$$
1\le \inf\{ r_\xi(\phi(e_A)): \xi\in {\rm Irr}(C\otimes {\cal K})\}\le \sup \{ r_\xi(\phi(e_A)): \xi\in {\rm Irr}(C\otimes {\cal K})\} \le N.
$$
%$d_{\tau\otimes {\rm Tr}}(\phi(e_A))\le N$ for all 
%$\tau\in T(C).$
%%%%%%%%%%%%%%%%%%%%%%%%%%%%%%%%%%%%%%%%%%%
%
\iffalse
there exists 
$N\ge 1$ such that 
$d_{\tau\otimes Tr}(\phi(e_A))\le N$  for all $\tau\in T(B),$ 
% for any tracial state of $\tau\in T(C),$ 
%$(\tau\otimes Tr)(\phi(e_A))<\infty,$ 
where $Tr$ is the densely defined trace on ${\cal K}.$
Write
$$
C=\{(f,b)\in C([0,1], F_2)\oplus F_1: f(0)= {\psi}_0(b)\andeqn f(1)={{\psi}}_1(b)\},
$$
where $F_1$ and $F_2$ are finite dimensional \CA s, and ${{\psi_i}}: F_1\to F_2$
is a \hm, $i=0,1.$
\fi
%%%%%%%%%%%%%%%%%%%%%%%%%%%%%%%%%
%Therefore, 
% for each $t\in [0,1],$
%$\phi(e_A)(t)$ has finite rank.
Fix a strictly positive element $e_C\in C.$ 
Then, there is $N_1\ge 1$ such that
$$
1\le \inf\{ r_\xi(e_C): \xi\in {\rm Irr}(C)\}\le \sup\{ r_\xi(e_C): \xi\in {\rm Irr}(C)\}\le N_1.
$$
Let $\{e_{i,j}\}\subset {\cal K}$ be a system of matrix units
and $E_j=\sum_{i=1}^je_{i,i}$ (for $j\ge 1$). 
%We now identify $e_C$ with $e_C\otimes e_{1,1}.$ 
Put $c_n:=e_C\otimes E_n.$ Then $r_\xi(c_n)=n\cdot r_\xi(e_C).$
%, for some integer $n\ge 1,$ 
%Note 
%Therefore there is an integer $n\ge 1$ such that   
%$$
%r_\xi(e_C\otimes E_m)=m\cdot r_\xi(e_C),\,\,\, m=1,2,....
%$$
Therefore there is an integer $n\ge 1$ 
% and 
 %a strictly positive element
%$c\in M_n(C)$  for some integer $n\ge 1$
%(for example $c:=e_C\otimes E_n$)
 such that $d_\tau(\phi(e_A))< d_\tau(c_n)$
for all $\tau\in T(C).$
Working in ${\tilde C}$ if $C$ is not unital,
%Put $C=C([0,1])$ and view $C_0\otimes {\cal K}\subset C\otimes {\cal K}.$
by 3.18 of \cite{GLN}, $\phi(e_A)\lesssim c_n$ in ${\rm Cu}(C).$
%Since $C_0\otimes {\cal K}$ is a hereditary \SCA\, of $C\otimes {\cal K},$
%one has $\phi(e_A)\lesssim c$ in ${\rm Cu}(C_0).$
Note that $\phi(A)$ is a full hereditary \SCA\, $C\otimes{\cal K}.$ 
Since $C\otimes {\cal K}$ has stable rank one, by Theorem 1.0.1 of \cite{Rl},
there is a \hm\, $\iota: \phi(A)\to M_n(C)$ such 
that ${\rm Cu}^\sim (\iota)={\rm id}_{{\rm Cu}^\sim(C)}.$

Note, if $\iota(c)=0$ for some $c\in C_+,$ then ${\rm Cu}^\sim (\iota)([c])=0.$
Thus, $\iota$ is injective.
%there is a partial isometry $w\in
%(C\otimes {\cal K})^{**}$ as described in the lemma (see \ref{LuniH}).
To see $\iota\circ \phi(A)$ is full, one needs to show that $\iota\circ \phi(e_A)$ is full 
in $M_n(C).$ But $\iota\circ \phi(e_A)\sim \phi(e_A)$ and $\phi(e_A)$ is full
since $\phi$ is an isomorphism. 
\end{proof}

We will use the following known  and easy fact.

 \begin{lem}\label{Lproj}
Let $B$ be as in \ref{151}. Suppose that $b\in M_n(\td B)_+$ is such that $[b]$ is a compact element in 
${\rm Cu}^\sim(B).$     Then there is $g\in C_0((0, \|b\|])$ such that $g(b)$ is a projection. 
\end{lem}

\begin{proof}
By Theorem 6.1 of \cite{RS} (recall $B$ has stable rank at most 2),  there is a projection $p\in M_N(\td B)$ for some integer $N\ge 1$ such 
that $[b]=[p]$ in ${\rm Cu}^\sim (\td B).$ 

If $0$ is not an isolated point of ${\rm sp}(b),$ for any $\ep>0,$ there is a nonzero element $c\le b$ such that $c\perp f_\ep(b).$
Since $B$ is simple, $\tau(c)\not=0$ for any $\tau\in T(B).$ It follows that
\beq\label{Lproj-1}
d_\tau(f_\ep(b))<d_\tau(p)\rforal \tau\in T(B).
\eneq
However, since $p$ is compact, for all small $\ep,$ $[p]\le [f_\ep(b)]$ in ${\rm Cu}(\td B)^{\circeq}.$
This contradicts with \eqref{Lproj-1}. So 
$0$ must be an isolated point of ${\rm sp}(b).$ Thus there is a such $g$ so that $g(b)$ is a projection.
\end{proof}

 \begin{thm}\label{Lunithm}
 Let $C$ be a \CA\, in ${\cal I}_0$ with a strictly positive element $e_C$ and $B$ be a simple \CA\, which satisfies conditions in \ref{151}.
 Suppose that $\lambda: {\rm Cu}^\sim (\td C)\to {\rm Cu}^\sim (\td B)$ is a morphism in ${\bf Cu}$ such that
 $\lambda([1_{\td C}])=[b]$ for some (compact element) $b\in M_n(\td B)_+$   (for some integer $n\ge 1$).
  Suppose also that there exists a sequence of \hm s $\phi_k: C\to M_n(\td B)$
 such that $\lim_{k\to\infty}{\rm Cu}^\sim (\phi_k)=\lambda|_{{\rm Cu}^\sim(C)}.$
 
 (1) If $\lambda([e_C])$ is not a compact element, 
 then there exists a sequence of \hm s $\psi_k: \td C\to \overline{bM_n(\td B)b}$
 such that 
 $$
 \lim_{k\to\infty}{\rm Cu}^\sim(\psi_k)=\lambda.
 $$
 
 (2) If $\lambda([e_C])$ is a compact element and $\lambda([c])\not=0$ for any $c\in C_+\setminus \{0\},$ 
then there exists a sequence of \hm s $\psi_k: \td C\to M_n(\td B)$ such that
$$
\lim_{k\to\infty}{\rm Cu}^\sim(\psi_k)=\lambda.
$$

%(3) If $\lambda([e_C])$ is a compact element and $\lambda([c])\not=0$
%for any $c\in {\td C}_+\setminus \{0\},$ then the conclusion of (1) also holds.

 \end{thm}
 
 \begin{proof}
 Consider case (2) first. 
 If $\lambda([e_C])$ is a compact element, then, 
 for all sufficiently small $0<\ep<1,$ 
 \beq
 \lambda([e_C])\le \lambda([f_\ep(e_C)])\le \lambda([e_C]).
 \eneq
 Let $g\in C_0((0,\|e_C\|])_+$ with the support in $(0, \ep/2].$ Then 
 $\lambda([g(e_C)])=0.$ 
 The assumption on $\lambda$ 
 % $\lambda$ is strictly positive, 
  implies that $g(e_C)=0.$
 It follows that $C$ is unital.  Since $[e_C]=[1_C]\ll [1_C],$ this implies 
 that $[\phi_k(1_C)]=\lambda([1_C])$ (for all large $k$).
 
 Let $e_0:=1_{\td C}-1_C.$ By Lemma \ref{Lproj}, we may assume that $b=p$
 for some projection $p\in M_n(\td B).$ If $\lambda([e_0])=0,$   then $\lambda([1_C])=\lambda([1_{\td C}]).$
 Define $\psi_k: \td C\to M_n(\td B)$ by ${\psi_k}|_{C}=\phi_k$  and 
 $\psi_k(1_{\td C})=\phi_k(1_C).$  (Warning: we only have $[\psi_k(1_{\td C})]+2[1_{\td B}]=
 [p]+2[1_{\td B}]$ in ${\rm Cu}(\td B)$ for large $k.$)
 %This proves (2).
 
 If $\lambda([e_0])\not=0,$ then, for each $k,$ 
 \beq
 d_\tau(\phi_k(1_C))\le d_\tau(p)\rforal \tau\in T(\td B)\andeqn d_\tau(\phi_k(1_C))<d_\tau(p)\rforal \tau\in T(B).
 \eneq
 It follows from Corollary A.4 of \cite{eglnkk0} that  (since $\hat{p}$ is continuous on $T(B)$)
 \beq
 \phi_k(1_C)\lesssim p\,\,\,\hspace{0.2in} {{{\rm in}\,\,\, {\rm Cu}(\td B).}}
 \eneq
There is a partial isometry $v_k\in M_n(\td B)$ such that  $v_kv_k^*=\phi_k(1_C)$ and $v_k^*v_k\le p.$ 
%$v_k^*\phi_k(1_C)v_k\le p.$
Define $\psi_k: \td C\to pM_n(A)p$ by 
$\psi_k(c)=v_k^*\phi_k(c)v_k$ for all $c\in C$ and $\psi_k(1_{\td C})=p.$ 
Since $C$ is unital, ${\rm Cu}^\sim({\psi_k}|_{C})={\rm Cu}^\sim(\phi_k).$ It follows 
from  Lemma \ref{Lunitalext} that  (2)  holds.  
 
For (1), we assume that
 $\lambda([e_C])$ is not a compact element.   
 Again, by  Lemma \ref{Lproj} and the fact $1_{\td C}$ is a projection, we may assume $b=p$ is a projection.
 By Lemma \ref{Llimintmaps} ((ii) of (1)), we may assume that there is a sequence of \hm s $\phi_k$ which maps 
 $C$ into $\overline{pM_n(\td B)p}$    such that $\lim_{k\to\infty}{\rm Cu}^\sim(\phi_k)=\lambda|_{{\rm Cu}^\sim(C)}.$
 Define $\psi_k: \td C\to \overline{pM_n(\td B)p}$
 such that ${\psi_k}|_C=\phi_k$ and $\psi_k(1_{\td C})=p.$ Then  
 %(by Lemma \ref{Lunitalext})
 \beq\nonumber
 {\rm Cu}^\sim({\psi_k}|_{C})={\rm Cu}^\sim(\phi_k)\andeqn \,\,{\rm{(see\,\,\, Lemma\,\, \ref{Lunitalext})}}\,\,
 \lim_{k\to\infty}{\rm Cu}^\sim({\psi_k})=\lambda.
 \eneq
  \end{proof}

 The condition that $\lambda([c])\not=0$ for all $c\in C_+\setminus \{0\}$ may be called 
 ``strictly positive".
 
\begin{cor}\label{corunithm}
 Let $C$ be a \CA\, in ${\cal I}_0$  with a strictly positive element $e_C$ and $B$ be a simple \CA\, which satisfies conditions in \ref{151}.
 Suppose that $\lambda: {\rm Cu}^\sim (C)\to {\rm Cu}^\sim (\td B)$ is a morphism in ${\bf Cu}$ such that
 $\lambda([e_C])\le [e]$ for some nonzero projection $e\in M_n(\td B)$ (for some $n\in \N$).
 Suppose also that there exists a sequence of \hm s $\phi_k: C\to M_n(\td B)$
 such that $\lim_{k\to\infty}{\rm Cu}^\sim (\phi_k)=\lambda.$ 
 Then there exists a sequence of \hm s $\psi_k: \td C\to M_{n}(\td B)$
 such that 
 $$
 \lim_{k\to\infty}{\rm Cu}^\sim({\psi_k}|_{C})=\lambda.
 $$
 Moreover, if $\lambda([e_C])$ is not a compact element, then we may require that 
 $\psi_k(C)\subset eM_n(\td B)e.$
 \end{cor}

 \begin{proof}
 Define ${\psi_k}|_C=\phi_k$ and $\psi_k(1_{\td C})=1_n.$ Then the first part of the statement 
 follows.
 % from Lemma \ref{Lunitalext}.  
 For the second part, we note that, 
 since $B$ is simple and stably proectionless, and $e\in M_n(\td B)$  is a nonzero projection, 
 $e$ is a full element in $M_n(\td B).$ It follows that $eM_n(\td B)e\otimes {\cal K}\cong \td B\otimes {\cal K}.$
 Put $D=eM_n(\td B)e.$
 By  theorem
 5.5 of \cite{RS}, ${\rm Cu}^\sim(D)\cong {\rm Cu}^\sim (\td B).$ 
 Then the second
 % ``Moreover" 
 part of the corollary follows from 
 %the  combination of Lemma \ref{Lunitalext} and 
 part (1) of Theorem \ref{Lunithm}.
 
% In general situation, 
 %we may   define ${\psi_k}|_C=\phi_k$ and $\psi_k(1_{\td C})=1_n.$
 %
 \iffalse
 If $\lambda([e_C])$ is a compact element, then, there is $k_0\ge 1$ such that
 $\phi_k(C)$ is unital for all $k\ge k_0,$ by Lemma \ref{Lproj}. 
 Put $C_k:=C/{\rm ker}\phi_k,$ $k=1,2,...$ Let $p_k$ be the unit of $\phi_k(C)$ and let $1_{m}$ be 
 the identity of $M_m(\td B)$ for integer $m\ge 1.$
 Then $[p_k]\le [1_n]\le [1_{n+1}].$  It follows from Corollary A.4 of \cite{eglnkk0} that
 \beq
 p_k\lesssim 1_{n+1}.
 \eneq
Therefore there is a partial isometry $v_k\in \td B\otimes {\cal K}$ such that $v_k^*p_kv_k\le 1_{n+1}.$
Define $\psi_k: \td C\to M_{n+1}(\td B)$ such that $\psi_k(c)=v_k^*\phi_k(c)v_k$ for all $c\in C$ and $\psi_k(1_{\td C})
=v_k^*p_kv_k,$ $k=1,2,....$ 
Note that, since $C_k$ is unital, ${\rm Cu}^\sim({\psi_k}|_{C})={\rm Cu}^\sim(\phi_k)$ for 
each $k.$ Thus $\lim_{k\to\infty}{\rm Cu}^\sim({\psi_k}|_{C})=\lambda.$
\fi
%
 \end{proof}

\begin{exm}\label{Counexm}
By Theorem 5.27 of \cite{GLIII}, there is a separable simple stably projectionless \CA\, $A$ with nontrivial $K_0(A)$  and with 
continuous scale such that
 ${\rm ker}\rho_A=K_0(A)$  and  $A=\lim_{n\to\infty}(C_n, \phi_n),$
where $C_n\in {\cal I}_0$ and $\phi_{n, \infty}: C_n\to C$ is injective. 
Choose $C_n$ so that ${\phi_{n, \infty}}_{*0}(K_0(C_n))\not=0.$ This also implies that $K_0(C_n)\not=\{0\}.$
Note  that, since $A$ is stably projectionless,  $C_n$ is also stably projectionless. 

Let $B:=A\otimes {\cal W},$ where ${\cal W}$ is the unique separable amenable $KK$-contractible \CA\, 
with a unique tracial state (see \cite{eglnkk0}). Then $B$ has continuous scale and $T(B)=T(A),$
and $B$ is $KK$-contractible. By the classification theorem in \cite{eglnkk0}, $B$ is in fact 
a simple  inductive limit of Razak algebras.   Then, by Proposition 6.2.3 of \cite{Rl}, 
\beq\nonumber
{\rm Cu}^\sim(B) ={{\{0\}}}\sqcup \LAff_+^\sim (T(B))=\{0\}\sqcup \LAff_+^\sim(T(A))\andeqn {\rm Cu}^\sim(A)=K_0(A)\sqcup \LAff_+^\sim(T(A)).
\eneq
By Theorem 1.0.1 of \cite{Rl}, there is a \hm\, $j: A\to B$ such that 
$$
{\rm Cu}(j)|_{K_0(A)}=0\andeqn
{\rm Cu}^\sim(j)|_{\LAff_+^\sim(T(A))}={\rm id}_{\LAff_+^\sim(T(A))}.
$$ 
There is also 
a \hm\, 
$
\iota: B\to A$ such that ${\rm Cu}^\sim(\iota)|_{\LAff_+^\sim(T(B))}={\rm id}|_{\LAff_+^\sim(T(B))}.
$
Let $\psi:=\iota\circ j\circ \phi_{n, \infty}:C_n\to A.$   Note  ${\rm Cu}^\sim(\iota\circ j)|_{\LAff_+^\sim(T(A))}=
\id_{\LAff_+^\sim(T(A))}.$
Since $C_n$ is stably projectionless, 
one has $${\rm Cu}(\psi)={\rm Cu}(\phi_{n, \infty}).$$
But,  since $\psi_{*0}=0$ and ${\phi_{n, \infty}}_{*0}\not=0,$
$${\rm Cu}^\sim(\psi)\not={\rm Cu}^\sim(\phi_{n, \infty})\andeqn  
%In particular, $
{\rm Cu}(\psi^\sim)\not={\rm Cu}(\phi_{n, \infty}^\sim).$$

\end{exm}

 \section{Existence}
 
% The following is known when $B$ has stable rank one (see \cite{Rl}). 
 
 \begin{lem}\label{LembedingC0}
 Let $A$ be a separable simple stably projectionless  \CA\, 
 with continuous scale such 
 that $M_m(A)$ has almost stable rank one for all $m\ge 1.$
 Suppose also  $QT(A)=T(A)$ and ${\rm Cu}(A)=\LAff_+(T(A)).$ 
 
 Then, for any ${\bf Cu}$ morphism $\lambda: {\rm Cu}^\sim (C_0((0,1])))\to {\rm Cu}^\sim(\td A)$
 with $\lambda([e_C])\le [a]$ for some $a\in M_n(\td A)_+$ (for some integer $n\ge 1$),
  where 
 $e_C$ is a strictly positive element of $C_0((0,1]),$ and $\lambda([c])\not=0$ for any $c\in C_0\otimes {\cal K}_+\setminus \{0\},$  
 there is a \hm\, $h: C_0((0,1])\to M_n(\td A)$   such that 
 ${\rm Cu}^\sim(h)=\lambda.$
 
 Moreover, if $\lambda: {\rm Cu}^\sim(C_0((0, 1])))\to {\rm Cu}^\sim (A)$ with $\lambda([e_C])\le [a]$
 for some $a\in M_n(A)_+,$ then there exists a \hm\, $h: C_0((0,1]))\to M_n(A)$
 such that ${\rm Cu}^\sim(h)=\lambda.$ 
 
  \end{lem}
 
 \begin{proof}
 Recall that $A$ shares the same condition that $B$ has in \ref{151}.
 Put $C_0:=C_0((0,1]).$ Recall that $K_i(C_0)=\{0\},$ $i=0,1.$ 
 Note that, since $C_0$ has stable rank one, ${\rm Cu}(C_0)$ is orderly embedded into 
 ${\rm Cu}^\sim(C_0).$   So ${\rm Cu}^\sim(C_0)_+={\rm Cu}(C_0)$  (see Lemma 3.1.2 of \cite{Rl}).
 Note  also that $\td A$ is unital  and has stable rank at most 2 (see the proof of Theorem 6.13 of \cite{RS}).
 Thus $\lambda$ maps ${\rm Cu}(C_0)$ to ${\rm Cu}(\td A)^{\circeq}$ (see Lemma \ref{TcomparisonintdA}
 and Corollary \ref{CtdBcomp}).
 Therefore it suffices to show that there is a \hm\, $h: C_0\to M_n(\td A)$ 
 such that ${\rm Cu}(h)=\lambda|_{{\rm Cu}(C_0)}.$ 
 
 Recall, by Theorem \ref{TcomparisonintdA},  that
 %\beq
$ {\rm Cu}(\td A)^{\circeq}=(K_0(\td A)_+\setminus \{0\})\sqcup\LAff_+(T(\td A))^\diamond.$
% \eneq
 %
 Suppose that $\lambda([c])$ is compact
 for some non-zero $c\in (C_0\otimes {\cal K})_+.$   Note 
 $[c]=\sup\{[f_{1/2^n}(c)]: n\in\N\}.$ It follows, 
 for some $n\ge1,$ 
 \beq\label{LembedingC0-n1}
 \lambda([c])\le \lambda([f_{1/2^n}(c)]).
 \eneq
 However, since $C_0$ is stably projectionless,  there is $c_0\in {\rm Her}(c)_+\setminus\{0\}$
 such that $c_0\perp f_{1/2^n}(c).$ By the assumption, $\lambda([c_0])\not=0.$
 This contradicts with \eqref{LembedingC0-n1} as $C_0$ has stable rank one.
 Hence $\lambda([c])$ is not compact for any 
 $[c]\in C_0\otimes {\cal K}_+\setminus \{0\}.$ 
 
Thus
  $\lambda({\rm Cu}(C_0))\subset \LAff_+(T(\td A))^\diamond$ (see Theorem 6.1 of \cite{RS}).

 It follows from Theorem 2.8 of \cite{eglnkk0}  that there is a separable simple 
 \CA\, $A_1$ which is an inductive limit of Razak algebras  with continuous scale such that $T(A_1)=T(A).$
 Note that $K_i(A_1)=\{0\},$ $i=0,1.$ 
 By Theorem A.26 of \cite{eglnkk0}, there is an embedding $\iota: A_1\to A$ which maps strictly positive elements 
 to  strictly positive elements such that $\iota$ induces an affine homeomorphism 
 $\iota_T: T(A)\to T(A_1).$  
 %where ${\cal W}$ is the Razak algebra, the unique separable simple 
 %amenable $KK$-contractible \CA\, with a unique tracial state (see \cite{eglnkk0}). 
 Let $\iota^\sim: (A_1)^\sim\to \td A$ 
 be the unital extension.   We also write  $\iota^\sim$ for 
 the extension from $M_n(A_1^\sim)$ to $M_n(\td A)$
 for each integer $n\ge 1.$  Thus $\iota^\sim$ induces an isomorphism 
 $\iota^{\sharp\sim}$ form $\LAff_+(T(A_1^\sim))$ onto $\LAff_+(T(\td A)).$ 
 %By Theorem \ref{TcomparisonintdA}, ${\rm Cu}(\td A)^{\circeq}=K_0(\td A)_+\sqcup\LAff_+(T(\td A)).$ 
 %Since $K_i(C_0(0,1])=\{0\},$ $i=0,1,$ 
 Since ${A_1}^\sim$ has stable rank one, 
 it follows from Theorem 1.0.1 of \cite{Rl} that there is a 
 \hm\, $\phi:C\to M_n(\td A_1)$ such that 
 \beq
 {\rm Cu}^\sim(\phi)=(\iota^{\sharp\sim})^{-1}\circ \lambda.
 \eneq
 Define $h: C\to A$ by $h=\iota^\sim\circ \phi.$   Then ${\rm Cu}^\sim(h)=\lambda.$

 %Therefore it  suffices to construct a \hm\, $h: C_0\to M_n(A_1^\sim)$
 %such that ${\rm Cu}(h)=(\iota^{\sharp\sim})^{-1}\circ \lambda.$  Since ${A_1}^\sim$ has stable rank one, 
 %the existence of such $h$ follows from Theorem 1.0.1 of \cite{Rl}.
 %
 For the ``Moreover"  part,  we first note that  ${\rm Cu}(A)=\LAff_+(T(A)),$ as $A$ is stably projectionless.
 Therefore, working in ${\rm Cu}(A),$ the above argument also works and produces 
 a \hm\, $\phi: C\to M_n(A)$ such that ${\rm Cu}(\phi)=\lambda.$  This part also follows from \cite{Rlz}.
 %
 %
 %%%%%%%%%%%%%
 \iffalse
 =\LAff_+(T(A))$ is order embedded into 
 $\LAff_+(T(\td A))$ in a natural way (
 K_0(A)\sqcup \LAff^\sim(T(A))$
 (see  ? of \cite{Rl}). 
 
 $\widehat{[a]}(\tau)=0$ 
 
 we note that, by Theorem 5.3 of \cite{RS}, ${\rm Cu}^\sim(A)$ is ordered embedded into
 ${\rm Cu}^\sim(\td A).$ So we may view $\lambda$ maps into ${\rm Cu}^\sim(\td A).$ In the above argument, 
  by Theorem 1.0.1 of \cite{Rl}, $\phi$ maps 
 $C$ into $M_n(A_1),$ a \SCA\, of $M_n(\td A_1),$ as $\iota^\shap^{-1}(\hat{[a]})$ is in ${\rm Cu}^\sim(A_1).$ 
 \fi
 %
 %%%%%%%%%%%%%%%%%
 
 \end{proof}

 \begin{df}
 Let ${\cal A}_0$ be the family of \CA s in ${\cal I}_0$ which   consists  of one  \CA\, $C_0((0,1]).$
 %are one of \CA s described in \ref{Lglii-R}.
 A \CA\, $A$ is  in ${\cal A}_n$ if $A\in {\cal I}_0$ and, if $A\otimes {\cal K}\cong B\otimes {\cal K},$  or, if $A=\td B,$ 
 or if $\td A=B$ 
 for some 
 $B\in {\cal A}_{n-1},$ $n=1,2,....$ 
  \end{df}
 
 \begin{thm}\label{TTEXTC0}
 Let $C\in {\cal I}_0$  be a \CA\,
 % with a strictly positive element $e_C$ 
 and let $A$ be a separable simple stably projectionless  \CA\, 
 with continuous scale such 
 that $M_m(A)$ has almost stable rank one for all $m\ge 1.$
 Suppose also $QT(A)=T(A)$ and ${\rm Cu}(A)=\LAff_+(T(A)).$ 
Let $e_C$ and $e_A$ be strictly positive elements of $C$ and $A,$ respectively.
 
% be a separable simple stably projectionless \CA\, which 
 %is ${\cal Z}$-stable with continuous scale and with a strictly positive element $e_A.$  
 
 (1) Suppose that  there is a morphism $\lambda: {\rm Cu}^\sim(C)\to {\rm Cu}^\sim(A)$ in ${\bf Cu}$
 such that $\lambda([e_C])\le n[e_A]$   for some integer $n\ge 1$ and $\lambda([c])\not=0$
 for any $c\in C_+\setminus \{0\}.$
 Then there is an integer $m\ge n$ and a sequence of \hm s $\phi_k: C\to M_m(A)$ such that
 $\lim_{k\to\infty}{\rm Cu}^\sim(\phi_k)=\lambda.$  
 
 (2) Also, if there is a morphism 
 $\lambda: {\rm Cu}^\sim(C)\to {\rm Cu}^\sim (\td A)$ in {\bf Cu} such that $\lambda([e_C])\le n[1_{\td A}]$
 and $\lambda([c])\not=0$ for all $c\in (C\otimes {\cal K})_+
 \setminus \{0\},$ 
 then there exists an integer $m\ge n$ and a sequence of  \hm s $\phi_k: C\to M_m(\td A)$  such that 
 $\lim_{k\to\infty}{\rm Cu}^\sim (\phi_k)=\lambda.$ 
 \end{thm}

 \begin{proof}
 It follows from (the proof of) Proposition 5.2.2  of \cite{Rl} that, $C\in {\cal A}_m$ for some $m\ge 0.$ 
 By Lemma \ref{LembedingC0}, the lemma holds for any  \CA\, $C\in {\cal A}_0$  and any $A$ which meets the requirement of the lemma.  
 
 Assume that lemma holds for  any \CA\, $C$ in ${\cal A}_{m-1}.$ It suffices to show that the lemma holds 
 for any \CA\,  $C$ in  ${\cal A}_m$ and any $A$ as described in the lemma. 
 Fix $C\in {\cal A}_m.$ 
 
% We first consider (2) of the statement.
 
 Case (I) :  Suppose that $h: C\otimes {\cal K}\to B\otimes {\cal K}$ is an isomorphism for some $B\in {\cal A}_{m-1}.$
 In situation (2), suppose that $\lambda: {\rm Cu}^\sim(C)\to {\rm Cu}^\sim(\td A)$ is a morphism in ${\bf Cu}$
 such that $\lambda([e_C])\le n[1_{\td A}].$ 

By Lemma \ref{LheredC0}, there is an injective \hm\, $\iota: h(C)\to M_L(B)$ for 
some integer $L\ge 1$ such that ${\rm Cu}^\sim(\iota)={\rm Cu}^\sim(\id_{h(C)}).$
Since $B\in {\cal A}_{m-1},$ by the inductive assumption, there exists an integer $m_0\ge n$ and a 
sequence of \hm s $\psi_k: M_L(B)\to M_{Lm_0}(\td A)$ such 
that 
$$
\lim_{k\to\infty}{\rm Cu}^\sim(\psi_k)=\lambda\circ {\rm Cu}^\sim(h^{-1}).
$$
Define $\phi_k: C\to M_{Lm_0}(\td A)$ by
$\phi_k(c)=\psi_k\circ \iota\circ h(c)$ for all $c\in C.$ It follows that
$$
\lim_{k\to\infty}{\rm Cu}^\sim(\phi_k)=\lambda.
$$
In situation (1), $\lambda$ maps ${\rm Cu}^\sim(C)$ to ${\rm Cu}^\sim(A),$ then the argument above also works
(but $\psi_k$ maps $M_L(B)$ into $M_{Lm_0}(A)$).

 Case (II):  $C=\td B$ for some $B\in {\cal A}_{m-1}.$   Note that $C$ is unital
 and $A$ is stably projectionless.
 Hence $\lambda: {\rm Cu}^\sim(C)\to {\rm Cu}^\sim(\td A).$  (We do not need 
 to consider the case $\lambda: {\rm Cu}^\sim \to {\rm Cu}^\sim (A).$)
 Let $e_B\in B$ be a strictly positive element. 
 Note that ${\rm Cu}^\sim(B)$ is orderly embedded into ${\rm Cu}^\sim(\td B)$
 (see Proposition 3.1.6 of \cite{Rl}).
 Since $B\in {\cal A}_{m-1}$ and $\lambda([e_B])\le \lambda([e_C])\le n[1_{\td A}],$
 by the inductive assumption, there is an integer $m_0\ge n$ and a sequence of \hm s $\psi_k: B\to M_{m_0}(\td A)$
 such that 
 \beq
 \lim_{k\to\infty}{\rm Cu}^\sim(\psi_k)=\lambda|_{{\rm Cu}^\sim(B)}.
 \eneq
 If $\lambda([e_B])$ is not compact, we apply part (1) of Theorem  \ref{Lunithm} to obtain the desired 
 maps $\phi_k.$
 %Since $\lambda$ is strictly positive, if 
 If $\lambda([e_B])$ is compact, 
 %then $B$ is unital, since $\lambda$ is strictly positive. 
 %If $B$ is unital, $e_C-e_B\not=0.$ 
 %Since $\lambda$ is strictly positive, by 
 %Thus, 
 since $\lambda$ is strictly positive, by   (2) of 
 %then, either $B$ is unital, in which case, 
% $\lambda([1_{\td B}-1_B])\not=0,$ or $\lambda([e_B])$ is not compact. 
 Theorem \ref{Lunithm}, there is also a sequence of \hm s $\phi_k: C=\td B\to M_{nm_0}(\td A)$
 such that 
 \beq
 \lim_{k\to\infty}{\rm Cu}^\sim(\phi_k)=\lambda.
 \eneq

Case (3): $\td C=B$ for some $B\in {\cal A}_{m-1}.$
Let $\lambda: {\rm Cu}^\sim(C)\to {\rm Cu}^\sim (\td A)$ be such that $\lambda([e_C])\le n[1_{\td A}].$
 By Lemma \ref{Lunitalext},  there is an extension $\lambda^\sim: {\rm Cu}^\sim (\td C)\to {\rm Cu}^\sim(\td A)$ in 
 ${\bf Cu}$ such that $\lambda^\sim|_{{\rm Cu}^\sim(C)}=\lambda$ and 
 $\lambda(1_{\td C})=(n+1)[1_{\td A}].$  
 Consider the following splitting short exact sequence (see Proposition 3.1.6 of \cite{Rl}):
 \beq
 0\to {\rm Cu}^\sim(C)\to {\rm Cu}^\sim(\td C)\stackrel{{\rm Cu}^\sim(\pi_\C^C)}{\to} {\rm Cu}^\sim(\C)\to 0,
 \eneq
 where $\pi_\C^C: \td C\to \C$ is the quotient map (and its extension).
 Let $a\in  (\td C\otimes {\cal K})_+\setminus \{0\}.$ If $\pi_\C^C(a)=0,$
 then $\lambda^\sim([a])=\lambda([a])\not=0.$
 If $\pi_\C^C(a)\not=0,$ then, by the definition, 
 $\lambda^\sim([a])\not=0.$ Thus $\lambda^\sim([a])\not=0$
 for any $a\in (\td C\otimes {\cal K})_+\setminus \{0\}.$
 Since $B\in {\cal A}_{m-1},$  by the assumption,  there exists a sequence of \hm s $h_k: B=\td C\to M_{L}(\td A)$
 for some $L\ge n$
 such that  $\lim_{k\to\infty}{\rm Cu}^\sim(h_k)=\lambda^\sim.$
 Choose $\phi_k:={h_k}|_{C}.$ Then 
 $\lim_{k\to\infty}{\rm Cu}^\sim(\phi_k)=\lambda.$
 
If $\lambda: {\rm Cu}^\sim (C)\to {\rm Cu}^\sim(A)$ with $\lambda([e_C])\le n[e_A],$ 
 then, since ${\rm Cu}^\sim(A)\to {\rm Cu}^\sim (\td A)$ is an order embedding,  by Theorem 5.3 of \cite{RS}, 
 one may view $\lambda:  {\rm Cu}^\sim (C)\to {\rm Cu}^\sim(\td A).$ 
 It follows from Lemma  \ref{Lunitalext}, there is an extension $\lambda^\sim: {\rm Cu}^\sim(\td C)\to 
 {\rm Cu}^\sim (\td A)$ such that $\lambda^\sim|_{{\rm Cu}^\sim(C)}=\lambda$ and 
 $\lambda^\sim([1_{\td C}])=(n+1)[1_{\td A}].$  As proved above, $\lambda^\sim$ is strictly positive, i.e., $\lambda^\sim([c])\not=0$ for any $c\in (\td C\otimes {\cal K})_+\setminus \{0\}.$
 Since $B\in {\cal A}_{m-1},$ there exists a sequence of \hm s 
 $h_k: B=\td C\to M_L(\td A)$ such that $\lim_{k\to\infty}{\rm Cu}^\sim(h_k)=\lambda^\sim.$
 Define $\phi_k={h_k}|_{C}.$ 
 Then $\lim_{k\to\infty}{\rm Cu}^\sim(\phi_k)=\lambda.$
 
 %
 %%%%%%%%%%%%%%%%%%%%%%%%%%%%
 %
 %
 \iffalse
 %
 Note that $\lambda([e_C])\le n([e_A]).$ Therefore,
 \beq
 \widehat{\lambda([e_C])}(\tau)<\widehat{n[e_A]}(\tau)\rforal \tau\in T(A).
 \eneq
 Since $\widehat{n[e_A]}$ is continuous on $T(\td B),$  by (1) of Lemma \ref{Llimintmaps},
 we may assume that $\phi_k$ maps $C$ into $M_n(B).$
 \fi
 %%%%%%%%%%%%%%%%%%%%%%%%%%%%%%%%%%%%
 %
 %
 %%%%%%%%%%%%%%%
 \iffalse
One then notes that, for any $\ep>0,$   $\phi_k'([f_{\ep/4}(e_C)])\le n[e_A]$
for all large $k.$  It follows that
\beq
d_\tau(\phi_k'([f_{\ep/4}(e_C)))<nd_\tau(e_A)\rforal \tau\in T(A).
\eneq
Since $A$ has the strict comparison, 
% Since $d_\tau(e_A)$ is continuous on $T(A),$ by Corollary A.4 of \cite{eglnkk0}, 
 \beq
 \phi_k'(f_{\ep/4}(e_C))\lesssim \diag(\overbrace{e_A, e_A,...,e_A}^n).
 \eneq
 By Lemma 3.2 of \cite{eglnp}, 
 there exists a unitary $U\in M_n(\td A)$ 
 such that
 \beq
 U^*\phi_k'(f_{\ep}(e_C)U\in M_n(A).
 \eneq
 It follows, since $C$ is semiprojective, that there is a sequence of \hm s 
 $\phi_k: C\to M_n(A)$ such that
 \beq
 \lim_{k\to\infty}{\rm Cu}^\sim(\psi_k)=\lambda.
 \eneq
 \fi
 
 This completes the induction. Theorem follows.
 \end{proof}
 
 \begin{cor}\label{CEXTmaps}
 Let $C\in {\cal I}_0$ be a \CA\, with a strictly positive element $e_C$ 
 and let $A$ be a finite separable simple  stably projectionless \CA\, which 
 is ${\cal Z}$-stable with continuous scale such that $QT(A)=T(A).$
  Suppose that  there is a morphism $\lambda: {\rm Cu}^\sim(C)\to {\rm Cu}^\sim(A)$ in ${\bf Cu}$
 such that $\lambda([e_C])\le [a]$   for some $a\in A_+$ and $\lambda([c])\not=0$
 for all $c\in C_+\setminus \{0\}.$  
 %with $\widehat{[a]}$ is continuous on $T(B).$
 Then 
 %there is an integer $m\ge n$ and a sequence of \hm s $\phi_n: C\to M_m(A)$ such that
 %$\lim_{k\to\infty}{\rm Cu}^\sim(\phi_k)=\lambda.$  Also, if there is a morphism 
 %$\lambda: {\rm Cu}^\sim(C)\to {\rm Cu}^\sim (\td A)$ such that $\lambda([e_C])\le n[1_{\td A}]$
 %then 
 there exists  a sequence of  \hm s $\phi_k: C\to
 \overline{aAa}$ 
 %for some 
 %$a_0\in A_+$ with $[a_0]\circeq [a]$  
 such that 
 $\lim_{k\to\infty}{\rm Cu}^\sim (\phi_k)=\lambda.$ 
 
 Moreover,  there exists a sequence of injective \hm s $\phi_k: C\to 
 \overline{aAa}$ 
 such that
 $\lim_{k\to\infty}^w {\rm Cu}^\sim (\phi_k)=\lambda.$
\end{cor}
 
 \begin{proof}
 {{Recall that $A$ satisfies the condition that $B$ satisfies in \ref{151}.}}
 Let $e_A\in A$ be a strictly positive element. Then $\widehat{[e_A]}$ is continuous on $T(A).$ 
 It follows from Theorem \ref{TTEXTC0} that there exists an integer $m\ge 2$ and a  sequence of \hm s 
 $\psi_k: C\to M_m(A)$ such that
 \beq
 \lim_{k\to\infty}{\rm Cu}^\sim(\psi_k)=\lambda.
 \eneq
 Since $A$ is stably projectionless and $\lambda$ is strictly positive, $\lambda([e_C])$ is not compact.
 %Since $[e_A]\in S(\td A)$ as it is continuous,  
 %By Theorem A.6 of \cite{eglnkk0} (as mentioned at the end of ), there is $a_1\in (\td A\otimes {\cal K})_+$
 %such that $[a_1]\in S(\td A)$ and $[a_1]\circeq [a]$ in ${\rm Cu}(\td )^{\circeq}.$ 
 %Since $a\in A,$ and $\widehat{[1_{\td A}]}$ is continuous on $T(A),$ 
 %by Theorem A.6 of \cite{eglnkk0}, $a_1\lesssim .$
% Since $A$ is stably projectionless, $\lambda([e_C])$ is not compact.
 %where $e_C$ is a strictly positive element of $C.$
 Applying (2) of Lemma \ref{Llimintmaps}, we 
 %may assume 
 obtain a sequence of \hm s $\phi_k': C\to  A$
 %\overline{aAa}$
  such that 
 \beq
  \lim_{k\to\infty}{\rm Cu}^\sim(\phi_k')=\lambda.
 \eneq
 %that $\psi_k$ maps $C$ into $A.$
 %Thus $\phi_k(e_C)\in M_2(B).$ Since $B$ has strict comparison, for any $\ep>0,$ 
 %\beq
% f_\ep(\phi_k(e_C))\lesssim a.
 %\eneq
 
To see the last part of the statement and to make  \hm s injective, for each $k\ge 1,$ choose $0<\ep_k<1/2^{k+1}$
and define $L_k: C\to A_k:={\rm Her}(f_{\ep_k}(a))$ by
\beq
L_k(c)=f_{\ep_k}(a)\phi_k(c)f_{\ep_k}(a)\rforal c\in C.
\eneq
 Since $C$ is semiprojective, by choosing small $\ep_k,$ one obtains 
 a \hm\, $\phi_k'': C\to A_k$ such that (see also Theorem \ref{LappCum})
 \beq
 &&\lim_{k\to\infty}\|\phi_k''(c)-\phi_k(c)\|=0\rforal c\in C\andeqn\\\label{CEXTmaps-4}
&& \lim_{k\to\infty}{\rm Cu}^\sim (\phi_k'')=\lambda.
 \eneq
 Choose a nonzero function in $g_k\in C_0((0,1])_+$ with support in $(0, \ep_k/3)$
 and nowhere zero in $(0,\ep_k/3).$ 
 Put $B_k={\rm Her}(g_k(a)).$   Since $A$ is stably projectionless, we may assume
 that $B_k$ is nonzero. Note also $B_k\perp A_k.$  
 Put 
 \beq\label{CEXTmaps-5--}
 \sigma_k:=\sup\{d_\tau(g_k): \tau\in T(A)\}>0.
 \eneq
 Since $A$ has continuous scale, we have
 that
 \beq\label{CEXTmaps-5}
 \lim_{k\to\infty}\sigma_k=0.
 \eneq
 Note that $B_k$ is a hereditary \SCA\, of $A$ and therefore it is also ${\cal Z}$-stable (Cor.  3.1 of \cite{TW}).
 Choose a nonzero  hereditary \SCA\, $D_k\subset B_k$ which has continuous scale (see Remark 5.3 of \cite{eglnp}).
 By Theorem 6.11 of \cite{RS}, ${\rm Cu}^\sim(D_k)=K_0(D_k)\sqcup \LAff_+^\sim(T(D_k)).$ 
 By  Corollary A.8 of \cite{eglnkk0} and Theorem 4.1 of \cite{GLIII}, there exists $\tau_0\in T(D_k)$ such that 
 $\rho_{D_k}(x)(\tau_0)=0$ for all $x\in K_0(D_k),$ where $\rho_{D_k}: K_0(D_k)\to \Aff (T(D_k))$
 is the usual paring.   Recall ${\cal W}$ is the unique  separable $KK$-contractible 
 amenable  simple ${\cal Z}$-stable \CA\, with a unique tracial state $\tau_W.$ 
  Define 
 $\gamma: {\rm Cu}^\sim(D_k)\to {\rm Cu}^\sim ({\cal W})$ by 
 $\gamma|_{K_0(D_k)}=0$ and $\gamma(f)(\tau_W)=rf(\tau_0)$ for all $f\in 
 %\LAff_+^\sim(\td T(C)).
 \LAff_+^\sim(\td T(D_k))$   for a choice of $0<r<1.$
  {{Recall ${\rm Cu}^\sim(D_k)={\rm Cu}^\sim (A).$}}
 Note that $\gamma \circ \lambda([c])\not=0$ for all $c\in C_+
 \setminus \{0\}.$  We choose $r$ so that  $\gamma\circ \lambda([e_C])(\tau_W)<1.$
By Theorem 1.0.1 of \cite{Rl}, there is an injective \hm\, (since $\gamma\circ\lambda$ is strictly positive)
$h_k: C\to {\cal W}$ such that ${\rm Cu}^\sim(h_k)=\gamma\circ \lambda.$
Let $E$ be a separable $KK$-contractible amenable simple ${\cal Z}$-stable \CA\, with $T(E)=T(D_k)$ 
and has stable rank one (see Theorem 2.8 of \cite{eglnkk0}).
Let $h_{E,D}: E\to {{D_k}}$ be a  nonzero \hm\, given  by Theorem A.26 of \cite{eglnkk0}
so that $h_{E,D}$ induces the identification of $T(E)=T(D).$
Let $\eta: {\rm Cu}^\sim({\cal W})\to {\rm Cu}^\sim (E)$ be defined by
$\eta(f)(\tau)=f(\tau_W)$ for all $f\in \LAff_+^\sim(\td T({\cal W})).$  Applying Theorem 1.0.1 of \cite{Rl} again, 
there is a monomorphism $h_{W, E}: {\cal W}\to E$ such that ${\rm Cu}^\sim(h_{W, E})=\eta.$

Define $h_{k, C, D}:=h_{E, D}\circ h_{W, E}\circ h_k: C\to D_k.$ Then $h_{k,C, D}$ is an injective \hm. 
Define $\phi_k: C\to {\rm Her}(a)$ by
$\phi_k(c)=\phi_k'(c)+h_{k,C, D}(c)$ for all $c\in C.$ Recall that $D_k\perp B_k.$ 
The map $\phi_k$ is injective.  It remains to show that $\lim_{k\to\infty}^w{\rm Cu}^\sim(\phi_k)=\lambda.$

%Note also, for any $c\in M_N(\td C)_+$ (for integer $N\ge 1$), 
%\beq\label{Inj-n10}
%\phi_k^\sim(c)\ge \phi_k'^\sim(c).
%\eneq
%To see this, let $\{e_n\}$ be an approximate identity for $C.$ 
%Then, 
%\beq
%\phi_k(c^{1/2}e_nc^{1/2})\ge \phi_k'(c^{1/2}e_nc^{1/2}). 
%\eneq

Since $h_{k,C, D}$ factors through ${\cal W},$ 
${\rm Cu}^\sim(h_{k,C, D})|_{K_0(C)}=0.$   Note here we view $K_0(A)$ as a subset 
of ${\rm Cu}^\sim(C)$ (see subsection 6.1 and Theorem 6.1 of \cite{RS}).
 Then, 
%by \eqref{CEXTmaps-5}  and 
by \eqref{CEXTmaps-4}, for any finite subset $G\subset {\rm Cu}^\sim(C),$ 
there exists $N\ge 1$ such that, for any $k\ge N$ (see also \ref{DappCu}),
\beq\label{Inj-200}
{\rm Cu}^\sim(\phi_k)(x)=\lambda(x)\rforal x\in G\cap K_0(C).
%{\lim^w_{k\to\infty}}{\rm Cu}^\sim(\phi_k)=\lambda.
\eneq
Let  $f,g \in G,$ $f\ll g$ be such that  neither $f$ nor $g$ are compact.
Recall, by Theorem 5.3 of \cite{RS}, that ${\rm Cu}^\sim (A)$ is orderly embedded into ${\rm Cu}^\sim(\td A).$
Let $\lambda^\sim: {\rm Cu}^\sim(\td C)\to {\rm Cu}^\sim(\td A)$ be the unique extension of 
$\lambda$ given by Lemma \ref{Lunitalext}. As in the proof of case (3) in the proof of Theorem \ref{TTEXTC0},
$\lambda^\sim$ is strictly positive.

Let ${\bar f}$ and ${\bar g}$ be as in the proof of \ref{LappCum} with $\|{\bar f}\|, \|{\bar g}\|\le 1.$
We also retain other  notations in the proof \ref{LappCum} related to $f$ and $g.$
%By \eqref{CEXTmaps-4},  there is $N_1\ge 1$ such that, when $k\ge N_1,$ 
%\beq
%\lambda(f)\le {\rm Cu}^\sim(\phi_k')(g)\le {\rm Cu}^\sim (\phi_k)(g).
%{\rm Cu}^\sim(\phi_k')(f)\le \lambda(h)\ll \lambda^\sim({\bar g}.
%\eneq

Since $f$ and $g$ are not compact elements,   by (ii) of Theorem 6.1 of \cite{RS}, 
neither are ${\bar f}$ and ${\bar g}.$ 
%are not compact elements either. 
Since $\lambda^\sim$ is strictly positive,
$\lambda^\sim ({\bar f})$ and $\lambda^\sim ({\bar g})$ are not compact.
Let $d_f, d_g\in M_r(\td A)_+$ (for some $r\ge 1$) such that 
$[d_f]=\lambda^\sim({\bar f}),$  and $[d_g]=\lambda^\sim({\bar g}).$
Note, as $\lambda^\sim$ is the unique extension of $\lambda,$ 
\beq\label{Inj-n-100}
\lambda^\sim({\bar f})=\lambda(f)+m_f[1_{\td A}]+m_g[1_{\td A}]\andeqn
\lambda^\sim({\bar g})=\lambda(g)+m_g[1_{\td A}]+m_f[1_{\td A}]
\eneq
(see \eqref{Rl33}).
%
%Note
Then (recall that $[d_g]$ cannot be represented by a projection), there is $0<\dt<1/2$ such that 
\beq\label{Inj-1}
&&
\pi_\C^C(f_{\dt}(d_g))=\pi_\C^C(d_g),\,\,\, d_\tau(f_{\dt/2}(d_g))>\tau(f_{\dt}(d_g))\rforal \tau\in T(A)\andeqn\\
&&{[}{\bar f}{]}\ll  [f_{2\dt}(d_g)]\le [d_g].
\eneq
Thus, by \eqref{CEXTmaps-4}, there exists an integer $N_1\ge 1$ such that, for $k\ge N_1,$
\beq
[\phi_k'^\sim({\bar f})]\le [f_{2\dt}(d_g)].
\eneq
Therefore (see also \eqref{Inj-1})
\beq\label{Inj-99}
d_\tau(\phi_k'^\sim({\bar f}))<\tau(f_\dt(d_g))<\tau(f_{\dt/2}(d_g))\le d_\tau(f_{\dt/2}(d_g))\le d_\tau({\bar g})\rforal \tau\in T(A)\andeqn\\
d_\tau(\phi_k'^\sim({\bar f}))\le \tau(f_\dt(d_g))\le \tau(f_{\dt/2}(d_g))\le d_\tau(f_{\dt/2}(d_g))\le d_\tau({\bar g})\rforal \tau\in T(\td A).
\eneq
Note that  the lower semicontinuous function  $\widehat{[f_{\dt/2}(d_g)]}-\widehat{f_\dt(d_g)}$ 
is strictly positive
on the compact set $T(A).$ 
It follows 
that 
\beq\label{Inj-20}
\eta:=\inf\{d_\tau(f_{\dt/2}(d_g))-\tau(f_\dt(d_g)):\tau\in T(A)\}>0.
\eneq
Note that  we may assume that ${\bar f}\in M_{r+m_g}(\td C)$ (see the  lines below \eqref{512} and lines below 
\eqref{59} in 
the proof 
of \ref{LappCum}).
We may also assume, for all $k\ge N_1,$
\beq
(r+m_g)\sigma_k<\eta/4.
%\andeqn
%{\rm Cu}^\sim(\phi_k')([{\bar f}])\le \lambda([f_\dt(d_g)])\ll \lambda({\bar g}),
\eneq
For any $1/2>\ep_0>0,$ 
write
\beq
f_{\ep_0}({\bar f})=S+c_{f, \ep_0},
\eneq
where $S\in M_{r+m_g}(\C)_+$ and $c_{f, \ep_0}\in M_{r+m_g}(C)_{s.a.}$
and $\|S\|\le 1$ and $\|c_{f, \ep_0}\|\le 2.$ 
Recall (identifying $S$ with the scalar matrix),
\beq\label{Inj-10}
\phi_k'^\sim(f_{\ep_0}({\bar f}))=S+\phi_k'(c_{f,\ep_0}) \andeqn \phi_k^\sim (f_{\ep_0}({\bar f}))=S+\phi_k'^\sim (c_{f, \ep_0})+h_{k, C, D}(c_{f, \ep_0}).
\eneq
We estimate that, by \eqref{CEXTmaps-5--},  for all $\tau\in T(A).$ 
\beq
|\tau(h_{k, C,D}((c_{f, \ep_0})))|\le 2 (r+m_g)\sigma_k<\eta/2.
\eneq
%Note also $\phi_k'(f_{\ep_0}({\bar f}))=S+\phi_k'(c_{f,\ep_0}).$
Combining this with \eqref{Inj-10}, \eqref{Inj-99},  and \eqref{Inj-20},
we obtain, for any $1/2>\ep_0>0,$ if  $k\ge N_1,$
\beq
&&d_\tau(\phi_k^\sim(f_{\ep_0}({\bar f})))<d_{\tau}(d_g)\rforal \tau\in T(A)\andeqn\\
&&d_{\tau}(\phi_k^\sim (f_{\ep_0}({\bar f})))\le d_\tau(d_g) \rforal \tau\in  T(\td A).
\eneq
%By 
%Theorem A.6 of \cite{eglnkk0}, there is $d_g'\in S(\td A)$ such that 
%$d_g'=[d_g].$ 
It follows from  
Theorem \ref{TBtildC},
%A.6 of \cite{eglnkk0}, 
%ref{TBtildC},  
%for any $1/2>\ep_0>0,$ 
if  $k\ge N_1$  (in ${\rm Cu}(\td A)$),
$
[\phi_k^\sim(f_{\ep_0}({\bar f}))]\le 
% [d_g'] \circeq 
[d_g].
$
%\eneq
Since $N_1$ does not depend on $\ep_0,$ this implies 
that (in ${\rm Cu}(\td A)$)
$
{\rm Cu}(\phi_k^\sim)({\bar f})\le [d_g].
$
%\eneq
%
In other words (see also \eqref{Inj-n-100} and the lines below \eqref{59}), if $k\ge N_1,$ 
\beq
{\rm Cu}^\sim(\phi_k)(f)+(m_f+m_g+2)[1_{\td A}]=[\phi_k(a^f)]+m_g[1_{\td A}]\le \lambda(g)+(m_g+m_f+2)[1_{\td A}]
\eneq
(recall that $A$ has stable rank at most 2).
Thus, if $k\ge N_1,$ 
\beq
{\rm Cu}^\sim(\phi_k)(f)\le \lambda(g).
\eneq
The same  argument shows that, if $k\ge N_1,$ 
\beq
\lambda(f)\le {\rm Cu}^\sim(\phi_k(g)).
\eneq
Hence, combining with the last two displays and \eqref{Inj-200}, one obtains 
\beq\nonumber
{\lim}^w_{n\to\infty}{\rm Cu}^\sim(\phi_k)=\lambda.
\eneq
\vspace{-0.1in}
 \end{proof}

\begin{df}\label{DlambdaT}
Let $C$ be a separable \CA\, such that $T(C)\not=\emptyset$ and $QT(C)=T(C).$ 
Let $B$  be  a separable simple \CA\,  with continuous scale such that   $QT(B)=T(B)$ and 
 ${\rm Cu}(B)=\LAff_+(T(B)).$  
 %and ${\rm Cu}^\sim (B)=K_0(B)\sqcup \LAff_+^\sim(T(B)).$
 Let $\lambda:{\rm Cu}^\sim(C)\to {\rm Cu}^\sim(B)$ be a 
 morphism in ${\bf Cu}$ such that 
 $\lambda([e_C])\le [e_B],$ where $0\le e_C\le 1$ and $0
 \le e_B\le 1$ are strictly positive elements of $C$ and $B,$ respectively.
 Let $T_0(C)$ and $T_0(B)$ be the sets of all 
 traces on $C$ and $B$ with norm no more than $1,$ respectively. 
 
Let $a\in C_+$ with $\|a\|\le 1.$ For each $n$ consider 
$x_n=\sum_{k=1}^{2^n} (1/2^n) p_{(t_{k,n}, 1]},$ where $t_{k,n}=k/2^n$ and 
$p_{(t_{k,n},1]}$ is the open spectral projection of $a$ associated with the $(t_{k,n},1]$ in 
$C^{**}.$ Note that $d_\tau((a-t_{k,n})_+)=\tau(p_{(t_{k,n},1]})$ for all $\tau\in T_0(C)$ and 
$\tau(x_n)=\sum_{k=1}^{2^n} (1/2^n)d_\tau((a-t_{k,n})_+)$
for all $\tau\in T_0(C).$ 
Moreover, 
\beq\label{Integral-1}
\sup\{|\tau(x_n)-\tau(a)|: \tau\in T_0(C)\}\le 1/2^n.
\eneq
For each $s\in T_0(B),$ define, for each $a\in M_r(C)_+$ (for integer $r\ge 1$),
\beq
\lambda_T(s)(a)=\int_0^{\infty} \lambda([(a-t)_+])^{\widehat{}}(s)dt.
\eneq
By Proposition 4.2 of \cite{ESR-Cuntz}, $\lambda_T(s)$ defines a  lower semi-continuous quasitrace on $C\otimes {\cal K}.$ 
Note that $B$ has continuous scale. So $e_B\in {\rm Ped}(B).$  Since $\lambda([e_C])\le [e_B],$ 
if $a\in M_r(C)_+,$ 
$\lambda([(a-t)_+])^{\widehat{}}(s)\le r\|a\|$ for all  $t\in [0, \|a\|]$ and $s\in T(B).$
Since $QT(C)=T(C),$ $\lambda_T(s)$ is in $T_0(C).$ 
Proposition 4.2 of \cite{ESR-Cuntz}   also implies that the map $s\mapsto \lambda_T(s)$ is the 
affine continuous map from $T_0(B)$ to $T_0(C)$  induced 
by $\lambda.$
Note 
\beq\label{Approtrace-1}
\lambda_T(s)(a)=\lim_{n\to\infty}(\sum_{k=1}^{2^n} (1/2^n)\lambda([(a-t_{k,n})_+])^{\widehat{}}(s)).
\eneq
%On $(0,1],$  for each $\tau\in T_0(B),$  define 
%$\mu_\tau((t,1])=\tau(\lambda([(a-t)_+]))$ (for $t\in (0,1]$).  If $0\le t_1<t_2\le 1],$ 
%define $\mu_\tau((t_1, t_2])=\mu_\tau((t_1,1])-\mu_\tau((t_2,1]).$ 
%Recall that $\lambda$ is a morphism in ${\bf Cu}.$  It preserves the suprema of increasing sequence. 
%Recall also ${\rm Cu}(A)=\LAff_+(T(A)).$
%For any $0<t<t_{n+1}<t_n\le 1$ such that $t_n\to t,$  one has
%$\lambda([(a-t_n)_+])^{\widehat{}}\nearrow \lambda([(a-t)_+])^{\widehat{}}.$
%It is 
%a standard measure theory exercise that $\mu_\tau$ gives a Borel measure on 
%$(0,1].$ 
%Then it  is a standard step-function approximation 
%argument that limit in \eqref{Approtrace-1} exists.  
Moreover, for $a\in C_+$ with $\|a\|\le 1,$
\beq\label{Apptrace-0}
\lim_{n\to\infty}\sup\{|\lambda_T(s)(a)-(\sum_{k=1}^{2^n} (1/2^n)\lambda([(a-t_{k,n})_+])^{\widehat{}}(s))|:s\in T(B)\}=0.
\eneq
%
%For $x\in C_{s.a.}$ and $\|x\|\le 1,$  define 
%$\lambda_T(s)(x)=\lambda_T(s)(x_+)-\lambda_T(s)(x_-).$ One then extends 
%$\lambda_T$ to $C_{s.a.}.$ One checks that  $\lambda_T(s)$ gives a trace in $T_0(C).$
%In fact $\lambda_T: T_0(B)\to T_0(C)$ is an affine continuous map. 
%
Let $\phi: C\to A$ be a \hm. Then, for any $a\in C_+$ with $\|a\|\le 1,$ 
\beq\label{Approtrace-2}
\lim_{n\to\infty}\sup\{|\tau(\phi(c))-\sum_{k=1}^{2^n}(1/2^n)\tau(\phi(f_{1/2^{n+3}}(a-t_{k, n})_+))|:\tau\in T(B)\}=0.
\eneq
%Note that, for $a\in C_+$ with $\|a\|\le 1,$
%\beq\label{Apptrace-0}
%\lim_{n\to\infty}\sup\{|\lambda_T(s)(a)-(\sum_{k=1}^{2^n} (1/2^n)(\lambda([(a-t_{k,n})_+])^{\widehat{}}(s))|:s\in T(B)\}=0.
%\eneq

 Now suppose that $B$ is stably projectionless and $\lambda$ is strictly positive. 
If $\phi_k: C\to B$ is a sequence of injective \hm s such that $\lim^w_{n\to\infty}{\rm Cu}^\sim(\phi_k)=\lambda,$ 
then, for each fixed $a\in C_+$ with $\|a\|\le 1$ and $n>1,$   there is $N\ge 1$ such that, when $j\ge N,$
\beq\label{Approtrace-3}
&&\hspace{-0.4in}\lambda([(a-t_{k,n})_+])\le [\phi_j(f_{1/2^{n+2}}((a-t_{k+1, n})_+)]\\
&&\le [\phi_j(f_{1/2^{n+3}}((a-t_{k+1, n})_+)]
 \le \lambda([a-t_{k+1,n})_+),\,\,\,k=1,2,...,2^n.
\eneq
It follows that, for all $\tau\in T(B),$ 
\beq\label{Approtrace-5}
\lambda([(a-t_{k,n})_+])^{\widehat{}}(\tau)\le \tau(\phi_j(f_{1/2^{n+3}}((a-t_{k+1, n})_+)
 \le \lambda([a-t_{k+1,n})_+)^{\widehat{}}(\tau).
\eneq
By \eqref{Approtrace-2}, \eqref{Apptrace-0},  and \eqref{Approtrace-5},
\beq\label{Apptrcace-100}
\lim_{k\to\infty}\sup\{|\tau(\phi_k)(a)-\lambda_T(\tau)(a)|: \tau\in T_0(B)\}=0.
\eneq

\iffalse
Let $A$ and $B$  be separable simple \CA s  with continuous scale such that ${\rm Cu}(A)=\LAff_+(T(A))$ and
${\rm Cu}^\sim(A)=K_0(A)\sqcup \LAff_+^\sim (T(A)),$ ${\rm Cu}(B)=\LAff_+(T(B)),$  and ${\rm Cu}^\sim (B)=K_0(B)\sqcup \LAff_+^\sim(T(B)).$
Suppose that $A$ has stable rank one and $B$ has almost  stable rank one. 
Suppose that $\lambda: {\rm Cu}^\sim(A)\to {\rm Cu}^\sim (B)$ such that $\lambda([e_A])\le [e_B],$
where $e_A$ and $e_B$ are strictly positive elements of $A$ and $B,$ respectively.
It follows that $\lambda$ maps ${\rm Cu}(A)$ into ${\rm Cu}^\sim(B)_+.$ Suppose 
also that $B$ is stably projectionless. Then $\lambda$ maps ${\rm Cu}(A)$ into $\LAff_+(T(A)).$
In other words, $\lambda$ maps $\LAff_+(T(A))$ into $\LAff_+(T(B)).$ 
\fi
%It follows that $\lambda$ induces a continuous affine map
%$\lambda_T: T(B)\to T(A)$ such that 
%$\lambda_T.$

%${\rm Cu}(B)^{\circeq}.$ 

\end{df}

%It should be reminded that
Recall that  if $A$ is a finite  exact separable simple 
%stably projectionless 
${\cal Z}$-stable \CA\, then $A$ satisfies the conditions in the following statement.

\begin{thm}\label{Cinductive}
Let $C=\lim_{n\to\infty} C_n$ with a strictly positive element $e_C,$  where each $C_n\in {\cal I}_0$  and 
each map $\iota_n: C_n\to C_{n+1}$ is injective,  and let 
$A$ be a  separable simple 
%stably projectionless  
\CA\, 
 with continuous scale  and with a strictly positive element $e_A$ such 
 that $M_m(A)$ has almost stable rank one for all $m\ge 1$ and 
$QT(A)=T(A),$ and ${\rm Cu}(A)=\LAff_+(T(A)).$ 
% in Corollary  \ref{CEXTmaps}
%\ref{151}    
Suppose that $\lambda: {\rm Cu}^\sim(C)\to {\rm Cu}^\sim(A)$
is a morphism in ${\bf Cu}$
such that $\lambda([e_C])\le [e_A]$ and $\lambda([c])\not=0$ for 
any $c\in C_+\setminus \{0\}.$ Then there exists a sequence of \morp s 
$L_n: C\to A$ 
and a sequence of injective  \hm s $h_n: C_n\to A$ such that
\beq\nonumber
&&\lim_{n\to\infty}\|L_n(ab)-L_n(a)L_n(b)\|=0\tforal a, b\in C,\\\nonumber
&& \,\text{and, for each fixed $m,$}
\lim_{n\to\infty}\|L_n(\iota_{m, \infty}(c))-h_n(\iota_{m,n}(c))\|=0\rforal c\in C_m,\\\nonumber
&&\andeqn
\lim_{n\to\infty} \sup_{\tau\in T(A)}\|\tau(L_n(a))-\lambda_T(\tau)(a)\|=0\tforal a\in C.
\eneq

\end{thm}

\begin{proof}
%If $A$ is not stably projectionless, by  Proposition \ref{P1}, $A$ has stable rank one. 
%Then, by 
%The conclusion follows from 
%Theorem 1.0.1 of \cite{Rl}, there is a \hm\, $H: C\to A$ such 
%that ${\rm Cu}^\sim(H)=\lambda.$
%
% as shown below. 
We first assume that $A$ is stably projectionless. 
For each $k,$ consider $\af_k:=\lambda\circ {\rm Cu}^\sim(\iota_{k, \infty}).$ 
By Corollary \ref{CEXTmaps}, there exists a sequence 
of injective \hm s $\phi_{k,n}: C_k\to A$ such that $\lim_{n\to\infty}^w{\rm Cu}^\sim(\phi_{k,n})=\af_k.$
%\lambda\circ {\rm Cu}^\sim(\iota_{k, \infty}).$
%For each $c\in C_k,$ denote by $\hat{a}$ the affine 
%function defined by $\hat{a}(\tau)=\tau(a)$ for $\tau\in T_0(C_k).$
Then (see \eqref{Apptrcace-100})
\beq
\lim_{n\to\infty}\sup\{|\tau\circ \phi_{k,n}(c)-{\af_k}_T(\tau)(c)|: \tau\in T_0(A)\}=0\rforal c\in C_k.
\eneq
One obtains a sequence of injective  \hm s 
$h_n: C_n\to A$ and, since $C$ is amenable,   a sequence of 
\morp s $L_n: C\to A$
such that, for any fixed $m,$
\beq\nonumber
&&\lim_{n\to\infty}\|L_n(\iota_{m, \infty}(c))-h_n(\iota_{m,n}(c))\|=0\rforal c\in C_m, \\\nonumber
&&\lim_{n\to\infty}\|L_n(ab)-L_n(a)L_n(b)\|=0\rforal a, b\in C\andeqn\\\nonumber
&&\lim_{n\to\infty} \sup_{\tau\in T(B)}\|\tau(L_n(a))-\lambda_T(\tau)(a))\|=0\rforal a\in C.
\eneq
If $A$ is not stably projectionless, by  Proposition \ref{P1}, $A$ has stable rank one. 
Then, by 
%The conclusion follows from 
Theorem 1.0.1 of \cite{Rl}, there is a \hm\, $H: C\to A$ such 
that ${\rm Cu}^\sim(H)=\lambda.$  Choose $L_n=H$ and $h_n=H\circ \iota_{n, \infty}.$ 
Then this case also follows. 
\end{proof}

\begin{cor}\label{MCC}
Let $C=\lim_{n\to\infty} (C_n, \iota_n)$ be as in Theorem \ref{Cinductive}  which is simple and  has continuous scale 
and $A$ be a finite  exact separable simple stably projectionless ${\cal Z}$-stable \CA\, with continuous scale.
%and $B$ be in Corollary \ref{CEXTmaps}. 
Suppose 
that there is an isomorphism
\beq\label{MCC-1}
\Gamma: (K_0(C), T(C), r_C)\cong (K_0(A), T(A), r_A).
\eneq
Then there exists a sequence of \morp s $L_n: C\to A$
and a sequence of injective \hm s $h_n: C_n\to A$ 
 such that
\beq
&&\lim_{n\to\infty}\|L_n(ab)-L_n(a)L_n(b)\|=0\rforal a, b\in C,\\
&&\hspace{-0.1in}\lim_{n\to\infty} \sup_{\tau\in T(B)}\|\tau(L_n(a))-\lambda_T(\tau)(a)\|=0\rforal a\in C\\
&&{\text{and,\,\,for  each fixed $m,$}}
\lim_{n\to\infty}\|L_n(\iota_{m, \infty}(c))-h_n(\iota_{m,n}(c))\|=0\rforal  c\in C_m,\\
&&\lim_{n\to\infty} \sup_{\tau\in T(B)}\|\tau(h_n(\iota_{m,n}(c)))-\lambda_T(\tau)(\iota_{m,\infty}(c))\|=0
\rforal c\in C_m,
\eneq
where $\lambda_T: T(A)\to T(C)$ is the affine homeomorphism given 
by $\Gamma.$
\end{cor}

\begin{proof}
By  Proposition 6.2.3 of \cite{Rl},  ${\rm Cu}^\sim(C)=K_0(C)\sqcup \LAff_+^\sim(T(C)).$ 
Also, by  Theorem 6.1.1 of \cite{RS} (see also subsection 6.3 of \cite{RS}), ${\rm Cu}^\sim(A)=K_0(A)\sqcup \LAff_+^\sim(T(A)).$
Let $e_C$ and $e_A$ be strictly positive elements of $C$ and $A,$ respectively. 
By \eqref{MCC-1}, there is an isomorphism $\lambda: {\rm Cu}^\sim(C)\to {\rm Cu}^\sim(A)$
(in ${\bf Cu}$) such that $\lambda([e_C])=[e_A].$
Thus, the corollary follows from Theorem \ref{Cinductive} immediately.
\end{proof}

Corollary \ref{MCC} plays an important role in   achieving 
%improving 
the following 
theorem which was first proved with the additional condition that $A$ has stable rank one.
The only place where we need the condition that $A$ has stable rank one 
was to have a \hm\, $h: C\to A,$ where 
$C=\lim_{n\to\infty}(C_n, \iota_n),$ $C_n\in {\cal I}_0$ and  $\iota_n: C_n\to C_{n+1}$
are injective, $C$ has continuous scale, and $(K_0(C),  T(C), r_C)=(K_0(A), T(A), r_A)$
such that $[h]$ induces the identification  map on $(K_0(C), T(C), r_C).$ 
Note the identification map on the invariant set gives 
a strictly positive morphism $\lambda: {\rm Cu}^\sim(C)\to {\rm Cu}^\sim(A)$ with $\lambda([e_C])=[e_A],$
where $e_C$ and $e_A$ are strictly positive elements of $C$ and $A,$ respectively.
So the existence of such $h$ follows from Theorem 1.0.1 of \cite{Rl}.  In fact, 
one only needs an approximate version of  Robert's result.  
%the existence of a sequence of \hm s $h_k: C\to A$ such that $\lim_{k\to\infty}^w{\rm Cu}^\sim(h_k)=\lambda$
%is sufficient 
Without assuming $A$ has stable rank one, one may not apply the result of L. Robert. 
However, one can apply Corollary \ref{MCC} to obtain a sequence of \hm s $h_k$  that approximates $\lambda$ which 
improves the original version of the following theorem. 

\begin{thm}[Theorem 7.12 of \cite{GLIII}]\label{Treduction}
Let $A$ be separable amenable  simple 
{{stably projectionless}} \CA\, with continuous scale  such that
$T(A)\not=\{0\}$ and   satisfying
%Suppose that every tracial states of $A$ is quasidiagonal and $A$
%satisfies
the UCT.
Then $A\otimes Q$  has generalized tracial rank one.
%has generalized tracial rank at most 1.
\end{thm}

%\begin{proof}
%By \cite{TWW}, every tracial state of $A$ is quasidiagonal. Then, by  Theorem \ref{TRangeM}
%and \ref{4.25},
%the range theorem {\bf{??}},
%and
%Theorem \ref{Treduction1}, $A\otimes Q\in {\cal D}.$
%\end{proof}

 \providecommand{\href}[2]{#2}

   \hspace{0.2in}
   
   \noindent
   hlin@uoregon.edu

      \end{document}